\documentclass{amsart}

\usepackage{latexsym,amsthm,mathrsfs,amsmath,amscd,enumerate,enumitem, amscd,color,dsfont, cancel, mathtools, textcomp, float}
\usepackage[utf8x]{inputenc}
\usepackage[cspex,bbgreekl]{mathbbol}
\usepackage{amsfonts}
\usepackage{amssymb}             % AMS Math

\DeclareSymbolFontAlphabet{\mathbbl}{bbold}
\DeclareSymbolFontAlphabet{\mathbb}{AMSb}
 \newtheorem{thm}{Theorem}[section]

 \newtheorem{prop}[thm]{Proposition}
\theoremstyle{definition}

 \theoremstyle{remark}
 \newtheorem{rem}[thm]{Remark}

\usepackage{hyperref}
\usepackage{tikz}
\usetikzlibrary{plotmarks}
\newcommand{\supp}{\mathop{\mathrm{supp}}}

\makeatletter
\DeclareRobustCommand\widecheck[1]{{\mathpalette\@widecheck{#1}}}
\def\@widecheck#1#2{%
    \setbox\z@\hbox{\m@th$#1#2$}%
    \setbox\tw@\hbox{\m@th$#1%
       \widehat{%
          \vrule\@width\z@\@height\ht\z@
          \vrule\@height\z@\@width\wd\z@}$}%
    \dp\tw@-\ht\z@
    \@tempdima\ht\z@ \advance\@tempdima2\ht\tw@ \divide\@tempdima\thr@@
    \setbox\tw@\hbox{%
       \raise\@tempdima\hbox{\scalebox{1}[-1]{\lower\@tempdima\box
\tw@}}}%
    {\ooalign{\box\tw@ \cr \box\z@}}}
\makeatother

\usepackage{caption}
\usepackage{subcaption}
\usepackage[left=2.2cm,top=2.5cm,right=2.2cm,bottom=2.5cm]{geometry} % Document margins

\numberwithin{equation}{section}
\allowdisplaybreaks
\begin{document}

\title{Higher order Riesz transforms in the inverse Gaussian setting and UMD Banach spaces}

\author[J. J. Betancor]{Jorge J. Betancor}
\address{Jorge J. Betancor, Lourdes Rodr\'{\i}guez-Mesa\newline
Departamento de An\'alisis Matem\'atico, Universidad de La Laguna,\newline
Campus de Anchieta, Avda. Astrof\'{\i}sico S\'anchez, s/n,\newline
38721 La Laguna (Sta. Cruz de Tenerife), Spain}
\email{jbetanco@ull.es, lrguez@ull.edu.es}

\author[L. Rodr\'{\i}guez-Mesa]{Lourdes Rodr\'{\i}guez-Mesa}

\thanks{This paper is partially supported by PID2019-106093GB-I00.}



%\date{\today}

\begin{abstract}
In this paper we study higher order Riesz transforms associated with the inverse Gaussian measure given by $\pi ^{n/2}e^{|x|^2}dx$ on $\mathbb{R}^n$. We establish $L^p(\mathbb{R}^n, e^{|x|^2 }dx)$-boundedness properties and obtain representations as principal values singular integrals for the higher order Riesz transforms. New characterizations of the Banach spaces having the UMD property by means of the Riesz transforms and imaginary powers of the operator involved in the inverse Gaussian setting are given. 
\end{abstract}

\maketitle

\section{Introduction}
Our setting is $\mathbb{R}^n$ endowed with the measure $\gamma_{-1}$ whose density with respect to the Lebesgue measure is $\pi ^{n/2}e^{|x|^2}$, $x\in \mathbb{R}^n$. The measure $\gamma_{-1}$ is called the inverse Gaussian measure. The study of harmonic analysis operators in $(\mathbb{R}^n,\gamma _{-1})$ was began by Salogni (\cite{Sa}). The principal motivation for the Salogni's studies was the connection with the Gaussian setting. However, as Bruno and Sj\"ogren (\cite{BrSj}) pointed out, $(\mathbb{R}^n,\gamma _{-1})$ can be seen as a model of a variety of settings where a theory of singular integrals has not been developed. Also, the natural Laplacian on $(\mathbb{R}^n,\gamma _{-1})$, that we will denote by $\mathcal{A}$, can be interpreted as a restriction of the Laplace-Beltrami operator associated with a warped-product manifold whose Ricci tensor is unbounded from below. A complete exposition of the theory of this kind of manifolds can be found in \cite{Chen}.

The aim of this paper is to study $L^p(\mathbb{R}^n,\gamma _{-1})$-boundedness properties of higher order Riesz transforms in the inverse Gaussian setting. Also, we characterize the UMD Banach spaces by using these Riesz transforms.

We consider the second order differential operator $\mathcal{A}_0$ defined by
$$
\mathcal{A}_0f(x)=-\frac{1}{2}\Delta f(x)-x\cdot \nabla f(x),\quad x\in \mathbb{R}^n,
$$
where $f\in C_c^\infty(\mathbb{R}^n)$, the space of the smooth functions with compact support in $\mathbb{R}^n$. Here, $\Delta$ and $\nabla$ denote the usual Euclidean Laplacian and gradient, respectively.

$\mathcal{A}_0$ is essentially selfadjoint in $L^2(\mathbb{R}^n,\gamma _{-1})$. $\mathcal{A}$ denotes the closure of $\mathcal{A}_0$ in $L^2(\mathbb{R}^n,\gamma _{-1})$. 

For every $k=(k_1,...,k_n)\in \mathbb{N}^n$ by $H_k$ we represent the $k$-th Hermite polynomial given by $H_k(x)=\prod_{i=1}^nH_{k_i}(x_i)$, $x=(x_1,...,x_n)\in \mathbb{R}^n$, where, for every $m\in \mathbb{N}$,
$$
H_m(z)=(-1)^me^{z^2}\frac{d^m}{dz^m}e^{-z^2},\quad z\in \mathbb{R}.
$$
We have that, for every $k=(k_1,...,k_n)\in \mathbb{N}^n$,
$$
\mathcal{A}\widetilde{H}_k=(|k|+n)\widetilde{H}_k,
$$
where $|k|=k_1+k_2+...+k_n$ and $\widetilde{H}_k(x)=e^{-|x|^2}H_k(x)$, $x\in \mathbb{R}^n$. The spectrum of $\mathcal{A}$ in $L^2(\mathbb{R}^n,\gamma _{-1})$ is the discrete set $ \{n+m\}_{m\in \mathbb{N}}$.

The operator $-\mathcal{A}$ generates a diffusion semigroup (in the Stein sense \cite{StLP}) $\{T_t^\mathcal{A}\}_{t>0}$ in $(\mathbb{R}^n,\gamma _{-1})$ where, for every $t>0$, we have that
$$
T_t^\mathcal{A}(f)(x)=\int_{\mathbb{R}^n}T_t^\mathcal{A}(x,y)f(y)dy,\quad x\in \mathbb{R}^n,
$$
for every $f\in L^p(\mathbb{R}^n,\gamma _{-1})$, $1\leq p<\infty$, and being
$$
T_t^\mathcal{A}(x,y)=\frac{e^{-nt}}{\pi ^{n/2}(1-e^{-2t})^{n/2}}\exp\left(-\frac{|x-e^{-t}y|^2}{1-e^{-2t}}\right),\quad x,y\in \mathbb{R}^n,\;t>0.
$$

The maximal operator $T_*^\mathcal{A}$ defined by
$$
T_*^\mathcal{A}f=\sup_{t>0}|T_t^\mathcal{A}f|,
$$
was studied by Salogni (\cite{Sa}). She proved that $T_*^\mathcal{A}$ is bounded from $L^1(\mathbb{R}^n,\gamma _{-1})$ into $L^{1,\infty}(\mathbb{R}^n,\gamma _{-1})$. From the general results in \cite{StLP} it can be deduced that $T_*^\mathcal{A}$ is bounded from $L^p(\mathbb{R}^n,\gamma _{-1})$ into itself, for every $1<p<\infty$. Recently, Betancor, Castro and de Le\'on-Contreras \cite{BCdL} have characterized the K\"oethe function spaces with the Hardy-Littlewood property by using the maximal operators
$$
T_{*,k}^\mathcal{A}f=\sup_{t>0}|t^k\partial _t^kT_tf|,\quad k\in \mathbb{N}.
$$

In \cite{Sa} $L^p(\mathbb{R}^n,\gamma _{-1})$-boundedness properties with $1<p<\infty$ for some spectral multipliers associated with the operator $\mathcal{A}$ were proved. The imaginary power $\mathcal{A}^{i\gamma}$, $\gamma \in \mathbb{R}\setminus\{0\}$, of $\mathcal{A}$ is a special case of the multipliers studied in \cite{Sa}. Bruno (\cite{Br}) established endpoints results for $\mathcal{A}^{i\gamma}$, $\gamma \in \mathbb{R}\setminus\{0\}$, proving that $\mathcal{A}^{i\gamma }$ is bounded from $L^1(\mathbb{R}^n,\gamma _{-1})$ into $L^{1,\infty}(\mathbb{R}^n,\gamma _{-1})$. Also, he showed that, for $\lambda \geq 1$, the shifted first order Riesz transform $\nabla (\mathcal{A}+\lambda I)^{-1/2}$ is bounded from $L^1(\mathbb{R}^n,\gamma _{-1})$ into $L^{1,\infty}(\mathbb{R}^n,\gamma _{-1})$. These operators are studied on new Hardy type $H^1$-spaces.

Higher order Riesz transforms associated with the operator $\mathcal{A}$ were studied by Bruno and Sj\"ogren \cite{BrSj}. For every $\alpha =(\alpha _1,...,\alpha _n)\in \mathbb{N}^n\setminus\{0\}$ the $\alpha$-th Riesz transform is defined by $R_\alpha =\partial ^\alpha \mathcal{A}^{-|\alpha |/2}$, where $\partial ^\alpha =\frac{\partial ^{|\alpha|}}{\partial x_1^{\alpha _1}...\partial x_n^{\alpha _n}}$ and $|\alpha|=\alpha _1+...+\alpha _n$. In \cite[Theorem 1.1]{BrSj} it was established that $R_\alpha$ is bounded from $L^1(\mathbb{R}^n,\gamma _{-1})$ into $L^{1,\infty}(\mathbb{R}^n,\gamma _{-1})$ if and only if $|\alpha|\leq 2$.

In \cite[Remark 2.6]{Br} Bruno proved that, for every $\alpha \in \mathbb{N}^n$ with $|\alpha |=1$, $R_\alpha$ is bounded from $L^p(\mathbb{R}^n,\gamma _{-1})$ into itself, for every $1<p<\infty$. In \cite{BrSj} Bruno and Sj\"ogren say that they do not know whether $R_\alpha$ is bounded from $L^p(\mathbb{R}^n,\gamma _{-1})$ into itself for every $1<p<\infty$ and $\alpha \in \mathbb{N}^n$, $|\alpha |>1$, though they expect so. In our first results we prove that, as they expected, $R_\alpha$ is bounded from $L^p(\mathbb{R}^n,\gamma _{-1})$ into itself when $1<p<\infty$ and $\alpha \in \mathbb{N}^n\setminus\{0\}$. We also obtain a representation of $R_\alpha$ as a principal value singular integral. In order to prove our result we need to use some properties of the negative power $\mathcal{A}^{-\beta}$, $\beta >0$, of $\mathcal{A}$. In Section 2 we analyze $\mathcal{A}^{-\beta}$, $\beta >0$. We obtain that, for every $\beta >0$, the operator $\mathcal{A}^{-\beta}$ is bounded from $L^1(\mathbb{R}^n,\gamma _{-1})$ into $L^{1,\infty}(\mathbb{R}^n,\gamma _{-1})$. This result contrasts with the one in \cite[Proposition 6.2]{GCMST2} where it is proved that $\mathcal{L}^{-\beta}$, $\beta >0$, is not bounded from $L^1(\mathbb{R}^n,\gamma _1)$ into $L^{1,\infty}(\mathbb{R}^n,\gamma _1)$, where $\mathcal{L}$ represents the Ornstein-Uhlenbeck operator and $\gamma _1$ denotes the Gaussian measure ($d\gamma _1(x)=\pi ^{-n/2}e^{-|x|^2}dx$) on $\mathbb{R}^n$.
\begin{thm}\label{Th1.1}
Let $\alpha =(\alpha _1,...,\alpha _n)\in \mathbb{N}^n\setminus \{0\}$. For every $f\in C_c^\infty (\mathbb{R}^n)$, the derivative $\partial ^{\alpha}\mathcal{A}^{-|\alpha|/2}(f)(x)$ exists for almost all $x\in \mathbb{R}^n$ and there exists $c_\alpha \in \mathbb{R}$ such that
$$
\partial_x ^{\alpha}\mathcal{A}^{-|\alpha|/2}(f)(x)=\lim_{\varepsilon \rightarrow 0^+}\int_{|x-y|>\varepsilon}R_\alpha (x,y)f(y)dy+c_\alpha f(x),\quad \mbox{for almost all }x\in \mathbb{R}^n,
$$
where $c_\alpha =0$ if $\alpha _i$ is odd for some $i=1,...,n$.

Furthermore, when $n=1$ and $\alpha$ is even, the last integral is actually absolutely convergent for every $x\in \mathbb{R}$ and in this case no principal value is needed.
Here 
$$
R_\alpha (x,y)=\frac{1}{\Gamma (\frac{|\alpha|}{2})}\int_0^\infty   \partial _x^{\alpha}T_t^\mathcal{A}(x,y)t^{\frac{|\alpha |}{2}-1}dt,\quad x,y\in \mathbb{R}^n,\;x\not=y.
$$
\end{thm}

Let $\alpha \in \mathbb{N}^n$. Since, for every $\ell \in \mathbb{N}$,
\begin{equation}\label{derivH}
\frac{d}{dz}\widetilde{H}_\ell (z)=-\widetilde{H}_{\ell +1}(z),\quad z\in \mathbb{R},
\end{equation}
we have that, for every $k\in \mathbb{N}^n$,
$$
\partial ^{\alpha}_x\mathcal{A}^{-|\alpha|/2}(\widetilde{H}_k)(x)=\frac{(-1)^{|\alpha|}}{(|k|+n)^{\frac{|\alpha|}{2}}}\widetilde{H}_{k+\alpha}(x),\quad x\in \mathbb{R}^n.
$$
Let $f\in L^2(\mathbb{R}^n,\gamma _{-1})$. We can write $f=\sum_{k\in \mathbb{N}^n}c_k(f)\widetilde{H}_k$ in $L^2(\mathbb{R}^n,\gamma _{-1})$, where, for every $k\in \mathbb{N}^n$,
$$
c_k(f)=\frac{\pi ^{\frac{n}{2}}}{\|\widetilde{H}_k\|_{L^2(\mathbb{R}^n,\gamma _{-1})}^2}\int_{\mathbb{R}^n}f(y)\widetilde{H}_k(y)e^{|y|^2}dy.
$$

We define 
$$
R_\alpha f=(-1)^{|\alpha|}\sum_{k\in \mathbb{N}^n}\frac{c_k(f)}{(|k|+n)^{\frac{|\alpha|}{2}}}\widetilde{H}_{k+\alpha}.
$$
For every $k=(k_1,...,k_n)\in \mathbb{N}^n$,
\begin{equation}\label{norma2}
\|\widetilde{H}_{k+\alpha}\|_{L^2(\mathbb{R}^n,\gamma _{-1})}^2=\int_{\mathbb{R}^n}(H_{k+\alpha }(y))^2e^{-|y|^2}dy=\pi ^n 2^{|\alpha|+|k|}\prod_{i=1}^n\Gamma (k_i+\alpha _i+1).
\end{equation}
Then
\begin{align*}
\|R_\alpha f\|_{L^2(\mathbb{R}^n,\gamma _{-1})}^2&=\sum_{k\in \mathbb{N}^n}\frac{|c_k(f)|^2\|\widetilde{H}_{k+\alpha}\|_{L^2(\mathbb{R}^n, \gamma _{-1})}^2}{(|k|+n)^{|\alpha |}}\\
&=\sum_{k\in \mathbb{N}^n}(c_k(f)\|\widetilde{H}_{k}\|_{L^2(\mathbb{R}^n, \gamma _{-1})})^2\frac{\|\widetilde{H}_{k+\alpha}\|_{L^2(\mathbb{R}^n, \gamma _{-1})}^2}{(|k|+n)^{|\alpha |}\|\widetilde{H}_{k}\|_{L^2(\mathbb{R}^n, \gamma _{-1})}^2}\\
&\leq C\sum_{k\in \mathbb{N}^n}(c_k(f)\|\widetilde{H}_{k}\|_{L^2(\mathbb{R}^n, \gamma _{-1})})^2=C\|f\|_{L^2(\mathbb{R}^n,\gamma_{-1})}^2.
\end{align*}
Hence $R_\alpha $ is bounded from $L^2(\mathbb{R}^n,\gamma_{-1})$ into itself.

If $f\in C_c^\infty (\mathbb{R}^n)$, then $R_\alpha (f)(x)=\partial ^{\alpha}_x\mathcal{A}^{-|\alpha|/2}(f)(x)$, $x\in \mathbb{R}^n$.

\begin{thm}\label{Th1.2}
Let $\alpha =(\alpha _1,...,\alpha _n)\in \mathbb{N}^n\setminus\{0\}$ and $1<p<\infty$. The Riesz transform $R_\alpha$ can be extended from $L^2(\mathbb{R}^n,\gamma_{-1})\cap L^p(\mathbb{R}^n,\gamma_{-1})$ to $L^p(\mathbb{R}^n,\gamma_{-1})$ as a bounded operator from $L^p(\mathbb{R}^n,\gamma_{-1})$ into itself. By denoting again $R_\alpha$ to this extension we have that, there exists $c_\alpha \in \mathbb{R}$ such that, for every $f\in L^p(\mathbb{R}^n,\gamma_{-1})$,
$$
R_\alpha (f)(x)=\lim_{\varepsilon \rightarrow 0^+}\int_{|x-y|>\varepsilon}R_\alpha (x,y)f(y)dy+c_\alpha f(x),\quad \mbox{for almost all }x\in \mathbb{R}^n,
$$
where $c_\alpha =0$ if $\alpha_i$ is odd for some $i=1,...,n$. 

When $n=1$ and $\alpha\in \mathbb{N}$ is even the integral defining $R_\alpha$ is absolutely convergent.
\end{thm}

As it was mentioned, Bruno and Sj\"ogren (\cite[Theorem 1.1]{BrSj}) proved that $R_\alpha $ is bounded from $L^1(\mathbb{R}^n,\gamma_{-1})$ into $L^{1,\infty}(\mathbb{R}^n,\gamma_{-1})$ if and only if $1\leq |\alpha|\leq 2$. This property also holds in the Gaussian setting (see \cite{FS} and \cite{GCMST1}). Aimar, Forzani and Scotto (\cite{AFS}) introduced Riesz type operators $\mathfrak{R}_\alpha$, $\alpha \in \mathbb{N}^n\setminus\{0\}$, related to the Ornstein-Uhlenbeck operator.  $\mathfrak{R}_\alpha$ is bounded from $L^1(\mathbb{R}^n,\gamma_1)$ into $L^{1,\infty}(\mathbb{R}^n,\gamma_1)$ for every $\alpha \in \mathbb{N}^n\setminus\{0\}$. Motivated by the results in \cite{AFS} we define Riesz transform in the inverse Gaussian setting whose behavior in $L^1(\mathbb{R}^n,\gamma_{-1})$ is different from the one for $R_\alpha$.

We can write $\mathcal{A}_0=\sum_{i=1}^n\delta _i\partial_{x_i}$, where, for every $i=1,...,n$, $\delta _i=-\frac{1}{2}e^{-x_i^2}\partial _{x_i}e^{x_i^2}$. We consider the operator $\mathcal{\bar{A}}_0=\sum_{i=1}^n\partial_{x_i}\delta_i$. We have that $\mathcal{\bar{A}}_0=-n+\mathcal{A}_0$. Let $\mathcal{\bar A}$ the closure of $\mathcal{\bar A}_0$ in $L^2(\mathbb{R}^n,\gamma_{-1})$. For every $k\in \mathbb{N}^n$, $\mathcal{\bar A}\widetilde{H}_k=|k|\widetilde{H}_k$, and the spectrum of $\mathcal{\bar A}$ in $L^2(\mathbb{R}^n,\gamma_{-1})$ is the set of nonnegative integers. For every $\ell , m\in \mathbb{N}$, we have that $\delta_u^m\widetilde{H}_\ell (u)=(-1)^m\frac{\Gamma (\ell +1)}{\Gamma (\ell -m +1)}\widetilde{H}_{\ell -m}(u)$, $u\in \mathbb{R}$. Here, we understand $\widetilde{H}_\ell=0$ when $\ell <0$. Then, for every $\alpha =(\alpha _1,...,\alpha _n)\in \mathbb{N}^n$, by denoting $\delta ^\alpha=\prod_{i=1}^n\delta _i^{\alpha_i}$, we get
$$
\delta ^\alpha \widetilde{H}_k=(-1)^{|\alpha|}\prod_{i=1}^n\frac{\Gamma (k_i+1)}{\Gamma (k_i-\alpha _i+1)}\widetilde{H}_{k-\alpha},\quad k=(k_1,...,k_n)\in \mathbb{N}^n,\;k_r\geq \alpha _r, \;r=1,...,n.
$$
If $\alpha =(\alpha _1,...,\alpha_n)\in \mathbb{N}^n\setminus\{0\}$ and $k=(k_1,...,k_n)\in \mathbb{N}^n\setminus\{0\}$, with $k_r\geq \alpha _r$, $r=1,...,n$,
$$
\delta ^\alpha \mathcal{\bar A}^{-|\alpha|/2}(\widetilde{H}_k)=\frac{(-1)^{|\alpha|}}{|k|^{\frac{|\alpha|}{2}}}\prod_{i=1}^n\frac{\Gamma (k_i+1)}{\Gamma (k_i-\alpha _i+1)}\widetilde{H}_{k-\alpha}.
$$
In other case, $\delta^\alpha \mathcal{\bar A}^{-|\alpha|/2}(\widetilde{H}_k)=0$ (see Section \ref{barA} for details).

Let $\alpha =(\alpha _1,...,\alpha_n)\in \mathbb{N}^n\setminus\{0\}$. We define the Riesz transform $\bar{R}_\alpha$ on $L^2(\mathbb{R}^n,\gamma_{-1})$ as follows
$$
\overline{R}_\alpha (f)=(-1)^{|\alpha|}\sum_{\substack{k=(k_1,...,k_n)\in \mathbb{N}^n \\k_r\geq \alpha_r, r=1,...,n}}\frac{1}{|k|^{\frac{|\alpha|}{2}}}\prod_{i=1}^n \frac{\Gamma (k_i+1)}{\Gamma (k_i-\alpha _i+1)}c_k(f)\widetilde{H}_{k-\alpha}.
$$
Thus, $\overline{R}_\alpha$ is bounded from $L^2(\mathbb{R}^n,\gamma_{-1})$ into itself. If $f\in C_c^\infty (\mathbb{R}^n)$ and $c_0(f)=0$ then $\overline{R}_\alpha f=\delta ^\alpha \mathcal{\bar A}^{-|\alpha|/2}f$.

\begin{thm}\label{Th1.3}
Let $\alpha =(\alpha _1,...,\alpha_n)\in \mathbb{N}^n\setminus\{0\}$. The Riesz transform $\overline{R}_\alpha$ can be extended from $L^2(\mathbb{R}^n,\gamma_{-1})\cap L^p(\mathbb{R}^n,\gamma_{-1})$ to $L^p(\mathbb{R}^n,\gamma_{-1})$ as a bounded operator from

(i) $L^p(\mathbb{R}^n,\gamma_{-1})$ into itself, for every $1<p<\infty$.

(ii) $L^1(\mathbb{R}^n,\gamma_{-1})$ into $L^{1,\infty}(\mathbb{R}^n,\gamma_{-1})$, provided that $n=1$ or $|\alpha|>n$, when $n>1$. 

By denoting again $\overline{R}_\alpha$ to the extension we have that, for every $f\in L^p(\mathbb{R}^n,\gamma_{-1})$, $1\leq p<\infty$, 
$$
\overline{R}_\alpha (f)(x)=\lim_{\varepsilon \rightarrow 0^+}\int_{|x-y|>\varepsilon}\overline{R}_\alpha (x,y)f(y)dy+c_\alpha f(x),\quad \mbox{for almost all }x\in \mathbb{R}^n,
$$
where $c_\alpha=0$ when $\alpha _i$ is odd for some $i=1,...,n$.
Here
$$
\overline{R}_\alpha (x,y)=\frac{1}{\Gamma (\frac{|\alpha|}{2})}\int_0^\infty \delta_x^\alpha T_t^\mathcal{\bar{A}}(x,y)t^{\frac{|\alpha|}{2}-1}dt,\quad x,y\in \mathbb{R}^n,\;x\not=y.
$$
\end{thm}

Let $X$ be a Banach space. Suppose that $\{M_r\}_{r=1}^m$ is a $X$-valued martingale. The sequence $\{d_r=M_r-M_{r-1}\}_{r=1}^m$, where $M_0$ is understood as 0, is called the martingale difference associated with $\{M_r\}_{r=1}^m$. We say that $\{d_r\}_{r=1}^m$ is a $L^p$-martingale difference sequence when it is the difference sequence associated with a $L^p$-martingale. If $1<p<\infty$, $X$ is said to be a $UMD_p$-space when there exists $\beta >0$ such that for all $X$-valued $L^p$-martingale difference sequence $\{d _r\}_{r=1}^m$ and for all $(\varepsilon _r)_{r=1}^m\in \{-1,1\}^{m}$,
$$
\mathbb{E}\Big\|\sum_{r=1}^m\varepsilon _rd_r\Big\|^p\leq \beta \mathbb{E}\Big\|\sum_{r=1}^md_r\Big\|^p.
$$

$UMD$ is an abbreviation of unconditional martingale difference. If $X$ is $UMD_p$ for some $1<p<\infty$, then $X$ is $UMD_p$ for every $1<p<\infty$. This fact justifies to call $UMD$ to the property without any reference to $p$. Burkholder \cite{Bu} and Bourgain \cite{Bou} proved that the $UMD$ property of $X$ is necessary and sufficient for the boundedness of the Hilbert transform in $L^p(\mathbb{R},X)$, $1<p<\infty$. The $UMD$ property is a central notion in the development of the harmonic analysis when the functions are taking values in infinite dimensional spaces. $UMD$ Banach spaces have been characterized by using other singular integrals that can be seen as Riesz transforms associated to orthogonal systems (see \cite{AT}, \cite{BFMT}, \cite{BFRT} and \cite{HTV}, for instance). In the following result we characterize the Banach spaces with the $UMD$ property by using Riesz transforms in the inverse Gaussian setting. For every $i=1,...,n$, we define $e^i=(e_1^i,...,e_n^i)$ where $e_j^i=0$, $i\not =j$, and $e_i^i=1$. 

\begin{thm}\label{Th1.4}
Let $X$ be a Banach space. The following assertions are equivalent.

(i) $X$ is $UMD$.

(ii) For every $i=1,...,n$, $R_{e^i}$ can be extended from $L^p(\mathbb{R}^n,\gamma _{-1})\otimes X$ to $L^p(\mathbb{R}^n,\gamma _{-1},X)$ as a bounded operator from $L^p(\mathbb{R}^n,\gamma _{-1},X)$ into itself, for every $1<p<\infty$.

(iii) For every $i=1,...,n$, $R_{e^i}$ can be extended from $L^p(\mathbb{R}^n,\gamma _{-1})\otimes X$ to $L^p(\mathbb{R}^n,\gamma _{-1},X)$ as a bounded operator from $L^p(\mathbb{R}^n,\gamma _{-1}, X)$ into itself, for some $1<p<\infty$.

(iv) For every $i=1,...,n$, $R_{e^i}$ can be extended from $L^1(\mathbb{R}^n,\gamma _{-1})\otimes X$ to $L^1(\mathbb{R}^n,\gamma _{-1},X)$ as a bounded operator from $L^1(\mathbb{R}^n,\gamma _{-1}, X)$ into $L^{1,\infty}(\mathbb{R}^n,\gamma _{-1}, X)$.

Also the equivalences hold when in the properties (ii), (iii) and (iv) we replace $R_{e^i}$ by the maximal operator $R_{e^i}^*$ defined by
$$
R_{e^i}^*(f)(x)=\sup_{\varepsilon >0}\Big\|\int_{|x-y|>\varepsilon}R_{e^i}(x,y)f(y)dy\Big\|,\quad x\in \mathbb{R}^n\mbox{ and }i=1,...,n.
$$
\end{thm}

\begin{thm}\label{Th1.5}
Let $X$ be a Banach space. The following assertions are equivalent. 

(i) $X$ is $UMD$.

(ii) For every $1\leq p<\infty$ there exists, for each $i=1,...,n$, the limit
$$
\lim_{\varepsilon \rightarrow 0^+}\int_{|x-y|>\varepsilon }R_{e^i}(x,y)f(y)dy,\quad \mbox{ for almost all }x\in \mathbb{R}^n,
$$
for every $f\in L^p(\mathbb{R}^n,\gamma _{-1},X)$.

(iii) For some $1\leq p<\infty$, there exists, for every $i=1,...,n$, the limit
$$
\lim_{\varepsilon\rightarrow  0^+}\int_{|x-y|>\varepsilon}R_{e^i}(x,y)f(y)dy,\quad \mbox{ for almost all }x\in \mathbb{R}^n,
$$
for each $f\in L^p(\mathbb{R}^n,\gamma _{-1},X)$.

(iv) For every $1\leq p<\infty$, $f\in L^p(\mathbb{R}^n,\gamma _{-1},X)$ and $i=1,...,n$, $R_{e^i}^*(f)(x)<\infty$, for almost all $x\in \mathbb{R}^n$.

(v) For some $1\leq p<\infty$ and for every $f\in L^p(\mathbb{R}^n,\gamma _{-1},X)$ and $i=1,...,n$, $R_{e^i}^*(f)(x)<\infty$, for almost all $x\in \mathbb{R}^n$.
\end{thm}

The properties stated in Theorems \ref{Th1.4} and \ref{Th1.5} can also be established when we replace $R$-Riesz transforms by $\overline{R}$-Riesz transforms.

Next aim is to state characterizations of the Banach spaces with the $UMD$ property by using imaginary powers $\mathcal{A}^{i\gamma}$, $\gamma \in \mathbb{R}\setminus\{0\}$, of $\mathcal{A}$. Salogni (\cite[Theorem 3.4.3]{Sa}) proved that, for every $1<p<\infty$ and $\gamma \in \mathbb{R}\setminus\{0\}$,
$$
\|\mathcal{A}^{i\gamma }\|_{L^p(\mathbb{R}^n,\gamma_{-1})\rightarrow L^p(\mathbb{R}^n,\gamma _{-1})}\sim e^{\phi^*_p|\gamma|},
$$
as $\gamma \rightarrow \infty$, when $\phi _p=\arcsin |\frac{2}{p}-1|$. Actually, $\mathcal{A}^{i\gamma}$ is a Laplace transform type multiplier associated with $\mathcal{A}$ defined by the function $\phi _\gamma (t)=\frac{t^{-i\gamma}}{\Gamma (1-i\gamma)}$, $t>0$, for every $\gamma \in \mathbb{R}\setminus\{0\}$. Then, since $\{T_t^\mathcal{A}\}_{t>0}$ is a Stein diffusion semigroup, the $L^p(\mathbb{R}^n,\gamma_{-1})$-boundedness of $\mathcal{A}^{i\gamma}$, $\gamma \in \mathbb{R}\setminus\{0\}$, follows from the general results established in \cite[Chapter III]{StLP}. Recently, Bruno (\cite[Theorem 4.1]{Br}) proved that $\mathcal{A}^{i\gamma}$, $\gamma \in \mathbb{R}\setminus\{0\}$, is bounded from $L^1(\mathbb{R}^n,\gamma_{-1})$ into $L^{1,\infty}(\mathbb{R}^n,\gamma_{-1})$.

Let $\gamma \in \mathbb{R}\setminus\{0\}$. We have that
$$
\mathcal{A}^{i\gamma}f=\sum_{k\in \mathbb{N}^n}(|k|+m)^{i\gamma}c_k(f)\widetilde{H}_k,\quad f\in L^2(\mathbb{R}^n,\gamma_{-1}).
$$
It is immediate to see that $\mathcal{A}^{i\gamma}$ is bounded from $L^2(\mathbb{R}^n,\gamma_{-1})$ into itself. For every $f\in L^2(\mathbb{R}^n,\gamma_{-1})\otimes X$ we define $\mathcal{A}^{i\gamma}$ in the obvious way when $X$ is a Banach space. In order to the operator $\mathcal{A}^{i\gamma}$ is bounded from $L^2(\mathbb{R}^n,\gamma_{-1})\otimes X$ into itself as subspace of $L^2(\mathbb{R}^n,\gamma_{-1},X)$ we need to impose some additional property to the Banach space $X$. For instance, if $X$ is isomorphic to a Hilbert space then $\mathcal{A}^{i\gamma}$ can be extended from $L^2(\mathbb{R}^n,\gamma_{-1})\otimes X$ to $L^2(\mathbb{R}^n,\gamma_{-1},X)$ as a bounded operator from $L^2(\mathbb{R}^n,\gamma_{-1},X)$ into itself. We are going to characterize the $UMD$ Banach spaces as those Banach spaces for which $\mathcal{A}^{i\gamma}$ can be extended from $(L^2(\mathbb{R}^n,\gamma_{-1})\cap L^p
(\mathbb{R}^n,\gamma_{-1}))\otimes X$ to $L^p(\mathbb{R}^n,\gamma_{-1},X)$ as a bounded operator from $L^p(\mathbb{R}^n,\gamma_{-1},X)$ into itself, when $1<p<\infty$, and from $L^1(\mathbb{R}^n,\gamma_{-1},X)$ into $L^{1,\infty}(\mathbb{R}^n,\gamma_{-1},X)$. Our result is motivated by the one in \cite[p. 402]{G} where $UMD$ Banach spaces are characterized by the $L^p(\mathbb{R},dx)$-boundedness properties of the imaginary power $(-\frac{d^2}{dx^2})^{i\gamma}$, $\gamma \in \mathbb{R}\setminus\{0\}$, of $-\frac{d^2}{dx^2}$. Guerre-Delabri\'ere's result was extended to higher dimensions by considering imaginary powers of the Laplacian in \cite[Proposition 1]{BCFR}. In \cite{BCCR} this kind of characterization for $UMD$ Banach spaces is obtained in Hermite and Laguerre settings. As far as we know this property has not been proved for the Ornstein-Uhlenbeck operator in the Gaussian framework.

\begin{thm}\label{Th1.6}
Let $\gamma \in \mathbb{R}\setminus\{0\}$. For every $f\in L^2(\mathbb{R}^n,\gamma_{-1})$, we have that
$$
\mathcal{A}^{i\gamma}(f)(x)=\lim_{\varepsilon \rightarrow 0^+}\Big(\int_{|x-y|>\varepsilon }K_\gamma ^\mathcal{A}(x,y)f(y)dy+\alpha (\varepsilon )f(x)\Big),\quad \mbox{ for almost all }x\in \mathbb{R}^n,
$$
where 
$$
K_\gamma ^\mathcal{A}(x,y)=-\int_0^\infty \phi _\gamma (t)\partial_tT_t^\mathcal{A}(x,y)dt,\quad x,y\in \mathbb{R}^n,\;x\not=y,
$$
and 
$$
\alpha (\varepsilon)=\frac{1}{\Gamma (\frac{n}{2})}\int_0^\infty \phi_\gamma \big(\frac{\varepsilon ^2}{4u}\big)e^{-u}u^{\frac{n}{2}-1}du,\quad \varepsilon \in (0,\infty),
$$
being $\phi_\gamma(t)=\frac{t^{-i\gamma }}{\Gamma (1-i\gamma )}$, $t\in (0,\infty)$.

Let $X$ be a Banach space. The following assertions are equivalent.

(i) $X$ is $UMD$.

(ii) For every $1<p<\infty$, $\mathcal{A}^{i\gamma}$ can be extended from $(L^2(\mathbb{R}^n,\gamma_{-1})\cap L^p(\mathbb{R}^n,\gamma_{-1}))\otimes X$ to $L^p(\mathbb{R}^n,\gamma_{-1},X)$ as a bounded operator from $L^p(\mathbb{R}^n,\gamma_{-1},X)$ into itself.

(iii) For some $1<p<\infty$, $\mathcal{A}^{i\gamma}$ can be extended from $(L^2(\mathbb{R}^n,\gamma_{-1})\cap L^p(\mathbb{R}^n,\gamma_{-1}))\otimes X$ to $L^p(\mathbb{R}^n,\gamma_{-1},X)$ as a bounded operator from $L^p(\mathbb{R}^n,\gamma_{-1},X)$ into itself.

(iv) $\mathcal{A}^{i\gamma}$ can be extended from $(L^2(\mathbb{R}^n,\gamma_{-1})\cap L^1(\mathbb{R}^n,\gamma_{-1}))\otimes X$ to $L^1(\mathbb{R}^n,\gamma_{-1},X)$ as a bounded operator from $L^1(\mathbb{R}^n,\gamma_{-1},X)$ into $L^{1,\infty}(\mathbb{R}^n,\gamma_{-1},X)$.

We define the maximal operator $\mathcal{A}^{i\gamma}_*$ by
$$
\mathcal{A}_*^{i\gamma}(f)(x)=\sup_{\varepsilon >0}\left\|\int_{|x-y|>\varepsilon}K_\gamma ^\mathcal{A}(x,y)f(y)dy\right\|,\quad f\in \bigcup_{p\geq 1}L^p(\mathbb{R}^n,\gamma _{-1},X).
$$
The following assertions are equivalent to (i).

(v) For every $1<p<\infty$, $\mathcal{A}_*^{i\gamma}$ is bounded from $L^p(\mathbb{R}^n,\gamma _{-1},X)$ into itself.

(vi) For some $1<p<\infty$, $\mathcal{A}_*^{i\gamma}$ is bounded from $L^p(\mathbb{R}^n,\gamma _{-1},X)$ into itself. 

(vii) $\mathcal{A}_*^{i\gamma}$ is bounded from $L^1(\mathbb{R}^n,\gamma _{-1},X)$ into $L^{1,\infty}(\mathbb{R}^n,\gamma _{-1},X)$.

(viii) For every $1\leq p<\infty$ and every $f\in L^p(\mathbb{R}^n,\gamma _{-1},X)$ there exists the limit 
$$
\lim_{\varepsilon \rightarrow  0^+}\left(\int_{|x-y|>\varepsilon}K_\gamma ^\mathcal{A}(x,y)f(y)dy+\alpha (\varepsilon)f(x)\right),\quad \mbox{ for almost all }x\in \mathbb{R}^n.
$$

(ix) For some $1\leq p<\infty$ and every $f\in L^p(\mathbb{R}^n,\gamma _{-1},X)$ there exists the limit 
$$
\lim_{\varepsilon \rightarrow  0^+}\left(\int_{|x-y|>\varepsilon}K_\gamma ^\mathcal{A}(x,y)f(y)dy+\alpha (\varepsilon)f(x)\right),\quad \mbox{ for almost all }x\in \mathbb{R}^n.
$$

(x) For every $1\leq p<\infty$ and every $f\in L^p(\mathbb{R}^n,\gamma _{-1},X)$, $\mathcal{A}_*^{i\gamma}(f)(x)<\infty$, for almost all $x\in \mathbb{R}^n$.

(xi) For some $1\leq p<\infty$ and every $f\in L^p(\mathbb{R}^n,\gamma _{-1},X)$, $\mathcal{A}_*^{i\gamma}(f)(x)<\infty$, for almost all $x\in \mathbb{R}^n$.
\end{thm}

This paper is organized as follows. In Section 2 we study the negative power $\mathcal{A}^{-\beta}$, $\beta >0$, of $\mathcal{A}$. Higher order Riesz transforms in the inverse Gauss setting are considered in Section 3 where we prove Theorems \ref{Th1.1} and \ref{Th1.2}. Theorem \ref{Th1.3} is established in Section 4 and Theorems \ref{Th1.4} and \ref{Th1.5} in Section 5.  Finally, Section 6 is devoted to show the proof of Theorem \ref{Th1.6}.

Throughout this paper $C$ and $c$ denote positive constants that can change in each occurrence.

\section{Negative powers of $\mathcal{A}$.}

In this section we prove $L^p(\mathbb{R}^n,\gamma _{-1})$-boundedness properties of the negative powers $\mathcal{A}^{-\beta}$, $\beta >0$, of $\mathcal{A}$. These properties are different than the ones of the negative powers of the Ornstein-Uhlenbeck operator $\mathcal{L}=-\frac{1}{2}\Delta +x\nabla$. We prove that, for every $\beta >0$, $\mathcal{A}^{-\beta}$ defines a bounded operator from $L^1(\mathbb{R}^n,\gamma _{-1})$ into $L^{1,\infty}(\mathbb{R}^n,\gamma _{-1})$. However, in \cite[Proposition 6.2]{GCMST2} it was proved that if $\beta >0$, $\mathcal{L}^{-\beta}$ is not bounded from $L^1(\mathbb{R}^n,\gamma _1)$ into $L^{1,\infty}(\mathbb{R}^n,\gamma _1)$.

Let $\beta >0$. We define 
\begin{equation}\label{Abeta}
\mathcal{A}^{-\beta}(f)=\sum_{k\in \mathbb{N}^n}\frac{c_k(f)}{(|k|+n)^\beta}\widetilde{H}_k,\quad f\in L^2(\mathbb{R}^n,\gamma _{-1}).
\end{equation}

$\mathcal{A}^{-\beta}$ is bounded from $L^2(\mathbb{R}^n,\gamma _{-1})$ into itself. Moreover,  when $\beta >1$ the series also converges pointwisely in $\mathbb{R}^n$. Indeed, let $f\in L^2(\mathbb{R}^n,\gamma _{-1})$. We have that
$$
|c_k(f)|\leq \pi ^{\frac{n}{2}}\|f\|_{L^2(\mathbb{R}^n,\gamma _{-1})}\|\widetilde{H}_k\|_{L^2(\mathbb{R}^n,\gamma _{-1})}^{-1},\quad k\in \mathbb{N}^n.
$$
Also, for every $k\in \mathbb{N}^n$, (see \eqref{norma2}),
\begin{equation}\label{2.1}
\|\widetilde{H}_k\|_{L^2(\mathbb{R}^n,\gamma _{-1})}=\pi ^{\frac{n}{2}}2^{\frac{|k|}{2}}\Big(\prod_{j=1}^n\Gamma (k_j+1)\Big)^{\frac{1}{2}},
\end{equation}
and, according to \cite[p. 324]{San},
\begin{equation}\label{2.2}
|H_j(z)|\leq 2\sqrt{\Gamma (j+1)}2^{\frac{j}{2}}e^{\frac{z^2}{2}},\quad z\in \mathbb{R}\mbox{ and }j\in \mathbb{N}.
\end{equation}
Then, 
$$
|c_k(f)\widetilde{H}_k(x)|\leq 2^ne^{-\frac{|x|^2}{2}}\|f\|_{L^2(\mathbb{R}^n,\gamma_{-1})},\quad x\in \mathbb{R}^n,\;k\in \mathbb{N}^n,
$$
and if $\beta >1$ it follows that
$$
\sum_{k\in \mathbb{N}^n}\frac{|c_k(f)|}{(|k|+n)^\beta}|\widetilde{H}_k(x)|\leq Ce^{-\frac{|x|^2}{2}}\|f\|_{L^2(\mathbb{R}^n,\gamma _{-1})},\quad x\in \mathbb{R}^n.
$$

The series in \eqref{Abeta} converges pointwise absolutely for each $\beta >0$ when $f\in C_c^\infty (\mathbb{R}^n)$. Indeed, let $f\in C_c^\infty (\mathbb{R}^n)$. Partial integration allows us to see that, for every $r\in \mathbb{N}$ there exists $C=C(f,r)>0$ such that
\begin{equation}\label{2.3}
|c_k(f)|\leq \frac{C}{(|k|+n)^r\|\widetilde{H}_k\|_{L^2(\mathbb{R}^n,\gamma_{-1})}},\quad k\in \mathbb{N}^n.
\end{equation}
Then, 
$$
\sum_{k\in \mathbb{N}^n}\frac{|c_k(f)|}{(|k|+n)^\beta}|\widetilde{H}_k(x)|\leq Ce^{-\frac{|x|^2}{2}}\sum_{k\in \mathbb{N}^n}\frac{1}{(|k|+n)^{\beta +r}}<\infty .
$$

We also consider the operator defined, for every $f\in L^2(\mathbb{R}^n,\gamma_{-1})$, by
$$
S_\beta(f)=\frac{1}{\Gamma (\beta)}\int_0^\infty T_t^\mathcal{A}(f)t^{\beta -1}dt,
$$
where the integral is understood in the $L^2(\mathbb{R}^n,\gamma _{-1})$-Bochner sense.

For each $t>0$ we can write 
$$
T_t^\mathcal{A} (f)=\sum_{k\in \mathbb{N}^n}e^{-(|k|+n)t}c_k(f)\widetilde{H}_k,\quad f\in L^2(\mathbb{R}^n,\gamma_{-1}).
$$

We obtain
\begin{align*}
\|T_t^\mathcal{A}(f)\|_{L^2(\mathbb{R}^n,\gamma_{-1})}^2&=\sum_{k\in \mathbb{N}^n}e^{-2(|k|+n)t}|c_k(f)|^2\|\widetilde{H}_k\|_{L^2(\mathbb{R}^n,\gamma_{-1})}^2
\leq e^{-2nt}\sum_{k\in \mathbb{N}^n}|c_k(f)|^2\|\widetilde{H}_k\|_{L^2(\mathbb{R}^n,\gamma_{-1})}^2\\
&=e^{-2nt}\|f\|_{L^2(\mathbb{R}^n,\gamma_{-1})}^2,\quad f\in L^2(\mathbb{R}^n, \gamma_{-1}),\;t>0.
\end{align*}
Then,
\begin{align*}
\|S_\beta(f)\|_{L^2(\mathbb{R}^n,\gamma_{-1})}&\leq \frac{1}{\Gamma(\beta)}\int_0^\infty \|T_t^\mathcal{A}(f)\|_{L^2(\mathbb{R}^n,\gamma_{-1})}t^{\beta -1}dt\leq \frac{\|f\|_{L^2(\mathbb{R}^n,\gamma_{-1})}}{\Gamma(\beta)}\int_0^\infty e^{-nt}t^{\beta -1}dt\\
&\leq \frac{\|f\|_{L^2(\mathbb{R}^n,\gamma_{-1})}}{n^\beta},\quad f\in L^2(\mathbb{R}^n,\gamma_{-1}).
\end{align*}

Suppose now that $f\in C_c^\infty(\mathbb{R}^n)$. By using \eqref{2.3} we obtain that there exists $C=C(f)>0$ such that
$$
|T_t^\mathcal{A}(f)(x)|\leq Ce^{-\frac{|x|^2}{2}}e^{-nt},\quad x\in \mathbb{R}^n,\;t>0.
$$
We can write
\begin{align*}
S_\beta(f)(x)&=\frac{1}{\Gamma (\beta)}\sum_{k\in \mathbb{N}^n}c_k(f)\widetilde{H}_k(x)\int_0^\infty e^{-(|k|+n)t}t^{\beta -1}dt\\
&=\sum_{k\in \mathbb{N}^n}\frac{c_k(f)}{(|k|+n)^\beta}\widetilde{H}_k(x)=\mathcal{A}^{-\beta}(f)(x),\quad x\in \mathbb{R}^n.
\end{align*}
Since $C_c^\infty (\mathbb{R}^n)$ is dense in $L^2(\mathbb{R}^n,\gamma_{-1})$, $\mathcal{A}^{-\beta}(f)=S_\beta(f)$, $f\in L^2(\mathbb{R}^n,\gamma_{-1})$.

According to \cite[Theorem 2.5]{Br} we have that, for every $f\in C_c^\infty (\mathbb{R}^n)$
$$
\mathcal{A}^{-\beta}(f)(x)=\int_{\mathbb{R}^n}M_\beta (x,y)f(y)dy,
$$
for all $x$ outside the support of $f$, where
$$
M_\beta (x,y)=\frac{1}{\Gamma (\beta)}\int_0^\infty T_t^\mathcal{A}(x,y)t^{\beta -1}dt,\quad x,y\in \mathbb{R}^n,\;x\not =y.
$$

\begin{prop}\label{Prop2.1}
Let $\beta >0$. The operator $\mathcal{A}^{-\beta}$ can be extended from $L^p(\mathbb{R}^n, \gamma_{-1})\cap L^2(\mathbb{R}^n, \gamma_{-1})$ to $L^p(\mathbb{R}^n, \gamma_{-1})$ as a bounded operator from $L^p(\mathbb{R}^n, \gamma_{-1})$ into itself, when $1<p<\infty$, and from $L^1(\mathbb{R}^n, \gamma_{-1})$ into $L^{1,\infty}(\mathbb{R}^n, \gamma_{-1})$.
\end{prop}

\begin{proof}
We use the method consisting in decomposing the operator $\mathcal{A}^{-\beta}$ in two parts called local and global parts. This procedure of decomposition was employed by Muckenhoupt (\cite{Mu1} and \cite{Mu2}) in the Gaussian setting.

From now on we consider the function $m$ given by $m(x)=\min\{1,\frac{1}{|x|^2}\}$, $x\in \mathbb{R}^n\setminus\{0\}$, and $m(0)=1$ and the region $N$ defined by
$$
N=\left\{(x,y)\in \mathbb{R}^n\times \mathbb{R}^n:|x-y|\leq n\sqrt{m(x)}\right\}.
$$
We decompose $\mathcal{A}^{-\beta}$ as follows
$$
\mathcal{A}^{-\beta}=\mathcal{A}_{\rm loc}^{-\beta}+\mathcal{A}_{\rm glob}^{-\beta},
$$
where $\mathcal{A}_{\rm loc}^{-\beta}(f)(x)=\mathcal{A}^{-\beta}(f\chi _N(x,\cdot))(x)$, $x\in \mathbb{R}^n$.

According to \cite[Lemma 3.3.1]{Sa} we get
$$
M_\beta (x,y)\leq C\left(\int_0^{m(x)}\frac{e^{-\frac{c|x-y|^2}{t}}}{t^{\frac{n}{2}}}t^{\beta -1}dt+\int_{m(x)}^\infty \frac{e^{-nt}}{(1-e^{-2t})^{\frac{n}{2}}}t^{\beta -1}dt\right),\quad (x,y)\in N.
$$
By choosing $0<\varepsilon <\min\{2\beta,n\}$ we obtain
\begin{align*}
M_\beta(x,y)&\leq C\left(\int_0^{m(x)}\frac{e^{-\frac{c|x-y|^2}{t}}}{t^{\frac{n-\varepsilon}{2}}}t^{\beta -\frac{\varepsilon}{2}-1}dt+\int_{m(x)}^\infty \frac{dt}{t^{\frac{n-\varepsilon}{2}+1}}\right)\\
&\leq C\left(\frac{1}{|x-y|^{n-\varepsilon}}\int_0^1t^{\beta -\frac{\varepsilon}{2}-1}dt+m(x)^{-\frac{n-\varepsilon}{2}}\right)\leq \frac{C}{|x-y|^{n-\varepsilon}},\quad (x,y)\in N.
\end{align*}
In the last inequality we have taken into account that $\sqrt{m(x)}\sim \frac{1}{1+|x|}$, $x\in \mathbb{R}^n$, and that $(1+|x|)|x-y|\leq C$, provided that $(x,y)\in N$.

We have that
\begin{align*}
\sup_{x\in \mathbb{R}^n}\int_{\mathbb{R}^n}M_\beta(x,y)\chi_N(x,y)dy&\leq C\sup_{x\in \mathbb{R}^n}\int_{|x-y|\leq n\sqrt{m(x)}}\frac{dy}{|x-y|^{n-\varepsilon}}\\
&\leq C\sup_{x\in \mathbb{R}^n}\int_0^{n\sqrt{m(x)}}r^{\varepsilon -1}dr\leq C\sup_{x\in \mathbb{R}^n}m(x)^{\frac{\varepsilon}{2}}<\infty.
\end{align*}

Also, since $\sqrt{m(x)}\sim \frac{1}{1+|x|}\sim\frac{1}{1+|y|}\sim \sqrt{m(y)}$, $(x,y)\in N$,
$$
\sup_{y\in \mathbb{R}^n}\int_{\mathbb{R}^n}M_\beta(x,y)\chi_N(x,y)dx<\infty.
$$
Hence, the operator $M_{\beta ,{\rm loc}}$ defined by
$$
M_{\beta ,{\rm loc}}(f)(x)=\int_{\mathbb{R}^n}M_\beta(x,y)\chi_N(x,y)f(y)dy,\quad x\in \mathbb{R}^n,
$$
is bounded from $L^p(\mathbb{R}^n,dx)$ into itself, for every $1\leq p<\infty$. Since $M_{\beta ,{\rm loc}}$ is a local operator, by using \cite[Proposition 3.2.5]{Sa} we deduce that $M_{\beta, {\rm loc}}$ is bounded from $L^p(\mathbb{R}^n,\gamma_{-1})$ into itself, for every $1\leq p<\infty$.

Suppose that $f,g\in C_c^\infty (\mathbb{R}^n)$. As above we take $0<\varepsilon <\min\{2\beta, n\}$ and we get
\begin{align*}
\int_{\mathbb{R}^n}|f(x)|\int_{\mathbb{R}^n}\int_0^\infty t^{\beta -1}T_t^\mathcal{A}(x,y)\chi _N(x,y)|g(y)|dtdydx\leq C\int_{\mathbb{R}^n}|f(x)|\int_{\mathbb{R}^n}|g(y)|\frac{\chi _N(x,y)}{|x-y|^{n-\varepsilon}}dydx<\infty .
\end{align*}
Then,
$$
\int_{\mathbb{R}^n}f(x)\mathcal{A}_{\rm loc}^{-\beta}(g)(x)dx=\int_{\mathbb{R}^n}f(x)M_{\beta ,{\rm loc}}(g)(x)dx.
$$

We deduce that $\mathcal{A}_{\rm loc}^{-\beta}(g)(x)=M_{\beta,{\rm loc}}(g)(x)$, for almost all $x\in \mathbb{R}^n$. Since $\mathcal{A}_{\rm loc}^{-\beta}$ and $M_{\beta,{\rm loc}}$ are bounded from $L^2(\mathbb{R}^n,\gamma_{-1})$ into itself, $\mathcal{A}_{\rm loc}^{-\beta}(f)=M_{\beta,{\rm loc}}(f)$, $f\in L^2(\mathbb{R}^n,\gamma_{-1})$. It follows that, for every $1\leq p<\infty$, $\mathcal{A}_{\rm loc}^{-\beta}$ can be extended from $L^2(\mathbb{R}^n,\gamma_{-1})\cap L^p(\mathbb{R}^n,\gamma_{-1})$ to $L^p(\mathbb{R}^n,\gamma_{-1})$ as a bounded operator from $L^p(\mathbb{R}^n,\gamma_{-1})$ into itself.

We now study the operator $M_{\beta,{\rm glob}}$ defined by
$$
M_{\beta,{\rm glob}}(f)(x)=\int_{\mathbb{R}^n}M_\beta (x,y)\chi_{N^c}(x,y)f(y)dy,\quad x\in \mathbb{R}^n.
$$
By making the change of variables $s=1-e^{-2t}$, $t\in (0,\infty)$, and taking into account that $|x-ry|^2=|y-rx|^2+(1-r^2)(|x|^2-|y|^2)$, $x,y\in \mathbb{R}^n$, $r\in \mathbb{R}$, we obtain
$$
M_\beta (x,y)=\frac{\pi ^{-\frac{n}{2}}}{2^\beta\Gamma (\beta)}e^{|y|^2-|x|^2}\int_0^1\frac{e^{-\frac{|y-x\sqrt{1-s}|^2}{s}}}{s^{\frac{n}{2}}}(1-s)^{\frac{n}{2}-1}(-\log (1-s))^{\beta -1}ds,\quad x,y\in \mathbb{R}^n.
$$

We now use some notations that were introduced in \cite{PeSo} and proceed as in the proof of \cite[Proposition 2.2]{PeSo}. For every $x,y\in \mathbb{R}^n$ we define 
\begin{equation}\label{uPeSo}
a=|x|^2+|y|^2,\quad b=2\langle x,y\rangle,\quad  s_0=2\frac{\sqrt{a^2-b^2}}{a+\sqrt{a^2-b^2}}\quad  \mbox{and}\quad u(s)=\frac{|y-x\sqrt{1-s}|^2}{s}, \;\;s\in (0,1). 
\end{equation}
Assume that $(x,y)\not \in N$. Suppose first that $b\leq 0$. In this case, $u(s)\geq \frac{a}{s}-|x|^2$, $s\in (0,1)$. Furthermore, $s\leq C(-\log (1-s))\leq C(1-s)^{-1/(4\beta)}$, $s\in (0,1)$. Then, 
$$
M_\beta (x,y)\leq Ce^{|y|^2}\int_0^1\frac{e^{-\frac{a}{s}}}{s^{\frac{n}{2}+1}(1-s)^{\frac{3}{4}}}ds.
$$
By making the change of variable $r=\frac{1}{s}-1$, $s\in (0,1)$, and taking into account that $a\geq \frac{1}{2}$ we obtain 
$$
M_\beta (x,y)\leq Ce^{-|x|^2}\int_0^\infty \frac{e^{-\frac{r}{2}}(1+r)^{\frac{n}{2}-\frac{1}{4}}}{r^{\frac{3}{4}}}dr\leq Ce^{-|x|^2}.
$$

Suppose now $b>0$. We write $u_0=u(s_0)=\frac{|y|^2-|x|^2}{2}+\frac{|x+y||x-y|}{2}$. By using that 
$$
\sup_{s\in (0,1)}\frac{e^{-u(s)}}{s^{\frac{n}{2}}}\sim \frac{e^{-u_0}}{s_0^{n/2}},
$$ 
we have that
$$
M_\beta (x,y)\leq Ce^{|y|^2-|x|^2}\sup_{s\in (0,1)}\frac{e^{-u(s)}}{s^{\frac{n}{2}}}\int_0^1(1-s)^{\frac{n}{2}-1}(-\log (1-s))^{\beta -1}ds\leq Ce^{|y|^2-|x|^2}\frac{e^{-u_0}}{s_0^{n/2}}.
$$
Since $s_0\sim \frac{\sqrt{a^2-b^2}}{a}$ we conclude that, when $(x,y)\not \in N$,
\begin{equation}\label{Mbeta}
M_\beta (x,y)\leq C\left\{
\begin{array}{ll}
e^{-|x|^2},&\mbox{ if }\langle x,y\rangle \leq 0,\\
\displaystyle \Big(\frac{|x+y|}{|x-y|}\Big)^{\frac{n}{2}}\exp\Big(\frac{|y|^2-|x|^2}{2}-\frac{|x-y||x+y|}{2}\Big),&\mbox{ if }\langle x,y\rangle >0.
\end{array}
\right.
\end{equation}

Let $1<q<\infty$. Since $||y|^2-|x|^2|\leq |x+y||x-y|$, $x,y\in \mathbb{R}^n$, and $|x-y||x+y|\geq n$, when $(x,y)\in N^c$, as in \cite[p. 501]{Pe} we obtain
\begin{align*}
\int_{\mathbb{R}^n}e^{\frac{|x|^2}{q}-\frac{|y|^2}{q}}M_\beta (x,y)\chi_{N^c}(x,y)dy&\leq C\left(\int_{\mathbb{R}^n}e^{-|x|^2(1-\frac{1}{q})}e^{-\frac{|y|^2}{q}}dy\right.\\
&\quad +\left.\int_{\mathbb{R}^n}|x+y|^ne^{-|x+y||x-y|(\frac{1}{2}-|\frac{1}{q}-\frac{1}{2}|)}dy\right)\leq C,\quad x\in \mathbb{R}^n.
\end{align*}

Also, we have that
$$
\sup_{y\in \mathbb{R}^n}\int_{\mathbb{R}^n}e^{\frac{|x|^2}{q}-\frac{|y|^2}{q}}M_\beta (x,y)\chi_{N^c}(x,y)dx<\infty .
$$
We deduce that $M_{\beta ,{\rm glob}}$ is bounded from $L^p(\mathbb{R}^n,\gamma_{-1})$ into itself, for every $1<p<\infty$. 

Next, we are going to see that $M_{\beta ,{\rm glob}}$ is bounded from $L^1(\mathbb{R}^n,\gamma_{-1})$ into $L^{1,\infty}(\mathbb{R}^n,\gamma_{-1})$. We decompose $M_\beta (x,y)$, $x,y\in \mathbb{R}^n$, as follows 
$$
M_\beta (x,y)=\frac{\pi ^{-\frac{n}{2}}}{2^\beta \Gamma (\beta)}e^{|y|^2-|x|^2}\left(\int_0^{\frac{1}{2}}+\int_{\frac{1}{2}}^1\right)\frac{e^{-\frac{|y-x\sqrt{1-s}|^2}{s}}}{s^{\frac{n}{2}}}(1-s)^{\frac{n}{2}-1}(-\log (1-s))^{\beta -1}ds=I_1(x,y)+I_2(x,y).
$$

For every $x,y\in \mathbb{R}^n\setminus\{0\}$ we denote by $\theta (x,y)\in [0,\pi ]$ the angle between $x$ and $y$ (we understand $\theta (x,y)=0$, when $n=1$). By using \cite[Lemma 3.3.3]{Sa} we get that, for every $(x,y)\in N^c$, $x,y\not=0$,
\begin{align*}
I_1(x,y)&\leq Ce^{|y|^2-|x|^2}\sup_{s\in (0,1)}\frac{e^{-\frac{|y-x\sqrt{1-s}|^2}{s}}(1-s)^{\frac{n}{2}}}{s^{\frac{n}{2}}}\int_0^{\frac{1}{2}}\frac{(-\log (1-s))^{\beta -1}}{1-s}ds\\
&\leq Ce^{|y|^2-|x|^2}\sup_{r\in (0,1)}\frac{e^{-\frac{|(1+r)y-(1-r)x|^2}{r}}(1-r)^n}{r^{\frac{n}{2}}}\leq Ce^{|y|^2-|x|^2}\min\{(1+|x|)^n,(|x|\sin \theta (x,y))^{-n}\}.
\end{align*}
On the other hand, by proceeding as in \cite[Proposition 5.1]{BrSj} we estimate $I_2(x,y)$, $(x,y)\in N^c$. We first observe that
\begin{align*}
I_2(x,y)&\leq Ce^{|y|^2-|x|^2}\int_{\frac{1}{2}}^1e^{-|y-x\sqrt{1-s}|^2}(1-s)^{\frac{n}{2}-1}(-\log (1-s))^{\beta -1}ds\quad x,y\in \mathbb{R}^n.
\end{align*}
If $|y|\geq 2|x|$ then $|y-x\sqrt{1-s}|\geq 3|y|/4$, $s\in (1/2,1)$, and it follows that
$$
I_2(x,y)\leq Ce^{-c|y|^2}e^{|y|^2-|x|^2}\int_{\frac{1}{2}}^1(1-s)^{\frac{n}{2}-1}(-\log (1-s))^{\beta -1}ds\leq C\frac{e^{|y|^2-|x|^2}}{|y|^{n-1}}\leq C\frac{e^{|y|^2-|x|^2}}{|x|^{n-1}},\quad x,y\not=0.
$$
If $x,y\in \mathbb{R}^n\setminus\{0\}$ we define $r_0=\frac{|y|}{|x|}\cos \theta (x,y)$ and we write $y=y_x+y_\bot$, where $y_x$ is parallel to $x$ and $y_\bot$ is orthogonal to $y$.

By making the change of variables $r=\sqrt{1-s}$ and since $|y-rx|^2=|r-r_0|^2|x|^2+|y_\bot|^2$, it follows that, when  if $|y|\leq 2|x|$,
\begin{align*}
I_2(x,y)&\leq Ce^{-c|y_\bot|^2}e^{|y|^2-|x|^2}\int_0^{\frac{1}{2}}(|r-r_0|^{n-1}+|r_0|^{n-1})(-\log r)^{\beta -1}e^{-c|r-r_0|^2|x|^2}dr\\
&\leq Ce^{-c|y_\bot|^2}e^{|y|^2-|x|^2}\left(\frac{1}{|x|^{n-1}}+\Big(\frac{|y|}{|x|}\Big)^{n-1}\right)\int_0^{\frac{1}{2}}(-\log r)^{\beta -1}dr\\
&\leq Ce^{-c|y_\bot|^2}e^{|y|^2-|x|^2}\left (\frac{1}{|x|^{n-1}}+|x|\Big(\frac{|y|}{|x|}\Big)^{n-1}\right),\quad (x,y)\in N^c.
\end{align*}
In the last inequality we have used that $|x|\geq C$ when $(x,y)\in N^c$ and $2|x|\geq |y|$.

By combining the above estimates we obtain
\begin{align}\label{2.4}
M_\beta (x,y)&\leq Ce^{|y|^2-|x|^2}\Big(\min\big\{(1+|x|)^n,(|x|\sin \theta (x,y))^{-n}\big\}+\frac{1}{|x|^{n-1}}\\ \nonumber
&\quad+e^{-c|y_\bot|^2}|x|\Big(\frac{|y|}{|x|}\Big)^{n-1}\chi_{\{|y|\leq 2|x|\}}(x,y)
\Big),\quad (x,y)\in N^c.
\end{align}
According to \cite[Lemma 3.3.4]{Sa} and \cite[Lemmas 4.2 and 4.3]{BrSj} we can see that the operator $M_{\beta ,{\rm glob}}$ is bounded from $L^1(\mathbb{R}^n,\gamma_{-1})$ into $L^{1,\infty}(\mathbb{R}^n,\gamma_{-1})$.

The estimation \eqref{2.4} allows us to prove that, for every $g\in C_c^\infty (\mathbb{R}^n)$, $\mathcal{A}_{\rm glob}^{-\beta}(f)(x)=M_{\beta, {\rm glob}}(f)(x)$, for almost all $x\in \mathbb{R}^n$. Then, since $\mathcal{A}_{\rm glob}^{-\beta}$ and $M_{\beta, {\rm glob}}$ are bounded from $L^2(\mathbb{R}^n,\gamma_{-1})$ into itself, $\mathcal{A}_{\rm glob}^{-\beta}(f)=M_{\beta, {\rm glob}}(f)$, $f\in L^2(\mathbb{R}^n,\gamma_{-1})$. Hence, $\mathcal{A}_{\rm glob}^{-\beta}$ can be extended from $L^2(\mathbb{R}^n,\gamma_{-1})\cap L^p(\mathbb{R}^n,\gamma_{-1})$ to $L^p(\mathbb{R}^n,\gamma_{-1})$ as a bounded operator from $L^p(\mathbb{R}^n,\gamma_{-1})$ into itself, when $1<p<\infty$, and from $L^1(\mathbb{R}^n,\gamma_{-1})$ into $L^{1,\infty}(\mathbb{R}^n,\gamma_{-1})$. 

Thus the proof is finished.
\end{proof}

\section{Higher order Riesz transforms associated with the operator $\mathcal{A}$}

In this section we prove Theorems \ref{Th1.1} and \ref{Th1.2} concerning to the higher order Riesz transforms in the inverse Gaussian setting.

\subsection{Proof of Theorem \ref{Th1.1}.} Let $f\in C_c^\infty (\mathbb{R}^n)$. If $\ell\in \mathbb{N}^n$ we have that
$$
\partial_x ^\ell T_t^\mathcal{A}(x,y)=(-1)^{|\ell |}\frac{e^{-nt}e^{-\frac{|x-e^{-t}y|^2}{1-e^{-2t}}}}{\pi^{\frac{n}{2}}(1-e^{-2t})^{\frac{n+|\ell|}{2}}}H_\ell \Big(\frac{x-e^{-t}y}{\sqrt{1-e^{-2t}}}\Big),\quad x,y\in \mathbb{R}^n,\;t>0,
$$
and then
\begin{equation}\label{derivT}
|\partial _x^\ell T_t^\mathcal{A}(x,y)|\leq C\frac{e^{-nt}e^{-c\frac{|x-e^{-t}y|^2}{1-e^{-2t}}}}{(1-e^{-2t})^{\frac{n+|\ell|}{2}}},\quad x,y\in \mathbb{R}^n,\;t>0.
\end{equation}

Suppose that $k\in \mathbb{N}$ and $\ell \in \mathbb{N}^n$ such that $|\ell|<k$. Then,
\begin{equation}\label{3.1}
\int_{\mathbb{R}^n}|f(y)|\int_0^\infty \big|\partial_x^\ell T_t^\mathcal{A}(x,y)\big|t^{\frac{k}{2}-1}dtdy<\infty,\quad x\in \mathbb{R}^n.
\end{equation}

Indeed, by considering the function $m$ defined in Proposition \ref{Prop2.1} and using \eqref{derivT} and \cite[Lemma 3.3.1]{Sa} we obtain, for $\varepsilon \in (0,1)$, 
\begin{align*}
\int_0^\infty \big|\partial ^\ell _xT_t^\mathcal{A}(x,y)\big|t^{\frac{k}{2}-1}dt&\leq C\left(\int_0^{m(x)}\frac{e^{-c\frac{|x-y|^2}{t}}}{t^{\frac{n+|\ell|-k}{2}+1}}dt+\int_{m(x)}^\infty \frac{e^{-nt}}{(1-e^{-2t})^{\frac{n+|\ell|}{2}}}t^{\frac{k}{2}-1}dt\right)\\
&\hspace{-4cm}\leq C\left(\frac{1}{|x-y|^{n-\varepsilon}}\int_0^1t^{\frac{k-|\ell|-\varepsilon}{2}-1}dt+\int_{m(x)}^\infty \frac{dt}{t^{\frac{n+|\ell|-k}{2}+1}}\right)\leq C\left(\frac{1}{|x-y|^{n-\varepsilon}}+\frac{1}{m(x)^{\frac{n+|\ell|-k}{2}}}\right)\\
&\hspace{-4cm}\leq  C\left(\frac{1}{|x-y|^{n-\varepsilon}}+\frac{1}{m(x)^{\frac{n-\varepsilon}{2}}}\right)\leq \frac{C}{|x-y|^{n-\varepsilon}},\quad (x,y)\in N.
\end{align*}
On the other hand, by reading the proof of \cite[Lemma 3.3.3]{Sa} we deduce that
\begin{align}\label{3.2}
\int_0^\infty \big|\partial_x^\ell T_t^\mathcal{A}(x,y)\big|t^{\frac{k}{2}-1}dt&\leq Ce^{|y|^2-|x|^2}\sup_{t\in (0,\infty)}\frac{e^{-c\frac{|y-e^{-t}x|^2}{1-e^{-2t}}}}{(1-e^{-2t})^{\frac{n+|\ell|}{2}}}\int_0^\infty e^{-nt}t^{\frac{k}{2}-1}dt\nonumber\\
&\leq Ce^{|y|^2-|x|^2}(1+|x|)^{n+|\ell|},\quad (x,y)\in N^c.
\end{align}
Note that this estimation also holds when $|\ell|=k$.

Since $f\in C_c^\infty (\mathbb{R}^n)$, \eqref{3.1} holds. Hence, according to \cite[Lemma 4.2]{BFRT}, we have that, when $k\in \mathbb{N}$, $\ell \in \mathbb{N}^n$ and $|\ell|<k$,
\begin{equation}\label{derivA}
\partial_x^\ell\mathcal{A}^{-k/2}(f)(x)=\frac{1}{\Gamma (\frac{k}{2})}\int_{\mathbb{R}^n}f(y)\int_0^\infty \partial _x^\ell T_t^\mathcal{A}(x,y)t^{\frac{k}{2}-1}dtdy,\quad \mbox{ for almost all }x\in \mathbb{R}^n.
\end{equation}

We assume $\alpha =(\alpha _1,...,\alpha _n)\in \mathbb{N}^n\setminus\{0\}$, so we can suposse without loss of generality that $\alpha _1\geq 1$. Let us take $\ell =(\alpha _1-1,\alpha_2,...,\alpha _n)$. According to \eqref{derivA} we can write
$$
\partial_x^\alpha \mathcal{A}^{-|\alpha|/2}(f)(x)=\partial_{x_1}\left(\frac{1}{\Gamma (\frac{|\alpha|}{2})}\int_{\mathbb{R}^n}f(y)\int_0^\infty \partial _x^\ell T_t^\mathcal{A}(x,y)t^{\frac{|\alpha|}{2}-1}dtdy\right),\quad \mbox{for almost }x\in \mathbb{R}^n.
$$
Assume first $n>1$ and write 
$$
\int_{\mathbb{R}^n}f(y)\int_0^\infty \partial _x^\ell T_t^\mathcal{A}(x,y)t^{\frac{|\alpha|}{2}-1}dtdy=F(x)+G(x),\quad x\in \mathbb{R}^n,
$$
where
$$
F(x)=\int_{\mathbb{R}^n}f(y)\int_0^\infty \partial _x^\ell [T_t^\mathcal{A}(x,y)-W_t(x-y)]t^{\frac{|\alpha|}{2}-1}dtdy,\quad x\in \mathbb{R}^n,
$$
and
$$
G(x)=\int_{\mathbb{R}^n}f(y)\int_0^\infty \partial _x^\ell W_t(x-y)t^{\frac{|\alpha|}{2}-1}dtdy,\quad x\in \mathbb{R}^n.
$$
Here $W_t$, $t>0$, denotes the classical heat kernel $W_t(z)=\frac{e^{-|z|^2/(2t)}}{(2\pi t)^{n/2}}$, $z\in \mathbb{R}^n$.

Next we show that 
\begin{equation}\label{derivF}
\partial_{x_1}F(x)=\int_{\mathbb{R}^n}f(y)\int_0^\infty \partial _x^\alpha [T_t^\mathcal{A}(x,y)-W_t(x-y)]t^{\frac{|\alpha|}{2}-1}dtdy,\quad \mbox{for almost }x\in \mathbb{R}^n.
\end{equation}

By taking into account \eqref{derivT} we get 
\begin{align}\label{dif1}
\int_{m(x)}^\infty \big|\partial _x^\alpha [T_t^\mathcal{A}(x,y)-W_t(x-y)]\big|t^{\frac{|\alpha|}{2}-1}dt&\leq C\int_{m(x)}^\infty \left(\frac{e^{-nt}t^{\frac{|\alpha|}{2}-1}}{(1-e^{-2t})^{\frac{n+|\alpha|}{2}}}+\frac{e^{-c\frac{|x-y|^2}{t}}}{t^{\frac{n}{2}+1}}\right)dt\\\nonumber
&\leq C\int_{m(x)}^\infty \frac{dt}{t^{\frac{n}{2}+1}}\leq \frac{C}{m(x)^{\frac{n}{2}}}\leq C(1+|x|)^n,\quad x,y\in \mathbb{R}^n.
\end{align}

Also by using \eqref{3.2} we obtain, 
\begin{align}\label{dif2}
\int_0^{m(x)} \left|\partial _x^\alpha [T_t^\mathcal{A}(x,y)-W_t(x-y)]\right|t^{\frac{|\alpha|}{2}-1}dt&\leq C\left(e^{|y|^2-|x|^2}(1+|x|)^{n+|\alpha|}+\int_0^{m(x)} \frac{e^{-c\frac{|x-y|^2}{t}}}{t^{\frac{n}{2}+1}}dt\right)\\\nonumber
&\hspace{-5cm}\leq C\left(e^{|y|^2-|x|^2}(1+|x|)^{n+|\alpha|}+\frac{1}{|x-y|^{n+|\alpha|}}\right)\leq Ce^{|y|^2}(1+|x|)^{n+|\alpha|},\quad (x,y)\in N^c.
\end{align}
Now we are going to estimate
$$
I(x,y)=\int_0^{m(x)}\ \left|\partial _x^\alpha [T_t^\mathcal{A}(x,y)-W_t(x-y)]\right|t^{\frac{|\alpha|}{2}-1}dt,\quad (x,y)\in N.
$$
We can write 
\begin{align}\label{decomposition}
\partial _x^\alpha [T_t^\mathcal{A}(x,y)-W_t(x-y)]&=\frac{(-1)^{|\alpha|}}{\pi ^{\frac{n}{2}}}\left(\frac{e^{-nt}}{(1-e^{-2t})^{\frac{n+|\alpha|}{2}}}\widetilde{H}_\alpha\left(\frac{x-e^{-t}y}{\sqrt{1-e^{-2t}}}\right)-\frac{1}{(2t)^{\frac{n+|\alpha|}{2}}}\widetilde{H}_\alpha\left(\frac{x-y}{\sqrt{2t}}\right)\right)\nonumber\\
&\hspace{-4cm}=\frac{(-1)^{|\alpha|}}{\pi ^{\frac{n}{2}}}\left\{\frac{e^{-nt}-1}{(1-e^{-2t})^{\frac{n+|\alpha|}{2}}}\widetilde{H}_\alpha\left(\frac{x-e^{-t}y}{\sqrt{1-e^{-2t}}}\right)+\left(\frac{1}{(1-e^{-2t})^{\frac{n+|\alpha|}{2}}}-\frac{1}{(2t)^{\frac{n+|\alpha|}{2}}}\right)\widetilde{H}_\alpha\left(\frac{x-e^{-t}y}{\sqrt{1-e^{-2t}}}\right)\right.\nonumber\\
&\hspace{-4cm}\quad \left.+\frac{1}{(2t)^{\frac{n+|\alpha|}{2}}}\left(\widetilde{H}_\alpha\left(\frac{x-e^{-t}y}{\sqrt{1-e^{-2t}}}\right)-\widetilde{H}_\alpha \left(\frac{x-y}{\sqrt{2t}}\right)\right)\right\}=\sum_{j=1}^3 I_j(t,x,y),\quad x,y\in \mathbb{R}^n.
\end{align}
Then
$$
I(x,y)\leq \sum_{j=1}^3\int_0^{m(x)}|I_j(t,x,y)|t^{\frac{|\alpha|}{2}-1}dt=\sum_{j=1}^3I_j(x,y),\quad x,y\in \mathbb{R}^n.
$$
 
By proceeding as in the proof of \cite[Lemma 3.3.1]{Sa} we can see that
\begin{equation}\label{3.3}
\frac{e^{-c\frac{|x-e^{-t}y|^2}{1-e^{-2t}}}}{(1-e^{-2t})^{\frac{n}{2}}}\leq C\frac{e^{-c\frac{|x-y|^2}{t}}}{t^{\frac{n}{2}}},\quad (x,y)\in N,\;t\in (0,1).
\end{equation}
Then, by taking into account that
\begin{equation}\label{differences}
|e^{-nt}-1|\leq Ct\quad \mbox{ and }\quad \left|\frac{1}{(1-e^{-2t})^{\frac{n+|\alpha|}{2}}}-\frac{1}{(2t)^{\frac{n+|\alpha|}{2}}}\right|\leq \frac{C}{t^{\frac{n+|\alpha|}{2}-1}},\quad t\in (0,1),
\end{equation}
it follows by using \eqref{3.3} that
$$
I_1(x,y)+I_2(x,y)\leq C\int_0^{m(x)}\frac{e^{-c\frac{|x-y|^2}{t}}}{t^{\frac{n}{2}}}dt\leq \frac{C}{|x-y|^{n-\frac{1}{2}}}\int_0^1\frac{dt}{t^{\frac{1}{4}}}\leq \frac{C}{|x-y|^{n-\frac{1}{2}}},\quad (x,y)\in N.
$$

Let us analyze the term $I_3(x,y)$, $(x,y)\in N$. For every $z=(z_1,...,z_n)$, $w=(w_1,...,w_n)\in \mathbb{R}^n$ we write
\begin{align*}
\widetilde{H}_\alpha (z)-\widetilde{H}_\alpha (w)&=\prod_{j=1}^n \widetilde{H}_{\alpha_j}(z_j)-\prod_{j=1}^n\widetilde{H}_{\alpha _j}(w_j)\\
&=\sum_{k=1}^n\left(\prod_{j=1}^{k-1} \widetilde{H}_{\alpha_j}(w_j)\Big( \widetilde{H}_{\alpha_k}(z_k)- \widetilde{H}_{\alpha_k}(w_k)\Big)\prod_{j=k+1}^n \widetilde{H}_{\alpha_j}(z_j)\right).
\end{align*}

Let $(x,y)\in N$, $t\in (0,1)$, and consider $z=\frac{x-e^{-t}y}{\sqrt{1-e^{-2t}}}$ and $w=\frac{x-y}{\sqrt{2t}}$. By taking into account that $(1+|x|+|y|)|x-y|\leq C$ it follows that $e^{-c|z|^2}\leq Ce^{-c|w|^2}$ and then, for each $k=1,...,n$,
$$
\prod_{j=1}^{k-1} |\widetilde{H}_{\alpha_j}(w_j)|\prod_{j=k+1}^n |\widetilde{H}_{\alpha_j}(z_j)|\leq C\exp\Big(-c\sum_{j=1, j\not=k}^n\frac{|x_j-y_j|^2}{t}\Big),
$$
and, by considering also \eqref{derivH} and using the mean value theorem we get
\begin{align*}
\Big|\widetilde{H}_{\alpha_k}(z_k)- \widetilde{H}_{\alpha_k}(w_k)\Big|&\leq Ce^{-c\frac{|x_k-y_k|^2}{t}}|z_k-w_k|\leq Ce^{-c\frac{|x_k-y_k|^2}{t}}\left(\frac{(1-e^{-t})|y_k|}{\sqrt{1-e^{-2t}}}+|x_k-y_k|\Big|\frac{1}{\sqrt{1-e^{-2t}}}-\frac{1}{\sqrt{2t}}\Big|\right)\\
&\hspace{-2cm}\leq Ce^{-c\frac{|x_k-y_k|^2}{t}}\sqrt{t}(|y_k|+|x_k-y_k|)\leq Ce^{-c\frac{|x_k-y_k|^2}{t}}\sqrt{t}(|x_k|+|x_k-y_k|)\leq Ce^{-c\frac{|x_k-y_k|^2}{t}}\sqrt{t}(1+|x|). 
\end{align*}
Then, 
\begin{equation}\label{Halpha}
\left|\widetilde{H}_\alpha \left(\frac{x-e^{-t}y}{\sqrt{1-e^{-2t}}}\right)-\widetilde{H}_\alpha \left(\frac{x-y}{2t}\right)\right|\leq Ce^{-c\frac{|x-y|^2}{t}}\sqrt{t}(1+|x|),\quad (x,y)\in N,\;t\in (0,1),
\end{equation}
and thus,
\begin{align*}
I_3(x,y)&\leq C(1+|x|)\int_0^{m(x)}\frac{e^{-c\frac{|x-y|^2}{t}}}{t^{\frac{n+1}{2}}}dt\leq C\frac{1+|x|}{|x-y|^{n-\frac{1}{2}}}\int_0^{m(x)}t^{-\frac{3}{4}}dt\\
&=C\frac{(1+|x|)m(x)^{\frac{1}{4}}}{|x-y|^{n-\frac{1}{2}}}\leq C\frac{\sqrt{1+|x|}}{|x-y|^{n-\frac{1}{2}}},\quad (x,y)\in N.
\end{align*}
We deduce that 
\begin{equation}\label{acotI}
I(x,y)\leq C\frac{\sqrt{1+|x|}}{|x-y|^{n-\frac{1}{2}}},\quad (x,y)\in N.
\end{equation}
This estimation, jointly with \eqref{dif1} and \eqref{dif2}, leads to
\begin{align*}
\int_{\mathbb{R}^n}|f(y)|\int_0^\infty \left|\partial _x^\alpha [T_t^\mathcal{A}(x,y)-W_t(x-y)]\right|t^{\frac{|\alpha|}{2}-1}dtdy&\nonumber\\
&\hspace{-6cm}\leq C\int_{\mathbb{R}^n}|f(y)|\Big(e^{|y|^2}(1+|x|)^{n+|\alpha|}+\frac{\sqrt{1+|x|}}{|x-y|^{n-\frac{1}{2}}}\Big)dy \leq C(1+|x|)^{n+|\alpha|}<\infty,\quad x\in \mathbb{R}^n,
\end{align*}
where we have used that $f$ has compact support. According to \cite[Lemma 4.2]{BFRT} \eqref{derivF} is then established.

We are going to evaluate $\partial _{x_1}G(x)$. We write $G(x)=\int_{\mathbb{R}^n}f(x-y)\Phi(y)dy$, $x\in \mathbb{R}^n$, where
$$
\Phi (z)=\int_0^\infty \partial _z^\ell (W_t(z))t^{\frac{|\alpha|}{2}-1}dt,\quad z\in \mathbb{R}^n.
$$

Since $n>1$ we have that
$$
|\Phi (z)|\leq C\int_0^\infty \frac{e^{-c\frac{|z|^2}{t}}}{t^{\frac{n+1}{2}}}dt\leq \frac{C}{|z|^{n-1}},\quad z\in \mathbb{R}^n\setminus\{0\}.
$$
According to \cite[Lemma 4.2]{BFRT} we can derivate under the integral sign obtaining
$$
\partial_{x_1}G(x)=\int_{\mathbb{R}^n}\partial _{x_1}(f(x-y))\Phi (y)dy=-\int_{\mathbb{R}^n}\partial_{y_1}(f(x-y))\Phi (y)dy,\quad \mbox{ for almost all }x\in \mathbb{R}^n,
$$
where the last integral is absolutely convergent.

For every $z=(z_1,...,z_n)\in \mathbb{R}^n$  we define $\overline{z}=(z_2,...,z_n)$. Partial integration leads to
\begin{align*}
\int_{\mathbb{R}^n}\partial_{y_1}(f(x-y))\Phi (y)dy&=\lim_{\varepsilon \rightarrow 0^+}\int_{|y|>\varepsilon}\partial_{y_1}(f(x-y))\Phi (y)dy\\
&=\lim_{\varepsilon \rightarrow 0^+}\left(\int_{|\overline{y}|<\varepsilon}\int_{-\infty}^{-\sqrt{\varepsilon ^2-|\overline{y}|^2}}+\int_{|\overline{y}|<\varepsilon}\int_{\sqrt{\varepsilon ^2-|\overline{y}|^2}}^{+\infty}+\int_{|\overline{y}|>\varepsilon}\int_\mathbb{R}\right)\partial_{y_1}(f(x-y))\Phi (y)dy_1d\overline{y}\\
&=\lim_{\varepsilon \rightarrow 0^+}\left(-\int_{|y|>\varepsilon}f(x-y)\partial _{y_1}\Phi (y)dy+\int_{|\overline{y}|<\varepsilon}f(x-y)\Phi (y)\Big]_{y_1=-\infty}^{y_1=-\sqrt{\varepsilon ^2-|\overline{y}|^2}}d\overline{y}\right.\\
&\quad +\left.\int_{|\overline{y}|<\varepsilon}f(x-y)\Phi (y)\Big]_{y_1=\sqrt{\varepsilon ^2-|\overline{y}|^2}}^{y_1=+\infty}d\overline{y}\right)\\
&=\lim_{\varepsilon \rightarrow 0^+}\left(-\int_{|\overline{y}|>\varepsilon}f(x-y)\partial_{y_1}\Phi (y)dy+J_\varepsilon (x)\right),\quad x\in \mathbb{R}^n,
\end{align*}
where
\begin{align*}
J_\varepsilon (x)&=\int_{|\overline{y}|<\varepsilon}f(x_1+\sqrt{\varepsilon ^2-|\overline{y}|^2},\overline{x}-\overline{y})\Phi (-\sqrt{\varepsilon ^2-|\overline{y}|^2},\overline{y})d\overline{y}\\
&\quad -\int_{|\overline{y}|<\varepsilon}f(x_1-\sqrt{\varepsilon ^2-|\overline{y}|^2},\overline{x}-\overline{y})\Phi (\sqrt{\varepsilon ^2-|\overline{y}|^2},\overline{y})d\overline{y},\quad x\in \mathbb{R}^n.
\end{align*}

Let us estimate $\lim_{\varepsilon \rightarrow 0^+}J_\varepsilon (x)$, $x\in \mathbb{R}^n$. We recall that $\Phi$ can be written as follows
$$
\Phi (z)=\frac{(-1)^{|\alpha|-1}}{2^{\frac{|\ell|}{2}}(2\pi)^{\frac{n}{2}}}\int_0^\infty \widetilde{H}_\ell \Big(\frac{z}{\sqrt{2t}}\Big)\frac{dt}{t^{\frac{n+1}{2}}},\quad z\in \mathbb{R}^n.
$$
Suppose now that $\alpha _1$ is odd. Then,
$$
\Phi (-\sqrt{\varepsilon ^2-|\overline{y}|^2},\overline{y})=\Phi (\sqrt{\varepsilon ^2-|\overline{y}|^2},\overline{y}),\quad \overline{y}\in \mathbb{R}^{n-1},\;|\overline{y}|<\varepsilon.
$$
We have that, for every $x\in \mathbb{R}^n$ and $\varepsilon>0$,
\begin{align*}
J_\varepsilon (x)&=\int_{|\overline{y}|<\varepsilon}\big(f(x_1+\sqrt{\varepsilon ^2-|\overline{y}|^2},\overline{x}-\overline{y})-f(x_1-\sqrt{\varepsilon ^2-|\overline{y}|^2},\overline{x}-\overline{y})\big)\Phi (\sqrt{\varepsilon ^2-|\overline{y}|^2},\overline{y})d\overline{y}\\
&=\varepsilon ^{n-1}\int_{|\overline{z}|<1}\big(f(x_1+\varepsilon\sqrt{1-|\overline{z}|^2},\overline{x}-\varepsilon \overline{z})-f(x_1-\varepsilon\sqrt{1-|\overline{z}|^2},\overline{x}-\varepsilon\overline{z})\big)\Phi (\varepsilon\sqrt{1-|\overline{z}|^2},\varepsilon\overline{z})d\overline{z}.
\end{align*}
On the other hand, by performing the change of variable $s=\frac{\varepsilon ^2}{2t}$, we get
\begin{align*}
\Phi (\varepsilon\sqrt{1-|\overline{z}|^2},\varepsilon\overline{z})&=\frac{(-1)^{|\alpha|-1}}{2^{\frac{|\ell|}{2}}(2\pi)^{\frac{n}{2}}}\int_0^\infty \widetilde{H}_{\alpha _1-1}\Big(\frac{\varepsilon\sqrt{1-|\overline{z}|^2}}{\sqrt{2t}}\Big)\prod_{i=2}^n\widetilde{H}_{\alpha _i}\Big(\frac{\varepsilon z_i}{\sqrt{2t}}\Big)\frac{dt}{t^{\frac{n+1}{2}}}\\
&=\frac{(-1)^{|\alpha|-1}\varepsilon^{1-n}}{2^{\frac{|\alpha|}{2}}\pi^{\frac{n}{2}}}\int_0^\infty \widetilde{H}_{\alpha _1-1}(\sqrt{s(1-|\overline{z}|^2)})\prod_{i=2}^n\widetilde{H}_{\ell _i}(z_i\sqrt{s})s^{\frac{n-3}{2}}ds,\quad \varepsilon >0,\;\overline{z}\in \mathbb{R}^{n-1},\;|z|<1.
\end{align*}
It follows that
\begin{align*}
J_\varepsilon (x)&=\frac{(-1)^{|\alpha|-1}}{2^{\frac{|\alpha|}{2}}\pi^{\frac{n}{2}}}\int_{|\overline{z}|<1}\big(f(x_1+\varepsilon\sqrt{1-|\overline{z}|^2},\overline{x}-\varepsilon \overline{z})-f(x_1-\varepsilon\sqrt{1-|\overline{z}|^2},\overline{x}-\varepsilon \overline{z})\big)\\
&\quad \times\int_0^\infty H_{\alpha _1-1}(\sqrt{s(1-|\overline{z}|^2)})\prod_{i=2}^nH_{\ell _i}(z_i\sqrt{s})e^{-s}s^{\frac{n-3}{2}}ds d\overline{z},\quad x\in \mathbb{R}^n,\;\varepsilon >0.
\end{align*}
Then, by using the dominated convergence theorem we obtain
$$
\lim_{\varepsilon \rightarrow 0^+}J_\varepsilon (x)=0,\quad x\in \mathbb{R}^n.
$$

Suppose now that $\alpha _1$ is even. Then,
$$
\Phi (-\varepsilon\sqrt{1-|\overline{y}|^2},\overline{y})=-\Phi (\varepsilon\sqrt{1-|\overline{y}|^2},\overline{y}),\quad \varepsilon >0,\;\overline{y}\in \mathbb{R}^{n-1},\;|\overline{y}|<\varepsilon,
$$
and proceeding as above we get
\begin{align*}
J_\varepsilon (x)&=\frac{(-1)^{|\alpha|}}{2^{\frac{|\alpha|}{2}}\pi^{\frac{n}{2}}}\int_{|\overline{z}|<1}\big(f(x_1+\varepsilon\sqrt{1-|\overline{z}|^2},\overline{x}-\varepsilon \overline{z})+f(x_1-\varepsilon\sqrt{1-|\overline{z}|^2},\overline{x}-\varepsilon \overline{z})\big)\\
&\quad \times \int_0^\infty H_{\alpha _1-1}(\sqrt{s(1-|\overline{z}|^2)})\prod_{i=2}^nH_{\alpha _i}(z_i\sqrt{s})e^{-s}s^{\frac{n-3}{2}}dsd\overline{z},\quad x\in \mathbb{R}^n,\;\varepsilon >0.
\end{align*}
It follows that
$$
\lim_{\varepsilon \rightarrow 0^+}J_\varepsilon (x)=\frac{(-1)^{|\alpha|}}{2^{\frac{|\alpha|}{2}-1}\pi^{\frac{n}{2}}}f(x)\int_{|\overline{z}|<1}\int_0^\infty H_{\alpha _1-1}(\sqrt{s(1-|z|^2)})\prod_{i=2}^{n-1}H_{\alpha _i}(z_i\sqrt{s})e^{-s}s^{\frac{n-3}{2}}dsd\overline{z}=-c_\alpha f(x),\quad x\in \mathbb{R}^n.
$$
We note that if $\alpha _i$  is odd for some $i=2,...,n$, then $c_\alpha =0$. Thus, we conclude that
\begin{align}\label{derivG}
\partial _{x_1}G(x)&=\lim_{\varepsilon \rightarrow 0^+}\int_{|x-y|>\varepsilon}f(y)\partial _{x_1}\Phi (x-y)dy+c_\alpha f(x)\nonumber\\
&\quad =\lim_{\varepsilon \rightarrow 0^+}\int_{|x-y|>\varepsilon}f(y)\int_0^\infty \partial _x^\alpha W_t(x-y)t^{\frac{|\alpha|}{2}-1}dtdy+c_\alpha f(x),\quad \mbox{ for almost all }x\in \mathbb{R}^n,
\end{align}
where $c_\alpha=0$ when $\alpha_i$ is odd for some $i=1,...,n$.

We have obtained
\begin{align*}
\partial_x^{\alpha}\mathcal{A}^{-|\alpha|/2}(f)(x)&=\frac{1}{\Gamma (\frac{|\alpha|}{2})}\partial_{x_1}(F(x)+G(x))\\
&=\frac{1}{\Gamma (\frac{|\alpha|}{2})}\left(\int_{\mathbb{R}^n}f(y)\int_0^\infty \partial _x^\alpha [T_t^\mathcal{A}(x,y)-W_t(x-y)]t^{\frac{|\alpha|}{2}-1}dtdy\right.\\
&\quad \left.+\lim_{\varepsilon \rightarrow 0^+}\int_{|x-y|>\varepsilon}f(y)\int_0^\infty \partial _x^\alpha W_t(x-y)t^{\frac{|\alpha|}{2}-1}dtdy+c_\alpha f(x)\right)\\
&=\frac{1}{\Gamma (\frac{|\alpha|}{2})}\lim_{\varepsilon \rightarrow 0^+}\int_{|x-y|>\varepsilon }f(y)\int_0^\infty \partial _x^{\alpha }T_t^\mathcal{A}(x,y))t^{\frac{|\alpha|}{2}-1}dtdy+c_\alpha f(x),\quad \mbox{ for almost all }x\in \mathbb{R}^n. 
\end{align*}

We now deal with the case $n=1$. We have $\alpha \in \mathbb{N}$, $\alpha \geq 1$. According to \eqref{derivA} we can write 
$$
\frac{d^\alpha}{dx^\alpha}\mathcal{A}^{-\alpha /2}(f)(x)=\frac{1}{\Gamma (\frac{\alpha}{2})}\frac{d}{dx}\int_{\mathbb{R}^n}f(y)\int_0^\infty \partial_x^{\alpha -1}T_t^\mathcal{A}(x,y)t^{\frac{\alpha}{2}-1}dtdy=\frac{1}{\Gamma (\frac{\alpha}{2})}\frac{d}{dx}(\overline{F}(x)+\overline{G}(x)),\quad x\in \mathbb{R},
$$
where
$$
\overline{F}(x)=\int_\mathbb{R}f(y)\int_0^\infty\Big[\partial _x^{\alpha -1}T_t^\mathcal{A}(x,y)-\Big((\frac{d^{\alpha -1}}{dx^{\alpha -1}}W_t)(x-y)- (\frac{d^{\alpha -1}}{dx^{\alpha -1}}W_t)(0)\chi_{(1,\infty)}(t)\Big)\Big]t^{\frac{\alpha}{2}-1}dtdy,\quad x\in \mathbb{R},
$$
and 
$$
\overline{G}(x)=\int_\mathbb{R}f(y)\int_0^\infty \Big((\frac{d^{\alpha -1}}{dx^{\alpha -1}}W_t)(x-y)- (\frac{d^{\alpha -1}}{dx^{\alpha -1}}W_t)(0)\chi_{(1,\infty)}(t)\Big)t^{\frac{\alpha}{2}-1}dtdy,\quad x\in \mathbb{R}.
$$
Note that if $\alpha $ is even, then $(\frac{d^{\alpha -1}}{dx^{\alpha -1}}W_t)(0)=0$.

By proceeding as in the case of $n>1$ and taking into account that
\begin{align*}
\left|(\frac{d^{\alpha -1}}{dx^{\alpha -1}}W_t)(x-y)- (\frac{d^{\alpha -1}}{dx^{\alpha -1}}W_t)(0)\right|\leq \frac{C}{t^{\frac{\alpha}{2}}}\Big|\widetilde{H}_{\alpha -1}\big(\frac{x-y}{\sqrt{2t}}\big)-\widetilde{H}_{\alpha -1}(0)\Big|\leq C\frac{|x-y|}{t^{\frac{\alpha +1}{2}}},\quad x,y\in \mathbb{R},\;t>1,
\end{align*}
we can see that the integral defining $\overline{F}(x)$ is absolutely convergent for every $x\in \mathbb{R}$ and
$$
\frac{d}{dx}\overline{F}(x)=\int_\mathbb{R}f(y)\int_0^\infty \partial _x^\alpha [T_t^\mathcal{A}(x,y)-W_t(x-y)]t^{\frac{\alpha}{2}-1}dtdy,
$$
being also this integral absolutely convergent for almost $x\in \mathbb{R}$.

On the other hand, by considering
\begin{align*}
\overline{\Phi }(y)&=\int_0^\infty \Big((\frac{d^{\alpha -1}}{dx^{\alpha -1}}W_t)(x-y)- (\frac{d^{\alpha -1}}{dx^{\alpha -1}}W_t)(0)\chi_{(1,\infty)}(t)\Big)t^{\frac{\alpha}{2}-1}dt\\
&=\frac{(-1)^{\alpha -1}}{2^{\frac{\alpha }{2}}\sqrt{\pi}}\int_0^\infty \Big(\widetilde{H}_{\alpha-1}\Big(\frac{y}{\sqrt{2t}}\Big)-\widetilde{H}_{\alpha -1}(0)\chi_{(1,\infty )}(t)\Big)\frac{dt}{t},\quad y\in \mathbb{R}\setminus\{0\},
\end{align*}
we can write
$$
\overline{G}(x)=\int_\mathbb{R}f(x-y)\overline{\Phi}(y)dy,\quad x\in \mathbb{R}^n,
$$
and according to \cite[Lemma 4.2]{BFRT} we get
$$
\frac{d}{dx}\overline{G}(x)=-\int_\mathbb{R}\partial_y (f(x-y))\overline{\Phi}(y)dy\\
=-\lim_{\varepsilon \rightarrow 0^+}\int_{|y|>\varepsilon}\partial_y(f(x-y))\overline{\Phi}(y)dy,\quad \mbox{ for almost all }x\in \mathbb{R}.
$$

By partial integration we obtain
$$
\int_{|y|>\varepsilon}\partial_y(f(x-y))\overline{\Phi }(y)dy=-\int_{|y|>\varepsilon}f(x-y)\overline{\Phi}'(y)dy+f(x+\varepsilon )\overline{\Phi}(-\varepsilon)-f(x-\varepsilon)\overline{\Phi}(\varepsilon),\quad \varepsilon >0,\;x\in \mathbb{R}^n.
$$
Then
$$
\frac{d}{dx}\overline{G}(x)=\lim_{\varepsilon \rightarrow 0^+}\left(\int_{|y|>\varepsilon}f(x-y)\overline{\Phi}'(y)dy+\overline{J}_\varepsilon(x)\right),\quad \mbox{ for almost all }x\in \mathbb{R},
$$
where
$$
\overline{J}_\varepsilon(x)=f(x-\varepsilon)\overline{\Phi}(\varepsilon)-f(x+\varepsilon )\overline{\Phi}(-\varepsilon),\quad \varepsilon >0,\;x\in \mathbb{R}.
$$
By taking into account \eqref{derivH} and that $\widetilde{H}_\alpha (z)\leq Ce^{-cz^2}$, $z\in \mathbb{R}$, it follows that
$$
\overline{\Phi}'(y)=\frac{(-1)^{\alpha }}{2^{\frac{\alpha +1}{2}}\sqrt{\pi}}\int_0^\infty \widetilde{H}_\alpha \Big(\frac{y}{\sqrt{2t}}\Big)t^{-\frac{3}{2}}dt=\int_0^\infty (\frac{d^\alpha}{dx^\alpha}W_t)(y)t^{\frac{\alpha}{2}-1}dt,\quad y\in \mathbb{R}\setminus\{0\}.
$$

On the other hand we have that
\begin{align*}
\overline{\Phi}(y)&=\frac{(-1)^{\alpha-1}}{2^{\frac{\alpha}{2}}\sqrt{\pi}}\left(\int_0^1\widetilde{H}_{\alpha-1}\Big(\frac{y}{\sqrt{2t}}\Big)\frac{dt}{t}+\int_1^\infty \Big(\widetilde{H}_{\alpha -1}\Big(\frac{y}{\sqrt{2t}}\Big)-\widetilde{H}_{\alpha-1}(0)\Big)\frac{dt}{t}\right)\\
&=\frac{(-1)^{\alpha-1}}{2^{\frac{\alpha }{2}}\sqrt{\pi}}\left(\int_{\frac{y^2}{2}}^\infty \widetilde{H}_{\alpha -1}(\sqrt{s})\frac{ds}{s}+\int_0^{\frac{y^2}{2}}(\widetilde{H}_{\alpha-1}(\sqrt{s})-\widetilde{H}_{\alpha-1}(0))\frac{ds}{s}\right),\quad y\in \mathbb{R}\setminus\{0\}.
\end{align*}
Let $x\in \mathbb{R}$. Since $f\in C_c^\infty (\mathbb{R}^n)$ we get
$$
|f(x+\varepsilon)-f(x-\varepsilon)|\leq C\varepsilon,\quad \varepsilon >0.
$$
If $\alpha$ is odd, then $\Phi$ is even and, for every $\varepsilon \in (0,1)$,
\begin{align*}
|\overline{J}_\varepsilon(x)|&=|(f(x+\varepsilon)-f(x-\varepsilon))\overline{\Phi}(\varepsilon)|\leq C\varepsilon \left(\int_{\frac{\varepsilon ^2}{2}}^\infty |\widetilde{H}_{\alpha-1}(\sqrt{s})|\frac{ds}{s}+\int_0^{\frac{\varepsilon ^2}{2}}|\widetilde{H}_{\alpha -1}(\sqrt{s})-\widetilde{H}_{\alpha-1}(0)|\frac{ds}{s}\right)\\
&\leq C\varepsilon \left(\int_1^\infty e^{-s}ds+\int_{\frac{\varepsilon ^2}{2}}^1 \frac{ds}{s}+\int_0^{\frac{\varepsilon ^2}{2}}\sqrt{s}\frac{ds}{s}\right)\leq C\varepsilon \big(1+|\log\varepsilon|+\varepsilon \big).
\end{align*}
Hence, if $\alpha$ is odd, $\lim_{\varepsilon \rightarrow 0^+}\overline{J}_\varepsilon(x)=0$.

If $\alpha$ is even, then $\overline{\Phi}$ is odd, $H_{\alpha-1}(0)=0$, and
$$
\overline{\Phi }(\varepsilon)=-\frac{1}{2^{\frac{\alpha}{2}}\sqrt{\pi}}\int_0^\infty \widetilde{H}_{\alpha-1}(\sqrt{s})\frac{ds}{s},\quad \varepsilon >0.
$$
We obtain
$$
\lim_{\varepsilon \rightarrow 0^+}\overline{J}_\varepsilon(x)=-\frac{f(x)}{2^{\frac{\alpha}{2}-1}\sqrt{\pi}}\int_0^\infty \widetilde{H}_{\alpha-1}(\sqrt{s})\frac{ds}{s},
$$
provided that $\alpha$ is even.

We conclude that, for certain $c_\alpha \in \mathbb{R}$
$$
\frac{d}{dx}G(x)=\lim_{\varepsilon \rightarrow 0^+}\int_{|x-y|>\varepsilon}f(y)\int_0^\infty(\frac{d^\alpha}{dx^\alpha}W_t)(x-y)t^{\frac{\alpha}{2}-1}dtdy+c_\alpha f(x),\quad \mbox{for almost all }x\in \mathbb{R}.
$$
Then, we get 
$$
\frac{d^\alpha}{dx^\alpha}\mathcal{A}^{-\alpha/2}(f)(x)=\frac{1}{\Gamma (\frac{\alpha}{2})}\lim_{\varepsilon \rightarrow 0^+}\int_{|x-y|>\varepsilon }f(y)\int_0^\infty \partial_x^\alpha T_t^\mathcal{A}(x,y)t^{\frac{\alpha}{2}-1}dtdy+c_\alpha f(x),\quad \mbox{for almost all }x\in \mathbb{R},
$$
where $c_\alpha=0$ when $\alpha$ is odd.

Finally we observe that when $\alpha$ is even
\begin{align*}
\int_0^\infty (\frac{d^\alpha}{dz^\alpha}W_t)(z)t^{\frac{\alpha}{2}-1}dt&=\frac{1}{2^{\frac{\alpha +1}{2}}\sqrt{\pi}}\int_0^\infty e^{-\frac{z^2}{2t}}H_\alpha \Big(\frac{|z|}{\sqrt{2t}}\Big)t^{-\frac{3}{2}}dt=\frac{1}{2^{\frac{\alpha }{2}-1}\sqrt{\pi}}\frac{1}{|z|}\int_0^\infty e^{-s^2}H_\alpha (s)ds=0.
\end{align*}
By taking into account the arguments developed in this proof we can see that
$$
\int_{\mathbb{R}}\left|f(y)\int_0^\infty \partial_x^\alpha T_t^\mathcal{A}(x,y)t^{\frac{\alpha}{2}-1}dt\right|dy=\int_{\mathbb{R}}\left|f(y)\int_0^\infty \partial_x^\alpha \big(T_t^\mathcal{A}(x,y)-W_t(x-y)\big)t^{\frac{\alpha}{2}-1}dt\right|dy<\infty,
$$
for every $x\in \mathbb{R}$. Then, 
$$
\frac{d^\alpha}{dx^\alpha}\mathcal{A}^{-\alpha/2}(f)(x)=\frac{1}{\Gamma (\frac{\alpha}{2})}\int_{\mathbb{R}}f(y)\int_0^\infty \partial_x^\alpha T_t^\mathcal{A}(x,y)t^{\frac{\alpha}{2}-1}dtdy+c_\alpha f(x),\quad \mbox{ for almost all }x\in \mathbb{R}.
$$

The proof of Theorem \ref{Th1.1} is completed.

\subsection{Proof of Theorem \ref{Th1.2}.}
For every $f\in C_c^\infty (\mathbb{R}^n)$ we have that
$$
\mathcal{A}^{-|\alpha|/2}(f)(x)=\sum_{k\in \mathbb{N}^n}\frac{c_k(f)}{(|k|+n)^{\frac{|\alpha|}{2}}}\widetilde{H}_k(x),\quad x\in \mathbb{R}^n,
$$
and according to \eqref{2.1}, \eqref{2.2} and \eqref{2.3} the last series is pointwise absolutely convergent, it defines a smooth function on $\mathbb{R}^n$ and
$$
\partial _x^\alpha\mathcal{A}^{-|\alpha|/2}(f)(x)=\sum_{k\in \mathbb{N}^n}\frac{c_k(f)}{(|k|+n)^{\frac{|\alpha|}{2}}}\partial _x^\alpha\widetilde{H}_k(x)=(-1)^{|\alpha|}\sum_{k\in \mathbb{R}^n}\frac{c_k(f)}{(|k|+n)^{\frac{|\alpha|}{2}}}\widetilde{H}_{k+\alpha}(x),\quad x\in \mathbb{R}^n.
$$
Then, according to Theorem \ref{Th1.1}, for each $f\in C_c^\infty (\mathbb{R}^n)$,
\begin{equation}\label{limC_c}
R_\alpha (f)(x)=\lim_{\varepsilon \rightarrow 0^+}\int_{|x-y|>\varepsilon}R_\alpha (x,y)f(y)dy+c_\alpha f(x),\quad \mbox{for almost all }x\in \mathbb{R}^n.
\end{equation}
Here $c_\alpha =0$ when $\alpha _i$ is odd for some $i=1,...,n$.

To establish our result it is sufficient to show that for every $f \in L^p(\mathbb{R}^n,\gamma_{-1})$ the limit
$$
\lim_{\varepsilon \rightarrow 0^+}\int_{|x-y|>\varepsilon}R_\alpha (x,y)f(y)dy
$$
exists for almost $x\in \mathbb{R}^n$ and the operator $L_\alpha $ defined by
\begin{equation}\label{Lalpha}
L_\alpha (f)(x)=\lim_{\varepsilon \rightarrow 0^+}\int_{|x-y|>\varepsilon}R_\alpha (x,y)f(y)dy+c_\alpha f(x)
\end{equation}
is bounded from $L^p(\mathbb{R}^n,\gamma_{-1})$ into itself. Thus, $L_\alpha$ is the unique extension of $R_\alpha$ from $L^2(\mathbb{R}^n,\gamma_{-1})\cap L^p(\mathbb{R}^n,\gamma_{-1})$ to $L^p(\mathbb{R}^n,\gamma_{-1})$ as a bounded operator from $L^p(\mathbb{R}^n,\gamma_{-1})$ into itself.

For every $\beta >0$ we define the set
$$
N_\beta =\Big\{(x,y)\in \mathbb{R}^n\times\mathbb{R}^n:|x-y|\leq \beta n\min\big\{1,\frac{1}{|x|}\big\}\Big\}.
$$

Observe that $N_1=N$ and that if $\beta >0$, $a \in (0,1)$ and $(x,y)\in N_\beta^c$, then
$$
|a x-a y|\geq a \beta n\min\Big\{1,\frac{1}{|x|}\Big\}\geq a^2 \beta n\min\Big\{1,\frac{1}{|ax|}\Big\},
$$
that is, $(a x,a y)\in N_{a^2 \beta}^c$. In particular we have that if $a\in (0,1)$ then $(a x,a y)\in N^c$ provided that $(x,y)\in N_{1/a^2}^c$.

Let $\beta >0$. We consider the operators $R_{\alpha,{\rm loc}}$ and $R_{\alpha,{\rm glob}}$ defined on $C_c^\infty(\mathbb{R}^n)$ by
\begin{equation}\label{RlocRglob}
R_{\alpha,{\rm loc}}(f)(x)=R_\alpha (f\chi_{N_\beta}(x,\cdot))(x),\quad R_{\alpha,{\rm glob}}(f)(x)=R_\alpha (f\chi_{N_\beta ^c}(x,\cdot))(x),\quad x\in \mathbb{R}^n.
\end{equation}

We recall that
$$
R_\alpha (x,y)=\frac{(-1)^{|\alpha|}}{\pi ^{\frac{n}{2}}\Gamma (\frac{|\alpha |}{2})}\int_0^\infty \frac{e^{-nt}}{(1-e^{-2t})^{\frac{n+|\alpha|}{2}}}\widetilde{H}_\ell \Big(\frac{x-e^{-t}y}{\sqrt{1-e^{-2t}}}\Big)t^{\frac{|\alpha|}{2}-1}dt,\quad x,y\in \mathbb{R}^n,\;x\not=y.
$$
Then, 
\begin{equation}\label{acotRalphab}
|R_\alpha (x,y)|\leq C\int_0^\infty \frac{e^{-nt}e^{-\eta \frac{|x-e^{-t}y|^2}{1-e^{-2t}}}}{(1-e^{-2t})^{\frac{n+|\alpha|}{2}}}t^{\frac{|\alpha|}{2}-1}dt,\quad x,y\in \mathbb{R}^n,\;x\not=y,
\end{equation}
for every $\eta \in (0,1)$. From now on we consider $1/p<\eta <1$ and $\beta =1/\eta$. Since $|x-ry|^2=|y-rx|^2+(1-r^2)(|x|^2-|y|^2)$, $x,y\in \mathbb{R}^n$, $r\in \mathbb{R}$, by making the change of variables $s=1-e^{-2t}$ in the last integral we get
$$
|R_\alpha (x,y)|\leq Ce^{\eta (|y|^2-|x|^2)}\int_0^1\frac{e^{-\eta \frac{|y-x\sqrt{1-s}|^2}{s}}}{s^{\frac{n+|\alpha|}{2}}}(1-s)^{\frac{n}{2}-1}(-\log (1-s))^{\frac{|\alpha|}{2}-1}ds,\quad x,y\in \mathbb{R}^n,\;x\not=y.
$$

Assume first that $(x,y)\in N_\beta^c$. Then, $(\sqrt{\eta}x,\sqrt{\eta}y)\in N^c$. By proceeding as in the proof of \eqref{Mbeta} it follows that, when $\langle x,y\rangle\leq 0$,
$$
|R_\alpha (x,y)|\leq Ce^{-\eta |x|^2}\int_0^\infty \frac{e^{-\frac{r}{2}}(1+r)^{\frac{n+|\alpha|}{2}-\frac{1}{4}}}{r^{\frac{3}{4}}}dr\leq Ce^{-\eta|x|^2}.
$$

Suppose now that $\langle x,y\rangle>0$. Again, as in the estimation in \eqref{Mbeta}, since $(\sqrt{\eta}x,\sqrt{\eta}y)\in N^c$ we obtain 
$$
\int_{\frac{1}{2}}^1\frac{e^{-\eta \frac{|y-x\sqrt{1-s}|^2}{s}}}{s^{\frac{n+|\alpha|}{2}}}(1-s)^{\frac{n}{2}-1}(-\log (1-s))^{\frac{|\alpha|}{2}-1}ds\leq C\Big(\frac{|x+y|}{|x-y|}\Big)^{\frac{n}{2}}\exp\Big(\frac{\eta }{2}\big(y|^2-|x|^2-|x+y||x-y|\big)\Big).
$$

On the other hand, proceeding as in \cite[p. 862]{PeSo} and considering the notation in \eqref{uPeSo} and \cite[Lemma 2.3]{PeSo} we have
\begin{align*}
\int_0^{\frac{1}{2}}\frac{e^{-\eta \frac{|y-x\sqrt{1-s}|^2}{s}}}{s^{\frac{n+|\alpha|}{2}}}(1-s)^{\frac{n}{2}-1}(-\log (1-s))^{\frac{|\alpha|}{2}-1}ds&
\leq C\sup_{s\in (0,1)}\Big(\frac{e^{-\eta u(s)}}{s^{\frac{n}{2}}}\Big)^{1-\frac{1}{n}}\int_0^{\frac{1}{2}}\frac{e^{-\eta \frac{u(s)}{n}}}{\sqrt{s}}\frac{(-\log (1-s))^{\frac{|\alpha|}{2}-1}}{s^{\frac{|\alpha|}{2}}}ds\\
&\hspace{-5cm}\leq C\Big(\frac{e^{-\eta u_0}}{s_0^{n/2}}\Big)^{1-\frac{1}{n}}\int_0^{\frac{1}{2}}\frac{e^{-\frac{u(s)}{n}}}{s^{\frac{3}{2}}}ds\leq  C\Big(\frac{e^{-\eta u_0}}{s_0^{n/2}}\Big)^{1-\frac{1}{n}}\int_0^1\frac{e^{-\frac{u(s)}{n}}}{s^{\frac{3}{2}}\sqrt{1-s}}ds\leq C\Big(\frac{e^{-\eta u_0}}{s_0^{n/2}}\Big)^{1-\frac{1}{n}}\frac{e^{-\eta  \frac{u_0}{n}}}{\sqrt{s_0}}\\
&\hspace{-5cm}\leq C\frac{e^{-\eta u_0}}{s_0^{n/2}}\leq C\Big(\frac{|x+y|}{|x-y|}\Big)^{\frac{n}{2}}\exp\Big(\frac{\eta }{2}\big(y|^2-|x|^2-|x+y||x-y|\big)\Big).
\end{align*}

From the above estimates we conclude that, when $(x,y)\in N_\beta ^c$, 
\begin{equation}\label{acotRalpha}
|R_\alpha (x,y)|\leq C\left\{
\begin{array}{ll}
e^{-\eta |x|^2},&\langle x,y\rangle \leq 0,\\
\displaystyle \Big(\frac{|x+y|}{|x-y|}\Big)^{\frac{n}{2}}\exp\Big(\eta\Big(\frac{|y|^2-|x|^2}{2}-\frac{|x-y||x+y|}{2}\Big)\Big),& \langle x,y\rangle >0.
\end{array}
\right.
\end{equation}

On the other hand, we consider the kernel
\begin{equation}\label{Rclassical}
\mathbb{R}_\alpha (x,y)=\frac{1}{\Gamma(\frac{|\alpha|}{2})}\int_0^\infty\partial _x^\alpha W_t(x-y)t^{\frac{|\alpha|}{2}-1}dt,\quad x,y\in \mathbb{R}^n,\;x\not=y.
\end{equation}

Let us show that
\begin{equation}\label{acotdif}
|R_\alpha (x,y)-\mathbb{R}_\alpha (x,y)|\leq C\frac{\sqrt{1+|x|}}{|x-y|^{n- \frac{1}{2}}},\quad (x,y)\in N_\beta , \;x\not=y .
\end{equation}

For every $x\in \mathbb{R}^n$, we have that
\begin{align*}
R_\alpha (x,y)-\mathbb{R}_\alpha (x,y)&=\frac{1}{\Gamma(\frac{|\alpha|}{2})}\left(\int_0^{m(x)}+\int_{m(x)}^\infty\right) \partial _x^\alpha (T_t^\mathcal{A}
(x,y)-W_t(x-y))t^{\frac{|\alpha|}{2}-1}dt\\
&=I(x,y)+J(x,y),\quad x,y\in \mathbb{R}^n,\;x\not=y.
\end{align*}
The same proof of \eqref{acotI} allows us to obtain that
$$
|I(x,y)|\leq C\frac{\sqrt{1+|x|}}{|x-y|^{n-\frac{1}{2}}},\quad (x,y)\in N_\beta,\;x\not=y.
$$
Also, we get
\begin{align*}
|J(x,y)|&\leq C\left(\int_{m(x)}^\infty \frac{e^{-nt}}{(1-e^{-2t})^{\frac{n+|\alpha|}{2}}}\Big|\widetilde{H}_\alpha \Big(\frac{x-ye^{-t}}{\sqrt{1-e^{-2t}}}\Big)\Big|t^{\frac{|\alpha|}{2}-1}dt+\int_{m(x)}^\infty \Big|\widetilde{H}_\alpha \Big(\frac{x-y}{\sqrt{2t}}\Big)\Big|\frac{dt}{t^{\frac{n}{2}+1}}\right)\\
&\leq C\left(\int_{m(x)}^\infty e^{-nt}\frac{e^{-c\frac{|x-e^{-t}y|^2}{1-e^{-2t}}}}{(1-e^{-2t})^{\frac{n+|\alpha|}{2}}}t^{\frac{|\alpha|}{2}-1}dt+\int_{m(x)}^\infty \frac{e^{-c\frac{|x-y|^2}{2t}}}{t^{\frac{n}{2}+1}}dt\right)\leq C\int_{m(x)}^\infty \frac{dt}{t^{\frac{n}{2}+1}}\\
&= \frac{C}{m(x)^\frac{n}{2}}\leq C(1+|x|)^n\leq C\frac{\sqrt{1+|x|}}{|x-y|^{n- \frac{1}{2}}},\quad (x,y)\in N_\beta ,\; x\not=y,
\end{align*}
and thus, \eqref{acotdif} is established.

From \eqref{acotRalpha} we obtain that $R_{\alpha,{\rm glob}}$, can be extended to $L^p(\mathbb{R}^n,\gamma_{-1})$ as a bounded operator from $L^p(\mathbb{R}^n,\gamma_{-1})$ into itself, and the extension is given by \eqref{RlocRglob}. Indeed, when $(x,y)\in N_\beta ^c$, we have that $|x-y||x+y|\geq C$. By taking into account also that $||y|^2-|x|^2|\leq |x+y||x-y|$ we get
\begin{align*}
\int_{\mathbb{R}^n}|R_\alpha (x,y)|e^{\frac{|x|^2-|y|^2}{p}}\chi _{N_\beta ^c}(x,y)dy&\\
&\hspace{-4cm}\leq C\left(\int_{\mathbb{R}^n}e^{-(\eta -\frac{1}{p})|x|^2-\frac{|y|^2}{p}}dy+\int_{\mathbb{R}^n}|x+y|^n\exp\Big(\Big(\frac{\eta}{2}-\frac{1}{p}\Big)(|y|^2-|x|^2)-\frac{\eta }{2}|x-y||x+y|\Big)dy\right)\\
&\hspace{-4cm}\leq C\left(e^{-(\eta -\frac{1}{p})|x|^2}\int_{\mathbb{R}^n}e^{-\frac{|y|^2}{p}}dy+\int_{\mathbb{R}^n}|x+y|^n\exp\Big(-|x-y||x+y|\big(\frac{\eta}{2}-\big|\frac{1}{p}-\frac{\eta}{2}\big|\big)\Big)dy\right),\quad x\in \mathbb{R}^n.
\end{align*}
Since $\eta >1/p$ it follows that
$$
\sup_{x\in \mathbb{R}^n}\int_{\mathbb{R}^n}|R_\alpha (x,y)|e^{\frac{|x|^2-|y|^2}{p}}\chi_{N_\beta ^c}(x,y)dy<\infty,
$$
and in a similar way 
$$
\sup_{y\in \mathbb{R}^n}\int_{\mathbb{R}^n}|R_\alpha (x,y)|e^{\frac{|x|^2-|y|^2}{p}}\chi_{N_\beta ^c}(x,y)dx<\infty,
$$
from which we deduce that $R_{\alpha,{\rm glob}}$ is bounded on $L^p(\mathbb{R}^n,\gamma_{-1})$.

On the other hand, by using \eqref{acotdif} and that $\sqrt{m(x)}\sim \frac{1}{1+|x|}$, $x\in \mathbb{R}^n$, we have that
\begin{align*}
\int_{\mathbb{R}^n}|R_\alpha (x,y)-\mathbb{R}_\alpha (x,y)|\chi_{N_\beta }(x,y)dy&\leq C\sqrt{1+|x|}\int_{\mathbb{R}^n}\frac{1}{|x-y|^{n-\frac{1}{2}}}\chi_{N_\beta }(x,y)dy\\
&\hspace{-3cm}=C\sqrt{1+|x|}\int_0^{\beta n\sqrt{m(x)}}r^{-\frac{1}{2}}dr=C\sqrt{(1+|x|)\sqrt{m(x)}}\leq C,\quad x\in \mathbb{R}^n.
\end{align*}
Also, since $m(x)\sim m(y)$, $(x,y)\in N_\beta$, we get 
$$
\sup_{y\in \mathbb{R}^n}\int_{\mathbb{R}^n}|R_\alpha (x,y)-\mathbb{R}_\alpha (x,y)|\chi_{N_\beta }(x,y)dx<\infty.
$$
Hence the operator $S_\alpha$ defined by
$$
S_\alpha(f)(x)=\int_{\mathbb{R}^n}(R_\alpha (x,y)-\mathbb{R}_\alpha (x,y))\chi_{N_\beta}(x,y)f(y)dy,\quad x\in \mathbb{R}^n,
$$
is a bounded operator from $L^p(\mathbb{R}^n,dx)$ into itself. Since $S_\alpha$ is a local operator, by \cite[Proposition 3.2.5]{Sa} $S_\alpha$ is bounded from $L^p(\mathbb{R}^n, \gamma_{-1})$ into itself.

We now observe that the kernel $\mathbb{R}_\alpha$ is a standard Calder\'on-Zygmund kernel. Indeed, we get
$$
|\mathbb{R}_\alpha (x,y)|\leq C\int_0^\infty \Big|H_\alpha \big(\frac{x-y}{\sqrt{2t}}\big)\Big|\frac{e^{-\frac{|x-y|^2}{2t}}}{t^{\frac{n}{2}+1}}dt\leq C\int_0^\infty \frac{e^{-c\frac{|x-y|^2}{t}}}{t^{\frac{n}{2}+1}}dt\leq \frac{C}{|x-y |^n},\quad x,y\in \mathbb{R}^n,\;x\not =y.
$$
Let $i=1,...,n$ and denote $\alpha^i=(\alpha _1,...,\alpha_i+1,...,\alpha _n)$. We have that
$$
\partial_{x_i}\mathbb{R}_\alpha (x,y)=\frac{(-1)^{|\alpha|+1}}{(2\pi )^{\frac{n}{2}}2^{\frac{|\alpha|+1}{2}}\Gamma (\frac{|\alpha|}{2})}\int_0^\infty H_{\alpha ^i}\Big(\frac{x-y}{\sqrt{2t}}\Big)\frac{e^{-\frac{|x-y|^2}{2t}}}{t^{\frac{n+3}{2}}}dt,\quad x,y\in \mathbb{R}^n,\;x\not=y.
$$
Then,
$$
\Big|\partial _{x_i}\mathbb{R}_\alpha (x,y)\Big|\leq \frac{C}{|x-y |^{n+1}},\quad x,y\in \mathbb{R}^n,\;x\not =y.
$$

The Euclidean $\alpha$-order Riesz transform $\mathbb{R}_\alpha$ is bounded from $L^q(\mathbb{R}^n,dx)$ into itself, for every $1<q<\infty$. According to \cite[Proposition 3.2.5]{Sa} the operator $\mathbb{R}_{\alpha ,{\rm loc}}$ defined by
$$
\mathbb{R}_{\alpha ,{\rm loc}}(f)(x)=\mathbb{R}_{\alpha }(f\chi _{N_\beta }(x,\cdot))(x)
$$
is bounded from $L^q(\mathbb{R}^n,\gamma_{-1})$ into itself, for every $1<q<\infty$.

We can write $R_{\alpha,{\rm loc}}=S_\alpha (f)+\mathbb{R}_{\alpha ,{\rm loc}}(f)$ on $C_c^\infty(\mathbb{R}^n)$. Then, $R_{\alpha,{\rm loc}}$  can be extended from $C_c^\infty (\mathbb{R}^n)$ to $L^p(\mathbb{R}^n,\gamma_{-1})$ as a bounded operator from $L^p(\mathbb{R}^n,\gamma_{-1})$ into itself.

Since $R_\alpha=R_{\alpha,{\rm loc}}+R_{\alpha,{\rm glob}}$ we conclude that $R_\alpha$ can be extended from $C_c^\infty (\mathbb{R}^n)$ to $L^p(\mathbb{R}^n,\gamma_{-1})$ as a bounded operator from $L^p(\mathbb{R}^n,\gamma_{-1})$ into itself. 

Let us consider the maximal operator 
$$
R_{\alpha,*}(f)(x)=\sup_{\varepsilon >0}\left|\int_{|x-y|>\varepsilon }R_\alpha (x,y)f(y)dy\right|,\quad f\in L^p(\mathbb{R}^n,\gamma_{-1}),\;1<p<\infty.
$$
Let $f\in L^p(\mathbb{R}^n,\gamma_{-1})$, $1<p<\infty$. For every $\varepsilon >0$, by using the above estimates we can write 
\begin{align*}
\left|\int_{|x-y|>\varepsilon }R_\alpha (x,y)f(y)dy\right|& \leq \int_{|x-y|>\varepsilon}|R_\alpha (x,y)-\mathbb{R}_\alpha (x,y)|\chi _{N_\beta }(x,y)|f(y)|dy\\
&\hspace{-3cm} +\int_{|x-y|>\varepsilon}|R_\alpha (x,y)|\chi _{N_\beta ^c}(x,y)|f(y)|dy+\int_{|x-y|>\varepsilon}|\mathbb{R}_\alpha (x,y)|\chi _{N_\beta }(x,y)|f(y)|dy<\infty ,\quad x\in \mathbb{R}^n.
\end{align*}

We also have that
\begin{align*}
R_{\alpha ,*}(f)(x)&\leq C\left(\int_{\mathbb{R}^n}\frac{\sqrt{1+|x|}}{|x-y|^{n-\frac{1}{2}}}\chi_{N_\beta }(x,y)|f(y)|dy+\int_{\mathbb{R}^n}|R_\alpha (x,y)|\chi _{N_\beta ^c}(x,y)|f(y)|dy\right.\\
&\left.\quad +\sup_{\varepsilon >0}\left|\int_{|x-y|>\varepsilon}\mathbb{R}_\alpha (x,y)\chi _{N_\beta }(x,y)f(y)dy\right|\right),\quad x\in \mathbb{R}^n.
\end{align*}

Since the maximal operator 
$$
\mathbb{R}_{\alpha ,*}(f)(x)=\sup_{\varepsilon >0}\left|\int_{|x-y|>\varepsilon}\mathbb{R}_\alpha (x,y)f(y)dy\right|
$$
is bounded from $L^p(\mathbb{R}^n,dx)$ into itself, by using a vector valued version of \cite[Proposition 3.2.5]{Sa} (see \cite[Proposition 2.3]{HTV}), we deduce that the local maximal operator
$$
\mathbb{R}_{\alpha,{\rm loc }, *}(f)(x)=\sup_{\varepsilon >0}\left|\int_{|x-y|>\varepsilon}\mathbb{R}_\alpha (x,y)\chi _{N_\beta}(x,y)f(y)dy\right|
$$
is bounded from $L^p(\mathbb{R}^n,\gamma_{-1})$ into itself.

By using the same arguments as above we conclude that $R_{\alpha ,*}$ is bounded from $L^p(\mathbb{R}^n,\gamma_{-1})$ into itself.

From \eqref{limC_c} and since $C_c^\infty ( \mathbb{R}^n)$ is dense in $L^p(\mathbb{R}^n,\gamma_{-1})$ and $\mathbb{R}_{\alpha ,*}$ is bounded in $L^p(\mathbb{R}^n,\gamma_{-1})$, by using a standard procedure we can conclude  
that the limit
$$
\lim_{\varepsilon \rightarrow 0^+}\int_{|x-y|>\varepsilon}R_\alpha (x,y)f(y)dy
$$
exists for almost $x\in \mathbb{R}^n$ and $L_\alpha$ (defined in \eqref{Lalpha}) is a bounded operator from $L^p(\mathbb{R}^n,\gamma_{-1})$ into itself. 

\begin{rem}The $L^p(\mathbb{R}^n,\gamma_{-1})$-boundedness of the local part $R_{\alpha,{\rm loc}}$ of $R_\alpha$ can be proved also by using Calder\'on-Zygmund theory. We have preferred to do it by comparing $R_{\alpha,{\rm loc}}$ with the classical local Riesz transform $\mathbb{R}_{\alpha,{\rm loc}}$ because in this way we can know how the singularity of $R_{\alpha,{\rm loc}}$ is. Furthermore, these comparative results will be useful in the proof of Theorem \ref{Th1.4} and \ref{Th1.5}.
\end{rem}

\section{Riesz transform associated with the operator $\bar{\mathcal{A}}$}\label{barA}

We define $L_0^2(\mathbb{R}^n,\gamma _{-1})$ as the space that consists of all those $f\in L ^2(\mathbb{R}^n,\gamma _{-1})$ such that $c_0(f)=\int_{\mathbb{R}^n}f(x)dx=0$. Let $\beta >0$. For every $f\in L_0^2(\mathbb{R}^n,\gamma _{-1})$, $\bar{\mathcal{A}}^{-\beta}f$ is defined by
$$
\bar{\mathcal{A}}^{-\beta }(f)=\sum_{k\in \mathbb{N}^n\setminus\{0\}}\frac{c_k(f)}{|k|^\beta}\widetilde{H}_k.
$$
We have that
$$
\|\bar{\mathcal{A}}^{-\beta }(f)\|_{L^2(\mathbb{R}^n,\gamma_{-1})}^2=\sum_{k\in \mathbb{N}^n\setminus\{0\}}\frac{|c_k(f)|^2\|\widetilde{H}_k\|_{L^2(\mathbb{R}^n,\gamma_{-1})}^2}{|k|^{2\beta}}\leq \|f\|^2_{L^2(\mathbb{R}^n,\gamma_{-1})},\quad f \in L_0^2(\mathbb{R}^n,\gamma_{-1}).
$$
We introduce the operator $\overline{S}_\beta $ defined by
$$
\overline{S}_\beta (f)=\frac{1}{\Gamma (\beta)}\int_0^\infty(T_t^{\bar{\mathcal{A}}}(f)-c_0(f)e^{|\cdot|^2}) t^{\beta -1}dt,\quad f\in L^2(\mathbb{R}^n,\gamma_{-1}).
$$

Let $f\in L^2(\mathbb{R}^n,\gamma_{-1})$. We have that 
$$
T_t^{\bar{\mathcal{A}}}(f)-c_0(f)e^{-|\cdot|^2}=\sum_{k\in \mathbb{N}^n}e^{-|k|t}c_k(f)\widetilde{H}_k-e^{|\cdot|^2}c_0(f)=\sum_{k\in \mathbb{N}^n\setminus\{0\}}e^{-|k|t}c_k(f)\widetilde{H}_k,\quad t>0.
$$
Then,
$$
\|T_t^{\bar{\mathcal{A}}}(f)-c_0(f)e^{-|\cdot|^2}\|_{L^2(\mathbb{R}^n,\gamma_{-1})}\leq e^{-t}\|f\|_{L^2(\mathbb{R}^n,\gamma_{-1})},\quad t>0.
$$
Hence, the integral defining $\overline{S}_\beta (f)$ converges in the $L^2(\mathbb{R}^n,\gamma_{-1})$-Bochner sense.

Let $f\in C_c^\infty(\mathbb{R}^n)$. According to \eqref{2.1}, \eqref{2.2} and \eqref{2.3} we get
$$
\sum_{k\in \mathbb{N}^n\setminus\{0\}}e^{-|k|t}|c_k(f)||\widetilde{H}_k(x)|\leq Ce^{-t}e^{-\frac{|x|^2}{2}}\sum_{k\in \mathbb{N}^n\setminus\{0\}}\frac{1}{|k|^2}\leq Ce^{-t}e^{-\frac{|x|^2}{2}},\quad t>0,\;x\in \mathbb{R}^n.
$$
It follows that that the series that defines $T_t^{\bar{\mathcal{A}}}(f)(x)-c_0(f)e^{-|x|^2}$ converges pointwisely and absolutely and
$$
\overline{S}_\beta(f)(x)=\frac{1}{\Gamma (\beta)}\sum_{k\in \mathbb{N}^n\setminus\{0\}}c_k(f)\widetilde{H}_k(x)\int_0^\infty e^{-|k|t}t^{\beta -1}dt=\sum_{k\in \mathbb{N}^n\setminus\{0\}}\frac{c_k(f)}{|k|^\beta}\widetilde{H}_k(x)=\bar{\mathcal{A}}^{-\beta}(f_0)(x),\quad x\in \mathbb{R}^n,
$$
where $f_0(x)=f(x)-c_0(f)e^{-|x|^2}$, $x\in \mathbb{R}^n$.

On the other hand, since $\supp f$ is compact we have that 
\begin{align*}
\int_{m(x)}^\infty\int_{\mathbb{R}^n}\big|T_t^{\bar{\mathcal{A}}}(x,y)-e^{-|x|^2}\big||f(y)|dy t^{\beta -1}dt&\\
&\hspace{-4cm} \leq C\int_{m(x)}^\infty \left(\int_{\mathbb{R}^n}e^{-\frac{|x-e^{-t}y|^2}{1-e^{-2t}}}\left|(1-e^{-2t})^{-\frac{n}{2}}-1\right||f(y)|dy\right.\\
&\hspace{-4cm}\quad +\int_{\mathbb{R}^n}\left|e^{-\frac{|x-e^{-t}y|^2}{1-e^{-2t}}}-e^{-|x-e^{-t}y|^2}\right||f(y)|dy\left. +\int_{\mathbb{R}^n}\left|e^{-|x-e^{-t}y|^2}-\-e^{-|x|^2}\right||f(y)|dy\right)t^{\beta -1} dt\\
&\hspace{-4cm}\leq C\int_0^\infty e^{-t}t^{\beta -1}dt,\quad x\in \mathbb{R}^n,
\end{align*}
and, taking $0<\varepsilon <\min\{2\beta ,n\}$,
\begin{align*}
\int_0^{m(x)}t^{\beta -1}\int_{\mathbb{R}^n}\big|T_t^{\bar{\mathcal{A}}}(x,y)-e^{-|x|^2}\big||f(y)|dydt &\leq C\left(\int_0^{m(x)}t^{\beta -1}\int_{\mathbb{R}^n}\Big(e^{c|x|}\frac{e^{-c\frac{|x-y|^2}{t}}}{t^{\frac{n}{2}}}+e^{-|x|^2})|f(y)|dydt\right)\\
&\hspace{-2cm}\leq C\left(e^{c|x|}\int_{\mathbb{R}^n}\frac{|f(y)|}{|x-y|^{n-\varepsilon}}dy+m(x)^\beta e^{-|x|^2}\int_{\mathbb{R}^n}|f(y)|dy\right)\\
&\hspace{-2cm}\leq C\Big(e^{c|x|}+m(x)^\beta e^{-|x|^2}\Big),\quad x\in\mathbb{R}^n.
\end{align*}
Then, we obtain 
$$
\overline{S}_\beta (f)(x)=\int_{\mathbb{R}^n}\overline{K}_\beta (x,y)f(y)dy,\quad x\in \mathbb{R}^n,
$$
where 
$$
\overline{K}_\beta (x,y)=\frac{1}{\Gamma (\beta)}\int_0^\infty (T_t^{\bar{A}}(x,y)-e^{-|x|^2})t^{\beta -1}dt,\quad x,y\in \mathbb{R}^n,
$$
and 
$$
T_t^{\bar{\mathcal{A}}}(x,y)=e^{nt}T_t^\mathcal{A}(x,y)=\frac{e^{|y|^2-|x|^2}}{\pi ^{\frac{n}{2}}(1-e^{-2t})^{\frac{n}{2}}}\exp\left(-\frac{|y-e^{-t}x|^2}{1-e^{-2t}}\right), \quad x,y\in \mathbb{R}^n, \;t>0.
$$

By denoting $\Pi _0$ the projection from $L^2(\mathbb{R}^n,\gamma_{-1})$ to $L_0^2(\mathbb{R}^n,\gamma_{-1})$ we have proved that, for every $f\in C_c^\infty (\mathbb{R}^n)$,
$$
\bar{\mathcal{A}}^{-\beta}\Pi _0(f)(x)=\int_{\mathbb{R}^n}\overline{K}_\beta (x,y)f(y)dy,\quad x\in \mathbb{R}^n.
$$

Let $f\in C_c^\infty (\mathbb{R}^n)$. Next we show that
\begin{equation}\label{deltaA}
\delta_x^\alpha \bar{\mathcal{A}}^{-|\alpha|/2}\Pi _0(f)(x)=\lim_{\varepsilon \rightarrow 0^+}\frac{1}{\Gamma (\frac{|\alpha|}{2})}\int_{|x-y|>\varepsilon }f(y)\int_0^\infty \delta_x^\alpha T_t^{\bar{A}}(x,y)t^{\frac{|\alpha|}{2}-1}dtdy+c_\alpha f(x),
\end{equation}
for almost all $x\in \mathbb{R}^n$. Here $c_\alpha\in \mathbb{R}$ and, when $\alpha _i$ is odd for some $i=1,...,n$, $c_\alpha =0$.

For every $\ell \in \mathbb{N}^n$, 
\begin{align*}
\delta_x^\ell T_t^{\bar{\mathcal{A}}}(x,y)&=\frac{(-1)^{|\ell|}e^{-|x|^2}}{2^{|\ell|}}\partial ^\ell _x\big(e^{|x|^2}T_t^{\bar{\mathcal{A}}}(x,y)\big)\\
&=\frac{(-1)^{|\ell|}e^{|y|^2-|x|^2}}{2^{|\ell|}\pi ^{\frac{n}{2}}(1-e^{-2t})^{\frac{n+|\ell|}{2}}}e^{-|\ell|t}\widetilde{H}_\ell \left(\frac{y-e^{-t}x}{\sqrt{1-e^{-2t}}}\right),\quad x,y\in \mathbb{R}^n, \;t>0,
\end{align*}
Then, for each $\ell \in \mathbb{N}^n$,
$$
|\delta_x^\ell T_t^{\bar{A}}(x,y)|\leq Ce^{|y|^2-|x|^2}\frac{e^{-|\ell|t}e^{-c\frac{|y-e^{-t}x|^2}{1-e^{-2t}}}}{(1-e^{-2t})^{\frac{n+|\ell|}{2}}},\quad x,y\in \mathbb{R}^n,\;t>0.
$$

Since $\delta_x^\ell (e^{-|x|^2})=0$, $x\in \mathbb{R}^n$, when $\ell \in \mathbb{N}^n\setminus\{0\}$, by proceeding as in the proof of \eqref{derivA} we get that for every $\ell \in \mathbb{N}^n\setminus\{0\}$ and $k\in \mathbb{N}$ being $|\ell|<k$, 
$$
\delta_x^\ell \bar{\mathcal{A}}^{-k/2}\Pi _0(f)(x)=\frac{1}{\Gamma (\frac{k}{2})}\int_{\mathbb{R}^n}f(y)\int_0^\infty \delta_x^\ell T_t^{\bar{A}}(x,y)t^{\frac{k}{2}-1}dtdy,\quad\mbox{ for almost all }x\in \mathbb{R}^n.
$$

Without loss of generality we can assume that $\alpha _1\geq 1$ and consider $\ell =(\alpha _1-1,\alpha_2,...,\alpha_n)$. When $\ell \in \mathbb{N}^n\setminus\{0\}$, we can proceed as in the proof of Theorem \ref{Th1.1}. For $n>1$ we write
\begin{align}\label{decomp100n}
\delta_x^\alpha \bar{\mathcal{A}}^{-|\alpha|/2}\Pi _0(f)(x)&=\frac{1}{\Gamma (\frac{|\alpha|}{2})}\delta_{x_1}\left(\int_{\mathbb{R}^n}f(y)\int_0^\infty \delta_x^\ell \big(T_t^{\bar{\mathcal{A}}}(x,y)-e^{|y|^2-|x|^2}W_t(y-x)\big)t^{\frac{|\alpha|}{2}-1}dtdy\right.\nonumber\\
&\quad +\left. \int_{\mathbb{R}^n}f(y)\int_0^\infty \delta_x^\ell (e^{|y|^2-|x|^2}W_t(y-x))t^{\frac{|\alpha|}{2}-1}dtdy\right)=\frac{1}{\Gamma (\frac{|\alpha|}{2})}\delta_{x_1}\big(F(x)+G(x)\big),\quad x\in \mathbb{R}^n.
\end{align}
We observe that, for every $r\in \mathbb{N}^n$,
$$
\delta_x^r (e^{|y|^2-|x|^2}W_t(y-x))=\frac{(-1)^{|r|}}{2^{|r|}}e^{|y|^2-|x|^2}\partial_x^r(W_t(y-x))=\frac{(-1)^{|r|}}{2^{|r|}\pi ^{\frac{n}{2}}}\frac{e^{|y|^2-|x|^2}}{(2t)^{\frac{n+|r|}{2}}}\widetilde{H}_r \Big(\frac{y-x}{\sqrt{2t}}\Big),\quad x,y\in \mathbb{R}^n,\;t>0.
$$
By considering the decomposition
\begin{align}\label{decomp}
\delta_x^\alpha \big(T_t^{\bar{\mathcal{A}}}(x,y)-e^{|y|^2-|x|^2}W_t(y-x)\big)&=\frac{(-1)^{|\alpha|}}{2^{|\alpha|}\pi ^{\frac{n}{2}}}e^{|y|^2-|x|^2}\left(\frac{e^{-|\alpha|t}}{(1-e^{-2t})^{\frac{n+|\alpha|}{2}}}\widetilde{H}_\ell \left(\frac{y-e^{-t}x}{\sqrt{1-e^{-2t}}}\right)-\frac{1}{(2t)^{\frac{n+|\alpha|}{2}}}\widetilde{H}_\alpha \Big(\frac{y-x}{\sqrt{2t}}\Big)\right)\nonumber\\
&\hspace{-4cm}=\frac{(-1)^{|\alpha|}}{2^{|\alpha|}\pi ^{\frac{n}{2}}}e^{|y|^2-|x|^2}\left\{\frac{e^{-|\alpha|t}-1}{(1-e^{-2t})^{\frac{n+|\alpha|}{2}}}\widetilde{H}_\alpha\Big(\frac{y-e^{-t}x}{\sqrt{1-e^{-2t}}}\Big)+\Big(\frac{1}{(1-e^{-2t})^{\frac{n+|\alpha|}{2}}}-\frac{1}{(2t)^{\frac{n+|\alpha|}{2}}}\Big)\widetilde{H}_\alpha\left(\frac{y-e^{-t}x}{\sqrt{1-e^{-2t}}}\right)\right.\\
&\hspace{-4cm}\quad \left.+\frac{1}{(2t)^{\frac{n+|\alpha|}{2}}}\left(\widetilde{H}_\alpha\left(\frac{y-e^{-t}x}{\sqrt{1-e^{-2t}}}\right)-\widetilde{H}_\alpha \left(\frac{y-x}{\sqrt{2t}}\right)\right)\right\},\quad x,y\in \mathbb{R}^n.\nonumber
\end{align}
we can argue as in the proof of \eqref{derivF} to obtain that
$$
\delta_{x_1}F(x)=\int_{\mathbb{R}^n}f(y)\int_0^\infty \delta_x^\alpha \big(T_t^{\bar{\mathcal{A}}}(x,y)-e^{|y|^2-|x|^2}W_t(y-x)\big)t^{\frac{|\alpha|}{2}-1}dtdy, \quad \mbox{ for almost all }x\in \mathbb{R}^n.
$$

On the other hand, to deal with $\delta_{x_1}G(x)$ we consider $g(x)=f(x)e^{|x|^2}$ and $\Psi (x)=\int_0^\infty \partial_x^\ell (W_t)(-x)t^{\frac{|\alpha|}{2}-1}dt$, $x\in \mathbb{R}^n$, and write
$$
G(x)=\frac{e^{-|x|^2}}{2^{|\ell|}}\int_{\mathbb{R}^n}g(x-y)\Psi (y)dy,\quad x\in \mathbb{R}^n.
$$
We proceed as in the proof of \eqref{derivG} to get
$$
\delta_{x_1}G(x)=\lim_{\varepsilon \rightarrow 0^+}\int_{|x-y|>\varepsilon}f(y)\int_0^\infty \delta _x^\alpha (e^{|y|^2-|x|^2}W_t(x-y))t^{\frac{|\alpha |}{2}-1}dtdy+c_\alpha f(x),\quad \mbox{ for almost all }x\in \mathbb{R}^n,
$$
where $c_\alpha \in \mathbb{R}$ and $c_\alpha=0$ if $\alpha _i$ is odd for some $i=1,...,n$. Thus \eqref{deltaA} is established when $n>1$.

If $n=1$ we can also follow the proof of Theorem \ref{Th1.1} by using the decomposition
\begin{equation}\label{decomp1001}
\delta_x^\alpha \bar{\mathcal{A}}^{-\alpha/2}\Pi _0(f)(x)=\frac{1}{\Gamma (\frac{|\alpha|}{2})}\delta_x (\overline{F}(x)+\overline{G}(x)),\quad x\in \mathbb{R}^n,
\end{equation}
where
$$
\overline{F}(x)=\int_{\mathbb{R}^n}f(y)\int_0^\infty \delta_x^\ell T_t^{\bar{\mathcal{A}}}(x,y)-\delta _x^\ell \big(e^{|y|^2-|x|^2}[W_t(y-x)-\frac{1}{2^\ell}\frac{d^\ell}{dx^\ell}W_t (0)\chi_{(1,\infty)}(t)]\big)t^{\frac{|\alpha|}{2}-1}dtdy,
,\quad x\in \mathbb{R}^n,
$$
and
$$
\overline{G}(x)=\int_{\mathbb{R}^n}f(y)\int_0^\infty \delta_x^\ell \big(e^{|y|^2-|x|^2}[W_t(y-x)-\frac{1}{2^\ell}\frac{d^\ell}{dx^\ell}W_t (0)\chi_{(1,\infty)}(t)]\big)t^{\frac{|\alpha|}{2}-1}dtdy,\quad x\in \mathbb{R}^n.
$$

When $\ell =0$, that is, $\alpha =(1,0,...,0)$, we can replace $F$ in \eqref{decomp100n} and $\overline{F}$ in \eqref{decomp1001} by
$$
F(x)=\int_{\mathbb{R}^n}f(y)\int_0^\infty \big(T_t^{\bar{\mathcal{A}}}(x,y)-e^{-|x|^2}-e^{|y|^2-|x|^2}W_t(y-x)\big)\frac{dtdy}{\sqrt{t}},\quad x\in \mathbb{R}^n,
$$
and
$$
\overline{F}(x)=\int_{\mathbb{R}^n}f(y)\int_0^\infty \Big(T_t^{\bar{\mathcal{A}}}(x,y)-e^{-|x|^2}-\big(e^{|y|^2-|x|^2}[W_t(y-x)-\frac{1}{2^\ell}\frac{d^\ell}{dx^\ell}W_t (0)\chi_{(1,\infty)}(t)]\Big)\frac{dtdy}{\sqrt{t}},\quad x\in \mathbb{R}^n,
$$
respectively, and proceed as above to obtain \eqref{deltaA}.

According to \eqref{2.1}, \eqref{2.2} and \eqref{2.3} we get
$$
\overline{R}_\alpha f(x)=\delta ^\alpha \mathcal{A}^{-|\alpha|/2}\Pi_0f(x),\quad x\in \mathbb{R}^n,
$$
and then,
$$
\overline{R}_\alpha f(x)=\lim_{\varepsilon \rightarrow 0^+}\frac{1}{\Gamma (\frac{|\alpha|}{2})}\int_{|x-y|>\varepsilon }f(y)\int_0^\infty \delta_x^\alpha T_t^{\bar{A}}(x,y)t^{\frac{|\alpha|}{2}-1}dtdy+c_\alpha f(x),\quad \mbox{ for almost all }x\in \mathbb{R}^n,
$$
with $c_\alpha=0$, when $\alpha_i$ is odd for some $i=1,...,n$.

We are going to show the $L^p(\mathbb{R}^n,\gamma_{-1})$-boundedness properties of $\overline{R}_\alpha $. We recall that
$$
\overline{R}_\alpha (x,y)=\frac{(-1)^{|\alpha|}}{2^{|\alpha|}\pi ^{\frac{n}{2}}\Gamma(\frac{|\alpha|}{2})}e^{|y|^2-|x|^2}
\int_0^\infty \frac{e^{-|\alpha|t}}{(1-e^{-2t})^{\frac{n+|\alpha|}{2}}}\widetilde{H}_\alpha\Big(\frac{y-e^{-t}x}{\sqrt{1-e^{-2t}}}\Big)t^{\frac{|\alpha|}{2}-1}dt,\quad x,y\in \mathbb{R}^n,\quad x\not=y.
$$

Consider first $1<p<\infty$ and choose $1-\frac{1}{p}<\eta <1$. By making the change of variables $s=1-e^{-2t}$, $t\in (0,\infty)$, we obtain
$$
|\overline{R}_\alpha (x,y)|\leq Ce^{|y|^2-|x|^2}\int_0^1\frac{e^{-\eta\frac{|y-x\sqrt{1-s}|^2}{s}}}{s^{\frac{n+|\alpha|}{2}}}(1-s)^{\frac{|\alpha|}{2}-1}(-\log(1-s))^{\frac{|\alpha|}{2}-1}ds,\quad x,y\in \mathbb{R}^n,\;x\not=y.
$$

Let $\beta=\eta^{-1}$ and consider the local and global operators defined on $C_c^\infty (\mathbb{R}^n)$ by
$$
\overline{R}_{\alpha ,{\rm loc}}(f)(x)=\overline{R}_\alpha (f\chi _{N_\beta}(x, \cdot ))(x),\quad \mbox{ and }\quad \overline{R}_{\alpha ,{\rm glob}}(f)(x)=\overline{R}_\alpha (f\chi _{N_\beta^c}(x,\cdot ))(x),\quad x\in \mathbb{R}^n.
$$

By proceeding as in the proof of \eqref{acotRalpha} it follows that, for each $(x,y)\in N_\beta ^c$,
$$
\overline{R}_\alpha (x,y)|\leq C\left\{\begin{array}{ll}
   \displaystyle e^{(1-\eta)|y|^2-|x|^2},  & \langle x,y\rangle \leq 0, \\[0,5cm]
   \Big(\frac{|x+y|}{|x-y|}\Big)^n\exp\big((1-\frac{\eta}{2})(|y|^2-|x|^2)-\frac{\eta }{2}|x+y||x-y|\big), & \langle x,y\rangle >0.
\end{array}
\right.
$$

We have that
\begin{align*}
\int_{\mathbb{R}^n}e^{-\frac{|y|^2}{p}+\frac{|x|^2}{p}}|\overline{R}_\alpha (x,y)|\chi _{N_\beta ^c}(x,y)dy&\leq C\left(\int_{\mathbb{R}^n}e^{|y|^2(1-\eta -\frac{1}{p})}e^{(\frac{1}{p}-1)|x|^2}dy\right.\\
&\left.+\int_{\mathbb{R}^n}|x+y|^n\exp \Big(-|x+y||x-y|\big(\frac{\eta }{2}-|1-\frac{1}{p}-\frac{\eta}{2}|\big)\Big)dy\right),\quad x\in \mathbb{R}^n.
\end{align*}
Then, since $\eta >1-\frac{1}{p}$,
$$
\sup_{x\in \mathbb{R}^n}\int_{\mathbb{R}^n}e^{-\frac{|y|^2}{p}+\frac{|x|^2}{p}}|\overline{R}_\alpha (x,y)|\chi _{N_\beta ^c}(x,y)dy<\infty.
$$
Also we get
$$
\sup_{y\in \mathbb{R}^n}\int_{\mathbb{R}^n}e^{-\frac{|y|^2}{p}+\frac{|x|^2}{p}}|\overline{R}_\alpha (x,y)|\chi _{N_\beta ^c}(x,y)dx<\infty.
$$
We conclude that $\overline{R}_{\alpha ,{\rm glob}}$ is bounded from $L^p(\mathbb{R}^n,\gamma_{-1})$ into itself.

We are going to study the operator $\overline{R}_{\alpha, {\rm loc}}$. We write
\begin{align*}
\overline{R}_\alpha (x,y)&=\frac{1}{\Gamma (\frac{|\alpha|}{2})}\left(\int_0^\infty \delta _x^\alpha [T_t^\mathcal{\bar{A}}(x,y)-e^{|y|^2-|x|^2}W_t(y-x)]t^{\frac{|\alpha |}{2}-1}dt+\int_0^\infty \delta _x^\alpha[e^{|y|^2-|x|^2}W_t(y-x)]t^{\frac{|\alpha |}{2}-1}dt\right)\\
&=I(x,y)+J(x,y),\quad x,y\in \mathbb{R}^n, \;x\not=y.
\end{align*}

By taking into account that
$$
|\delta _x^\alpha [T_t^\mathcal{\bar{A}}(x,y)-e^{|y|^2-|x|^2}W_t(y-x)]|\leq Ce^{|y|^2-|x|^2}\left(\frac{e^{-|\alpha|t}}{(1-e^{-2t})^{\frac{n+|\alpha|}{2}}}+\frac{e^{-c\frac{|x-y|^2}{t}}}{t^{\frac{n+|\alpha|}{2}}}\right),\quad x,y\in \mathbb{R}^n,\;t>0,
$$
we get (see \eqref{dif1})
$$
\int_{m(x)}^\infty |\delta _x^\alpha [T_t^\mathcal{\bar{A}}(x,y)-e^{|y|^2-|x|^2}W_t(y-x)]|t^{\frac{|\alpha|}{2}-1}dt\leq Ce^{|y|^2-|x|^2}(1+|x|)^n,\quad x,y\in \mathbb{R}^n.
$$

Also, from \eqref{decomp} and proceeding as in the proof of \eqref{acotI} we can see that
$$
\int_0^{m(x)}|\delta _x^\alpha [T_t^\mathcal{\bar{A}}(x,y)-e^{|y|^2-|x|^2}W_t(y-x)]|t^{\frac{|\alpha|}{2}-1}dt\leq Ce^{|y|^2-|x|^2}\frac{\sqrt{1+|x|}}{|x-y|^{n-\frac{1}{2}}},\quad (x,y)\in N_\beta.
$$

Since $||y|^2-|x|^2|\leq C$ when $(x,y)\in N_\beta$, we obtain that 
\begin{equation}\label{acotIb}
|I(x,y)|\leq C\frac{\sqrt{1+|x|}}{|x-y|^{n-\frac{1}{2}}},\quad (x,y)\in N_\beta,\;x\not=y.
\end{equation}
On the other hand, we observe that
$$
J(x,y)=\frac{e^{|y|^2-|x|^2}}{2^{|\alpha|}\Gamma(\frac{|\alpha|}{2})}\int_0^\infty \partial _y^\alpha (W_t(y-x))t^{\frac{|\alpha|}{2}-1}dt=2^{-|\alpha|}e^{|y|^2-|x|^2}\mathbb{R}_\alpha (y,x),\quad x,y\in \mathbb{R}^n,\;x\not=y,
$$
where $\mathbb{R}_\alpha (\cdot,\cdot)$ is the classical kernel considered in \eqref{Rclassical}. The result can be established by proceeding as in the proof of Theorem \ref{Th1.2} by taking into account that the Euclidean Riesz transform $\mathbb{R}_\alpha$ is bounded from $L^p(\mathbb{R}^n,dx)$ into itself.

To deal with the case $p=1$ we consider the local and global operators $\overline{R}_{\alpha,{\rm loc}}$ and $\overline{R}_{\alpha ,{\rm glob}}$ defined above with $N$ instead of $N_\beta$.  Since  the classical Riesz transform $\mathbb{R}_\alpha$ is bounded from $L^1(\mathbb{R}^n,dx)$ into $L^{1,\infty}(\mathbb{R}^n,dx)$ we can use \eqref{acotIb} and argue as in the proof of Theorem 1.2 to obtain that $\overline{R}_{\alpha ,{\rm loc}}$ defines a bounded operator from $L^1(\mathbb{R}^n, \gamma_{-1})$ into $L^{1,\infty }(\mathbb{R}^n,\gamma_{-1})$.

In order to prove that $\overline{R}_{\alpha ,{\rm glob}}$ defines a bounded operator from $L^1(\mathbb{R}^n,\gamma_{-1})$ into itself we make the change of variables $r=e^{-t}$, $t\in (0, \infty )$, and write $\overline{R}_\alpha (x,y)=\overline{R}_{\alpha, 1}(x,y)+\overline{R}_{\alpha ,2}(x,y)$, where
$$
\overline{R}_{\alpha ,1}(x,y)=\frac{(-1)^{|\alpha|}}{2^{|\alpha|}\pi ^{\frac{n}{2}}\Gamma(\frac{|\alpha|}{2})}e^{|y|^2-|x|^2}\int_0^{\frac{1}{2}}\frac{r^{|\alpha|-1}}{(1-r^2)^{\frac{n+|\alpha|}{2}}}\widetilde{H}_\alpha \Big(\frac{y-rx}{\sqrt{1-r^2}}\Big)(-\log r)^{\frac{|\alpha|}{2}-1}dr,\quad x,y\in \mathbb{R}^n,\;x\not=y.
$$
Suppose that $n=1$ or $|\alpha |\geq n+1$ when $n>1$. By using \cite[Lemma 3.3.3]{Sa} it follows that
\begin{align*}
|\overline{R}_{\alpha ,1}(x,y)|&\leq Ce^{|y|^2-|x|^2}\int_0^{1/2}r^{|\alpha|-1}\frac{e^{-c\frac{|y-rx|^2}{1-r^2}}}{(1-r^2)^{\frac{n}{2}}}(-\log r)^{\frac{|\alpha|}{2}-1}dr\\
&\leq Ce^{|y|^2-|x|^2}\sup_{r\in (0,1)}\frac{r^n e^{-c\frac{|y-rx|^2}{1-r^2}}}{(1-r^2)^{\frac{n}{2}}}\int_0^1(-\log r)^{\frac{|\alpha|}{2}-1}dr\\
&\leq Ce^{|y|^2-|x|^2}\min\{(1+|x|)^n,(|x|\sin \theta (x,y))^{-n}\}, \quad (x,y)\in N^c.
\end{align*}

On the other hand, by proceeding as in the estimation of $K_2^0(x,y)$ in \cite[proof of Proposition 5.1]{BrSj} we obtain, for every $(x,y)\in N^c$,
$$
|\overline{R}_{\alpha ,2}(x,y)|\leq Ce^{|y|^2-|x|^2}\int_{1/2}^1\frac{e^{-c\frac{|y-rx|^2}{1-r^2}}}{(1-r^2)^{\frac{n+2}{2}}}dr\leq Ce^{|y|^2-|x|^2}(|x|^{-n}+\min\{(1+|x|)^n,(|x|\sin \theta (x,y))^{-n}\}).
$$
From \cite[Lemma 3.3.4]{Sa} and \cite[Lemma 4.2]{BrSj} we deduce that $\overline{R}_{\alpha ,{\rm glob}}$ defines a bounded operator from $L^1(\mathbb{R}^n, \gamma_{-1})$ into $L^{1,\infty }(\mathbb{R}^n,\gamma_{-1})$.

Thus, we conclude that $R_\alpha$ can be extended to $L^1(\mathbb{R}^n, \gamma_{-1})$ as a bounded operator from $L^1(\mathbb{R}^n, \gamma_{-1})$ into $L^{1,\infty }(\mathbb{R}^n,\gamma_{-1})$.

\section{UMD spaces and Riesz transforms in the inverse Gaussian setting}

\subsection{Proof of Theorem \ref{Th1.4}.} 

For every $i=1,...,n$, by $\mathbb{R}_{e^i}$ we denote the $i$-th Euclidean Riesz transform defined, for every $f\in L^p(\mathbb{R}^n,dx)$, $1<p<\infty$, by
$$
\mathbb{R}_{e^i}(f)(x)=\lim_{\varepsilon \rightarrow 0^+}\int_{|x-y|>\varepsilon} \mathbb{R}_{e^i}(x-y)f(y)dy,\quad \mbox{fol almost all }x\in \mathbb{R}^n,
$$
where
$$
\mathbb{R}_{e^i}(z)=\frac{1}{\sqrt{\pi}}\int_0^\infty \partial_{x_i}W_t(z)\frac{dt}{\sqrt{t}},\quad z\in \mathbb{R}^n,	\;z\not=0.
$$
Observe that 
$$
\mathbb{R}_{e^i}(z)=-\frac{\sqrt{2}\Gamma (\frac{n+1}{2})}{\pi ^{\frac{n+1}{2}}}\frac{z_i}{|z|^{n+1}},\quad z\in \mathbb{R}^n,\;z\not=0,\;i=1,...,n.
$$

Let $X$ be a Banach space. For every $i=1,...,n$, we define $\mathbb{R}_{e^i}$ on $L^p(\mathbb{R}^n,dx)\otimes X$, $1\leq p<\infty$, in the obvious way.

The UMD-property for $X$ can be characterized by using $\mathbb{R}_{e^i}$, $i=1,...,n$. The properties stated in Theorems \ref{Th1.4} and \ref{Th1.5} hold when $R_{e^i}$ is replaced by $\mathbb{R}_{e^i}$, $i=1,...,n$. The estimations established in the proofs of Theorems \ref{Th1.1} and \ref{Th1.2} allow us to pass from $\mathbb{R}_{e^i}$ to $R_{e^i}$, $i=1,...,n$.

Let $i=1,...,n$ and $1<p<\infty$. We are going to see that the following two assertions are equivalent:
\begin{enumerate}
\item[(i)] $R_{e^i}$ can be extended from $(L^2(\mathbb{R}^n,\gamma_{-1})\bigcap L^p(\mathbb{R}^n,\gamma_{-1}))\otimes X$ to $L^p(\mathbb{R}^n,\gamma_{-1}, X)$ as a bounded operator from $L^p(\mathbb{R}^n,\gamma_{-1}, X)$ into itself.

\item [(ii)] $\mathbb{R}_{e^i}$ can be extended from $(L^2(\mathbb{R}^n,dx)\bigcap L^p(\mathbb{R}^n,dx)\otimes X$ to $L^p(\mathbb{R}^n,dx, X)$ as a bounded operator from $L^p(\mathbb{R}^n,dx, X)$ into itself.
\end{enumerate}

We choose $\frac{1}{p}<\eta <1$ and consider the global and local operators as in the previous sections according to the region $N_\beta$, with $\beta =\eta ^{-1}$. 

Suppose that (ii) holds. We can write 
$
R_{e^i}=(R_{e^i,{\rm loc}}-\mathbb{R}_{e^i,{\rm loc}})+\mathbb{R}_{e^i,{\rm loc}}+R_{e^i,{\rm glob}}.
$
Since $\mathbb{R}_{e^i}$ is a Calder\'on-Zygmund operator, by using a vectorial version of \cite[Proposition 3.2.5]{Sa} (see \cite[Proposition 2.3]{HTV}) we deduce that $\mathbb{R}_{e^i,{\rm loc}}$ can be extended from $(L^2(\mathbb{R}^n,\gamma_{-1})\bigcap L^p(\mathbb{R}^n,\gamma_{-1}))\otimes X$ to $L^p(\mathbb{R}^n,\gamma_{-1}, X)$ as a bounded operator from $L^p(\mathbb{R}^n,\gamma_{-1}, X)$ into itself.

According to \eqref{acotRalpha} and \eqref{acotdif} we have that
\begin{equation}\label{5.1}
|R_{e^i}(x,y)-\mathbb{R}_{e^i}(x-y)|\leq L_i(x,y),\;\;(x,y)\in  N_\beta,\quad \mbox{ and }\quad |R_{e^i}(x,y)|\leq M_i(x,y),\;\;(x,y)\in N_\beta ^c, 
\end{equation}
and the integral operators
$$
L_i(f)(x)=\int_{\mathbb{R}^n}L_i(x,y)\chi _{N_\beta}(x,y)f(y)dy,\quad \mbox{ and }\quad M_i(f)(x)=\int_{\mathbb{R}^n}M_i(x,y)\chi _{N_\beta ^c}(x,y)f(y)dy
$$
are bounded from $L^p(\mathbb{R}^n,\gamma_{-1})$ into itself. Then, $L_i$ and $M_i$ define bounded operators from $L^p(\mathbb{R}^n,\gamma_{-1},X)$ into itself, and the same property holds for the operators $R_{e^i,{\rm loc}}-\mathbb{R}_{e^i,{\rm loc}}$ and $R_{e^i,{\rm glob}}$. We conclude that (i) holds.

Suppose now that (i) holds. By \eqref{acotRalphab} we get
\begin{align*}
|R_{e^i}(x,y)|&\leq C\int_0^\infty e^{-nt}\frac{e^{-c\frac{|x-e^{-t}y|^2}{1-e^{-2t}}}}{(1-e^{-2t})^{\frac{n+1}{2}}}\frac{dt}{\sqrt{t}}\leq C\left(\int_0^{m(x)}\frac{e^{-c\frac{|x-y|^2}{t}}}{t^{\frac{n}{2}+1}}dt+\int_{m(x)}^\infty \frac{dt}{t^{\frac{n}{2}+1}}\right)\\
&\leq C\left(\frac{1}{|x-y |^n}+\frac{1}{m(x)^{\frac{n}{2}}}\right)\leq \frac{C}{|x-y|^n},\quad (x,y)\in N_\beta ,\;x\not=y.
\end{align*}

In a similar way we get, for each $k=1,...,n$,
$$
|\partial _{x_k}R_{e^i}(x,y)|\leq C\int_0^\infty e^{-nt}\frac{e^{-c\frac{|x-e^{-t}y|^2}{1-e^{-2t}}}}{(1-e^{-2t})^{\frac{n}{2}+1}}\frac{dt}{\sqrt{t}}\leq \frac{C}{|x-y|^{n+1}},\quad (x,y)\in N_\beta,\;x\not=y .
$$

Then, according to a vector valued version of \cite[Propositions 3.2.5 and 3.2.7]{Sa} (see \cite[Propositions 2.3 and 2.4]{HTV}) we deduce that $R_{e^i,{\rm loc}}$ defines a bounded operator from $L^p(\mathbb{R}^n,dx, X)$ into itself. Also, $R_{e^i,{\rm loc}}-\mathbb{R}_{e^i,{\rm loc}}$ defines a bounded operator from $L^p(\mathbb{R}^n,dx, X)$ into itself. We conclude that $\mathbb{R}_{e^i,{\rm loc}}$ defines a bounded operator from $L^p(\mathbb{R}^n,dx, X)$ into itself. Since $\mathbb{R}_{e^i}$ is dilatation invariant, by proceeding as in the proof of \cite[Theorem 1.10, $(ii)\Rightarrow (i)$]{HTV} it follows that $\mathbb{R}_{e^i}$ can be extended from $(L^2(\mathbb{R}^n,dx)\bigcap L^p(\mathbb{R}^n,dx)\otimes X$ to $L^p(\mathbb{R}^n,dx, X)$ as a bounded operator from $L^p(\mathbb{R}^n,dx, X)$ into itself.

The same arguments allow us to prove that the following assertions are equivalent.
\begin{enumerate}
\item[(iii)] $R_{e^i}$ can be extended from $(L^1(\mathbb{R}^n,\gamma_{-1})\bigcap L^2(\mathbb{R}^n,\gamma_{-1}))\otimes X$ to $L^1(\mathbb{R}^n,\gamma_{-1}, X)$ as a bounded operator from $L^1(\mathbb{R}^n,\gamma_{-1}, X)$ into $L^{1,\infty}(\mathbb{R}^n,\gamma_{-1}, X)$.

\item[(iv)] $\mathbb{R}_{e^i}$ can be extended from $(L^1(\mathbb{R}^n,dx)\bigcap L^2(\mathbb{R}^n,dx))\otimes X$ to $L^1(\mathbb{R}^n,dx, X)$ as a bounded operator from $L^1(\mathbb{R}^n,dx, X)$ into $L^{1,\infty}(\mathbb{R}^n,dx, X)$.
\end{enumerate}

Furthermore, in a similar way we can see that (i)$\iff$(ii) and (iii)$\iff$(iv) when $R_{e^i}$ and $\mathbb{R}_{e^i}$ are replaced by $R_{e^i,*}$ and $\mathbb{R}_{e^i,*}$, respectively. 

The proof of Theorem \ref{Th1.4} is thus finished.

\subsection{Proof of Theorem \ref{Th1.5}.} Let $1<p<\infty$ and $i=1,...,n$. We are going to see that the following two assertions are equivalent.

(a) For every $f\in L^p(\mathbb{R}^n,\gamma_{-1}, X)$, there exists 
$$
\lim_{\varepsilon \rightarrow 0^+} \int_{|x-y|>\varepsilon}R_{e^i}(x,y)f(y)dy,\quad \mbox{for almost all }x\in \mathbb{R}^n.
$$

(b) For every $f\in L^p(\mathbb{R}^n,dx, X)$, there exists 
$$
\lim_{\varepsilon \rightarrow 0^+} \int_{|x-y|>\varepsilon}\mathbb{R}_{e^i}(x-y)f(y)dy,\quad \mbox{for almost all }x\in \mathbb{R}^n.
$$

We consider again $\frac{1}{p}<\eta<1$ and $\beta=\eta ^{-1}$. Suppose that $(a)$ is true. Let $f\in L^p(\mathbb{R}^n,dx, X)$. We can write 
\begin{align*}
\int_{|x-y|>\varepsilon}\mathbb{R}_{e^i}(x-y)f(y)dy&=\int_{|x-y|>\varepsilon}(\mathbb{R}_{e^i}(x-y)-R_{e^i}(x-y))\chi _{N_\beta}(x,y)f(y)dy+\int_{|x-y|>\varepsilon}R_{e^i}(x,y)\chi _{N_\beta}(x,y)f(y)dy\\
&\quad +\int_{|x-y|>\varepsilon}\mathbb{R}_{e^i}(x-y)\chi _{N_\beta ^c}(x,y)f(y)dy,\quad x\in \mathbb{R}^n,\;\varepsilon >0.
\end{align*}

Since \eqref{5.1} holds and the operator $L_i$ is bounded from $L^p(\mathbb{R}^n,dx, X)$ into itself, there exists the limit
$$
\lim_{\varepsilon \rightarrow 0^+}\int_{|x-y|>\varepsilon}(R_{e^i}(x,y))-\mathbb{R}_{e^i}(x-y)\chi _{N_\beta}f(y)dy,\quad \mbox{ for almost all }x\in \mathbb{R}^n.
$$
On the other hand, we get
\begin{align*}
\int_{\mathbb{R}^n}|\mathbb{R}_{e^i}(x-y)|\chi _{N_\beta ^c}(x,y)\|f(y)\|dy&\leq C\int_{\mathbb{R}^n}\frac{1}{|x-y|^n}\chi _{N_\beta ^c}(x,y)\|f(y)\|dy\\
&\hspace{-3cm}\leq C \left(\int_{|x-y|>\beta n\sqrt{m(x)}}\frac{dy}{|x-y|^{np'}}\right)^{1/p'}\|f\|_{L^p(\mathbb{R}^n, dx, X)}\\
&\hspace{-3cm}\leq C\left(\int_{\beta n\sqrt{m(x)}}^\infty \frac{dr}{r^{n(p'-1)+1}}\right)^{1/p'}\|f\|_{L^p(\mathbb{R}^n, dx, X)}\leq \frac{C}{m(x)^{n/(2p)}}\|f\|_{L^p(\mathbb{R}^n, dx, X)},\quad x\in \mathbb{R}^n.
\end{align*}

Then, there exists the limit
$$
\lim_{\varepsilon \rightarrow 0^+}\int_{\mathbb{R}^n}\mathbb{R}_{e^i}(x-y)\chi _{N_\beta ^c}(x,y)f(y)dy,\quad x\in \mathbb{R}^n.
$$

Suppose that $g\in L^p(\mathbb{R}^n, \gamma_{-1},X)$. It was seen in the proof of Theorem \ref{Th1.2} that
$$
\int_{\mathbb{R}^n}|R_{e^i}(x,y)|\chi _{N_\beta ^c}(x,y)\|g(y)\|dy\in L^p(\mathbb{R}^n,\gamma_{-1}).
$$
Then,
$$
\lim_{\varepsilon \rightarrow 0^+}\int_{\mathbb{R}^n}R_{e^i}(x,y)\chi _{N_\beta ^c}(x,y)g(y)dy=\int_{\mathbb{R}^n}R_{e^i}(x,y)\chi _{N_\beta^c}(x,y)g(y)dy,\quad \mbox{ for almost }x\in \mathbb{R}^n.
$$
Since (a) holds, there also exists the limit
$$
\lim_{\varepsilon \rightarrow 0^+}\int_{|x-y|>\varepsilon}R_{e^i}(x,y)\chi _{N_\beta}(x,y)g(y)dy,\quad \mbox{ for almost all }x\in\mathbb{R}^n. 
$$

Let $k\in \mathbb{N}$. We have that $|y|\leq \beta n+k$ provided that $|x|\leq k$ and $|x-y|\leq \beta n\min\{1,|x|^{-1}\}$. Then, for every $\varepsilon >0$,
$$
\int_{|x-y|>\varepsilon}R_{e^i}(x,y)\chi _{N_\beta}(x,y)f(y)dy=\int_{|x-y|>\varepsilon}R_{e^i}(x,y)\chi _{N_\beta}(x,y)\chi _{B(0,\beta n+k)}(y)f(y)dy,\quad |x|\leq k.
$$

Since $f\in L^p(\mathbb{R}^n,dx,X)$, $\chi_{B(0,\beta n+k)}f\in L^p(\mathbb{R}^n, \gamma_{-1},X)$ and then there exists
$$
\lim_{\varepsilon \rightarrow 0^+}\int_{|x-y|>\varepsilon}R_{e^i}(x,y)\chi _{N_\beta}(x,y)f(y)dy,\quad \mbox{ for almost all }x\in B(0,k). 
$$

Hence, we get that there exists 
$$
\lim_{\varepsilon \rightarrow 0^+}\int_{|x-y|>\varepsilon}R_{e^i}(x,y)\chi _{N_\beta}(x,y)f(y)dy,\quad \mbox{ for almost all }x\in\mathbb{R}^n. 
$$
We conclude that there exists 
$$
\lim_{\varepsilon \rightarrow 0^+}\int_{|x-y|>\varepsilon}\mathbb{R}_{e^i}(x-y)f(y)dy,\quad \mbox{ for almost all }x\in\mathbb{R}^n. 
$$

In a similar we can see that (a) holds provided that (b) is true. Note that $L^p(\mathbb{R}^n,\gamma_{-1},X)\subset L^p(\mathbb{R}^n,dx,X)$.

As it was proved in \cite{BrSj} the operator $S_i$ defined by
$$
S_i(f)(x)=\int_{\mathbb{R}^n}|R_{e^i}(x,y)|\chi _{N_\beta ^c}(x,y)f(y)dy,\quad x\in \mathbb{R}^n,
$$
is bounded from $L^1(\mathbb{R}^n, \gamma_{-1})$ into $L^{1,\infty}(\mathbb{R}^n, \gamma_{-1})$. Then, for every $f\in L^1(\mathbb{R}^n, \gamma_{-1})$ there exists 
$$
\lim_{\varepsilon \rightarrow 0^+}\int_{|x-y|>\varepsilon}R_{e^i}(x,y)\chi _{N_\beta}(x,y)f(y)dy=\int_{\mathbb{R}^n}R_{e^i}(x,y)\chi _{N_\beta}(x,y)f(y)dy,\quad \mbox{ for almost all }x\in\mathbb{R}^n. 
$$
The same arguments allow us to prove that (a) $\iff$ (b) when $p=1$.

By using a $n$-dimensional version of \cite[Theorem D]{TZ} we deduce that the properties $(i)$, $(ii)$ and $(iii)$ in Theorem \ref{Th1.5} are equivalent.

By proceeding as above we can see that the properties in (a) and (b) continue being equivalent when we replace the principal values by the maximal operators $R_{e^i,*}$ and $\mathbb{R}_{e^i,*}$ in (a) and (b), respectively. Then, the property (b) is equivalent to the property UMD for $X$ (see the comments before the proof of \cite[Theorem 1.10, p. 19]{HTV}). 

Thus the proof of Theorem \ref{Th1.5} is finished.

\section{UMD spaces and the imaginary powers of $\mathcal{A}$}
In this section we prove Theorem \ref{Th1.6}.

According to \cite[(11)]{BCFR} we have that, for every $f\in C_c^\infty (\mathbb{R}^n)$, 
$$
\Big(-\frac{\Delta }{2}\Big)^{i\gamma}f(x)=\lim_{\varepsilon \rightarrow 0^+}\Big(\alpha (\varepsilon )f(x)+\int_{|x-y|>\varepsilon }K_\gamma (x-y)f(y)dy\Big),\quad \mbox{ for a.e. }x\in \mathbb{R}^n,
$$
where
$$
K_\gamma (z)=-\int_0^\infty \phi_\gamma  (t)\partial _tW_t(z)dt,\quad z\in \mathbb{R}^n\setminus\{0\},
$$
$$
\alpha (\varepsilon )=\frac{1}{\Gamma (\frac{n}{2})}\int_0^\infty \phi_\gamma  \big(\frac{\varepsilon ^2}{4u}\big)e^{-u}u^{\frac{n}{2}-1}du,\quad \varepsilon \in (0,\infty ),
$$
and 
$$
\phi _\gamma (t)=\frac{t^{-i\gamma}}{\Gamma (1-i\gamma)},\quad t\in (0,\infty ).
$$
Note that the limit $\lim_{t\rightarrow 0^+}\phi_\gamma  (t)$ does not exist.

Let $f,g\in C_c^\infty (\mathbb{R}^n)$. By proceeding as in \cite[p. 213]{BCFR} we can see that
\begin{align*}
\langle \mathcal{A}^{i\gamma}f,g\rangle_{L^2(\mathbb{R}^n,\gamma_{-1})}&=\int_0^\infty \phi _\gamma (t)\Big(-\frac{d}{dt}\Big)\int_{\mathbb{R}^n}\int_{\mathbb{R}^n}T_t^\mathcal{A}(x,y)f(y)dy\overline{g(x)}\gamma_{-1}(x)dxdt\\
&=\int_0^\infty \phi _\gamma (t)\Big(-\frac{d}{dt}\Big)\int_{\mathbb{R}^n}\int_{\mathbb{R}^n}W_t(x-y)f(y)dy\overline{g(x)}\gamma_{-1}(x)dxdt\\
&\quad +\int_0^\infty \phi _\gamma (t)\Big(-\frac{d}{dt}\Big)\int_{\mathbb{R}^n}\int_{\mathbb{R}^n}(T_t^\mathcal{A}(x,y)-W_t(x-y))f(y)dy\overline{g(x)}\gamma_{-1}(x)dxdt.
\end{align*}

We have that 
\begin{align*}
\partial_tT_t^\mathcal{A}(x,y)&=\Big(-n-2e^{-t}\sum_{i=1}^ny_i(x_i-e^{-t}y_i)+\frac{2e^{-2t}|x-e^{-t}y|^2}{1-e^{-2t}}\Big)\frac{T_t^\mathcal{A}(x,y)}{1-e^{-2t}}\\
&=\Big(-n-2e^{-t}(\langle x,y\rangle-e^{-t}|y|^2)+\frac{2e^{-2t}|x-e^{-t}y|^2}{1-e^{-2t}}\Big)\frac{T_t^\mathcal{A}(x,y)}{1-e^{-2t}},\quad x,y\in \mathbb{R}^n,\;t>0,
\end{align*}
and writing $2e^{-t}\langle x,y\rangle=|x|^2 +e^{-2t}|y|^2-|x-e^{-t}y|^2$, we get
\begin{equation}\label{derivadaT}
\partial_tT_t^\mathcal{A}(x,y)=\Big(-n+e^{-2t}(|y|^2-|x|^2)- (1-e^{-2t})|x|^2+\frac{1+e^{-2t}}{1-e^{-2t}}|x-e^{-t}y|^2)\Big)\frac{T_t^\mathcal{A}(x,y)}{1-e^{-2t}},\quad x,y\in \mathbb{R}^n,\;t>0.
\end{equation}
Also, we have that
$$
\partial_tW_t(x-y)=\Big(-n+\frac{|x-y|^2}{t}\Big)\frac{W_t(x-y)}{2t},\quad x,y\in \mathbb{R}^n,\;t>0.
$$
We can write, for every $x,y\in \mathbb{R}^n$ and $t>0$,
\begin{align}\label{difbas}
\partial_tT_t^\mathcal{A}(x,y)-\partial_tW_t(x-y)&=\Big(-n+\frac{|x-y|^2}{t}\Big)\frac{T_t^\mathcal{A}(x,y)-W_t(x-y)}{2t}\nonumber \\
&\hspace{-3cm}+\left[-n\Big(\frac{1}{1-e^{-2t}}-\frac{1}{2t}\Big)+\frac{e^{-2t}}{1-e^{-2t}}(|y|^2-|x|^2)-|x|^2-\frac{|x-e^{-t}y|^2}{1-e^{-2t}}+2\Big(\frac{|x-e^{-t}y|^2}{(1-e^{-2t})^2}-\frac{|x-y|^2}{(2t)^2}\Big)\right]T_t^\mathcal{A}(x,y).
\end{align}

The derivative under the integral sign is justified and we get
\begin{align*}
\langle \mathcal{A}^{i\gamma}f,g\rangle_{L^2(\mathbb{R}^n,\gamma_{-1})}&=-\int_0^\infty \phi_\gamma (t)\int_{\mathbb{R}^n}\int_{\mathbb{R}^n}\partial _tW_t(x-y)f(y)dy\overline{g(x)}\gamma_{-1}(x)dxdt\\
&\quad -\int_0^\infty \phi_\gamma (t)\int_{\mathbb{R}^n}\int_{\mathbb{R}^n}\partial_t(T_t^\mathcal{A}(x,y)-W_t(x-y))f(y)dy\overline{g(x)}\gamma_{-1}(x)dxdt.
\end{align*}

To ensure the change on the order of integration we are going to see that
\begin{equation}\label{6.1}
J(x)=\int_0^\infty \int_{\mathbb{R}^n}\left|\partial_t(T_t^\mathcal{A}(x,y)-W_t(x-y))\right||f(y)|dydt<\infty, \quad x\in \mathbb{R}^n,
\end{equation}
and
\begin{equation}\label{6.2}
\int_0^\infty  \int_{\mathbb{R}^n}\int_{\mathbb{R}^n}\left|\partial _t(T_t^\mathcal{A}(x,y)-W_t(x-y))\right||f(y)||g(x)|\gamma_{-1}(x)dydxdt<\infty .
\end{equation}

We observe that, since $g\in C_c^\infty (\mathbb{R}^n)$, it is sufficient to see that $J(x)\leq h(x)$, $x\in \mathbb{R}^n$, for certain continuous function $h$.

We consider the decomposition
$$
J(x)=\left(\int_0^{m(x)}+\int_{m(x)}^\infty\right)\int_{\mathbb{R}^n}\left|\partial_t(T_t^\mathcal{A}(x,y)-W_t(x-y))\right||f(y)|dydt=J_1(x)+J_2(x),\quad x\in \mathbb{R}^n.
$$

Since $f\in C_c^\infty (\mathbb{R}^n)$ we have that
$$
J_2(x)\leq C(1+|x|)^2\int_{m(x)}^\infty \int_{{\rm supp }f}\frac{1}{t^{\frac{n}{2}+1}}dtdy\leq C\frac{(1+|x|)^2}{m(x)^{\frac{n}{2}}},\quad x\in \mathbb{R}^n.
$$
We now estimate $J_1(x)$, $x\in \mathbb{R}^n$. We take into account that, if $\supp f\subset B(0,r_f)$, with $r_f>0$, then
\begin{align*}
|x-e^{-t}y|^2&\geq |x-y|^2+(1-e^{-t})^2|y|^2-2|x-y||y|(1-e^{-t})\nonumber\\
&\geq |x-y|^2-d(x)(1-e^{-t}), \quad x\in \mathbb{R}^n,\;y\in {\rm supp }f,\; t>0,
\end{align*}
where $d(x)=2r_f(r_f+|x|)$, $x\in \mathbb{R}^n$. By considering the decomposition \eqref{decomposition} for $\alpha =0$, the estimations \eqref{differences} and the proof of \eqref{Halpha} we can see that
$$
|T_t^\mathcal{A}(x,y)-W_t(x-y)|\leq Ce^{d(x)}e^{-c\frac{|x-y|^2}{t}}\Big(\frac{1}{t^{\frac{n}{2}-1}}+(|y|+|x-y|)^n\Big),\quad x\in \mathbb{R}^n,\;\;y\in {\rm supp }f,\;t\in (0,1).
$$
Then, according to \eqref{difbas} we have that
\begin{align*}
|\partial_t(T_t^\mathcal{A}(x,y)-W_t(x-y))|&\leq Ce^{d(x)}e^{-c\frac{|x-y|^2}{t}}\left(\frac{1+|x|^2}{t^{\frac{n}{2}}}+\frac{(|y|+|x-y|)^n}{t}+\frac{|x-y||y+x|}{t^{\frac{n}{2}+1}}+\frac{|y|+ |x-y|}{t^{\frac{n}{2}}}\right)\\
&\leq Ce^{d(x)}e^{-c\frac{|x-y|^2}{t}}\left(\frac{1+|x|^2}{t^{\frac{n}{2}}}+\frac{(1+|x|)^n}{t}+\frac{1+|x|}{t^{\frac{n+1}{2}}}\right),\quad x\in \mathbb{R}^n,\;\;y\in {\rm supp }f,\;t\in (0,1).
\end{align*}
Since $m(x)\sim (1+|x|)^{-2}$, $x\in \mathbb{R}^n$, we get when $t\in (0,m(x))$,
$$
|\partial_t(T_t^\mathcal{A}(x,y)-W_t(x-y))|\leq Ce^{d(x)}e^{-c\frac{|x-y|^2}{t}}\left(\frac{1+|x|}{t^{\frac{n+1}{2}}}+\frac{(1+|x|)^n}{t}\right),\quad x\in \mathbb{R}^n,\;\;y\in {\rm supp }f.	 
$$
We deduce that
\begin{align*}
J_1(x)&\leq Ce^{d(x)}\int_0^{m(x)}\int_{\supp f}\left(\frac{1+|x|}{t^{\frac{3}{4}}}+(1+|x|)^nt^{\frac{n}{2}-\frac{5}{4}}\right)\frac{dy}{|x-y|^{n-\frac{1}{2}}}\\
&\leq Ce^{d(x)}\Big((1+|x|)m(x)^{\frac{1}{4}}+(1+|x|)^nm(x)^{\frac{n}{2}-\frac{1}{4}}\Big)\leq Ce^{d(x)}\sqrt{1+|x|},\quad x\in \mathbb{R}^n.
\end{align*}
Then, $|J(x)|\leq C((1+|x|)^{n+2}+e^{d(x)}\sqrt{1+|x|})$, $x\in \mathbb{R}^n$ and \eqref{6.1} and \eqref{6.2} are established.

Note that the estimations that we have just proved are depending on $f$.

By interchanging the order of integration we get
$$
\langle \mathcal{A}^{i\gamma}f,g\rangle _{L^2(\mathbb{R}^n,\gamma_{-1})}=\langle \big(-\frac{\Delta}{2}\big)^{i\gamma}f,g\rangle_{L^2(\mathbb{R}^n,\gamma_{-1})}+\int_{\mathbb{R}^n}\int_{\mathbb{R}^n}f(y)\int_0^\infty \phi _\gamma (t) \Big(\partial_tT_t^\mathcal{A}(x,y)-\partial_tW_t(x-y)\Big)dtdyg(x)\gamma_{-1}(x)dx.
$$
It follows that
\begin{align*}
\mathcal{A}^{i\gamma}(f)(x)&=\big(-\frac{\Delta}{2}\big)^{i\gamma}f(x)-\int_{\mathbb{R}^n}f(y)\int_0^\infty \phi _\gamma (t) \Big(\partial_tT_t^\mathcal{A}(x,y)-\partial_tW_t(x-y)\Big)dtdy\\
&=\big(-\frac{\Delta}{2}\big)^{i\gamma}f(x)-\lim_{\varepsilon \rightarrow 0^+}\int_{|x-y|>\varepsilon}f(y)\int_0^\infty \phi _\gamma (t) \Big(\partial _tT_t^\mathcal{A}(x,y)-\partial_tW_t(x-y)\Big)dtdy,
\end{align*}
for almost all $x\in \mathbb{R}^n$.

We conclude that 
$$
\mathcal{A}^{i\gamma}(f)(x)=\lim_{\varepsilon \rightarrow 0^+}\Big(\int_{|x-y|>\varepsilon}K_\gamma ^\mathcal{A}(x,y)f(y)dy+\alpha (\varepsilon )f(x)\Big),\quad\mbox{for almost all }x\in \mathbb{R}^n,
$$
where
$$
K_\gamma ^\mathcal{A}(x,y)=-\int_0^\infty \phi _\gamma (t)\partial_tT_t^\mathcal{A}(x,y)dt,\quad x,y\in \mathbb{R}^n,\;t>0.
$$

Salogni (\cite[Theorem 3.4.3]{Sa}) proved that $\mathcal{A}^{i\gamma}$ is bounded from $L^p(\mathbb{R}^n,\gamma_{-1})$ into itself, for every $1<p<\infty$, and Bruno (\cite[Theorem 4.1, (i)]{Br}) established that $\mathcal{A}^{i\gamma}$ is bounded from $L^1(\mathbb{R}^n,\gamma_{-1})$ into $L^{1,\infty}(\mathbb{R}^n,\gamma_{-1})$. In order to extend $\mathcal{A}^{i\gamma}$ to functions taking values in a Banach space we need to prove these results in a different way by making a comparation between $\mathcal{A}^{i\gamma}$ and $(-\frac{\Delta}{2})^{i\gamma}$.

Let $\beta >0$. We define the local and global part as follows
$$
\mathcal{A}^{i\gamma}_{\rm loc}(f)(x)=\lim_{\varepsilon \rightarrow 0^+}\Big(\alpha (\varepsilon )f(x)+\int_{|x-y|>\varepsilon}K_\gamma ^\mathcal{A}(x,y)\chi_{N_\beta}(x,y)f(y)dy\Big),\quad x\in \mathbb{R}^n,
$$
and
$$
\mathcal{A}^{i\gamma}_{\rm glob}(f)(x)=\int_{|x-y|>\varepsilon}K_\gamma ^\mathcal{A}(x,y)\chi_{N_\beta^c}(x,y)f(y)dy,\quad x\in \mathbb{R}^n.
$$

The operators $(-\frac{\Delta}{2})_{\rm loc}^{i\gamma}$ and $(-\frac{\Delta}{2})_{\rm glob}^{i\gamma}$ are defined in analogous way.

We are going to see that
\begin{equation}\label{6.4}
|K_\gamma ^\mathcal{A}(x,y)-K_\gamma (x-y)|\leq C\frac{\sqrt{1+|x|}}{|x-y|^{n-\frac{1}{2}}},\quad (x,y)\in N_\beta.
\end{equation}

We can write, 
\begin{align*}
|K_\gamma ^\mathcal{A}(x,y)-K_\gamma (x-y)|&\leq C\left(\int_0^{m(x)}|\partial_t(T_t^\mathcal{A}(x,y)-W_t(x-y))|dt\right.\\
&\left.+\int_{m(x)}^\infty|\partial_tT_t^\mathcal{A}(x,y)|dt+\int_{m(x)}^\infty|\partial_tW_t(x-y)|dt\right)=\sum_{i=1}^3I_i(x,y),\quad x,y\in \mathbb{R}^n.
\end{align*}

First we observe that
$$
I_3(x,y)\leq C\int_{m(x)}^\infty \frac{e^{-c\frac{|x-y|^2}{t}}}{t^{\frac{n}{2}+1}}dt\leq \frac{C}{m(x)^{\frac{n}{2}}}\leq C(1+|x|)^n\leq C\frac{\sqrt{1+|x|}}{|x-y|^{n-\frac{1}{2}}}, \quad (x,y)\in N_\beta.
$$
By using \eqref{derivadaT} and since $||y|^2-|x|^2|\leq |x-y||x+y|\leq C$ when $(x,y)\in N_\beta$, we have that
$$
I_2(x,y)\leq C(1+|x|^2)\int_{m(x)}^\infty \frac{e^{-\frac{n}{2}t}}{t^{\frac{n}{2}+1}}dt\leq C\frac{(1+|x|^2)}{m(x)^{\frac{n}{2}+1}}\leq C(1+|x|)^n\leq C\frac{\sqrt{1+|x|}}{|x-y|^{n-\frac{1}{2}}}, \quad (x,y)\in N_\beta.
$$

Finally, from \eqref{difbas}, proceeding as above for the estimation of $J_1$, but now by taking into account that $(x,y)\in N_\beta$, we obtain that
\begin{align*}
I_1(x,y)&\leq C\int_0^{m(x)}e^{-c\frac{|x-y|^2}{t}}\left(\frac{1+|x|}{t^{\frac{n+1}{2}}}+\frac{(1+|x|)^n}{t}\right)dt\leq C\frac{(1+|x|)m(x)^{\frac{1}{4}}+(1+|x|)^nm(x)^{\frac{n}{2}-\frac{1}{4}}}{|x-y|^{n-\frac{1}{2}}}\\
&\leq  C\frac{\sqrt{1+|x|}}{|x-y|^{n-\frac{1}{2}}},\quad (x,y)\in N_\beta.
\end{align*}
Thus, \eqref{6.4} is proved.

As it was established in Section 3.2,
$$
\sup_{x\in \mathbb{R}^n}\int_{\mathbb{R}^n}\frac{\sqrt{1+|x|}}{|x-y|^{n-\frac{1}{2}}}\chi_{N_\beta}(x,y)dy<\infty\quad \mbox{ and }\quad \sup_{y\in \mathbb{R}^n}\int_{\mathbb{R}^n}\frac{\sqrt{1+|x|}}{|x-y|^{n-\frac{1}{2}}}\chi_{N_\beta}(x,y)dx<\infty.
$$

Then, the operator $L_\beta$ defined by
$$
L_\beta (f)(x)=\int_{\mathbb{R}^n}|K_\gamma ^\mathcal{A}(x,y)-K_\gamma (x,y)|\chi_{N_\beta}(x,y)f(y)dy,\quad x\in \mathbb{R}^n,
$$
is bounded from $L^p(\mathbb{R}^n,dx)$ into itself for every $1\leq p<\infty$. From \cite[Proposition 3.2.5]{Sa} it follows that $L_\beta$ is bounded from $L^p(\mathbb{R}^n,\gamma_{-1})$ into itself, for every $1\leq p<\infty$.

Note that
$$
|\mathcal{A}^{i\gamma}_{\rm loc}(f)-\big(-\frac{\Delta}{2}\big)^{i\gamma}_{\rm loc}(f)|\leq CL_\beta (|f|).
$$

We now study the global operator $\mathcal{A}^{i\gamma}_{\rm glob}$. We recall that $\mathcal{A}^{i\gamma}_{\rm glob}$ is the integral operator defined by
$$
\mathcal{A}^{i\gamma}_{\rm glob}(f)(x)=-\int_{\mathbb{R}^n}\int_0^\infty \phi _\gamma (t)\partial_tT_t^\mathcal{A}(x,y)dt\chi_{N_\beta ^c}(x,y)f(y)dy,\quad x\in \mathbb{R}^n.
$$
We have, by making $r=e^{-t}$, $t\in (0,\infty)$, that
\begin{align*}
\int_0^\infty |\phi _\gamma (t)|\partial_tT_t^\mathcal{A}(x,y)|dt&\leq C\int_0^1 \Big|\partial_r\Big[r^n\frac{e^{-\frac{|x-ry|^2}{1-r^2}}}{(1-r^2)^{\frac{n}{2}}}\Big]\Big|dr\\
&\hspace{-2.5cm}\leq C\int_0^1\left|\frac{n(1-r^2) +2r(1-r^2)\sum_{i=1}^ny_i(x_i-ry_i)-2r^2|x-ry|^2}{(1-r^2)^2}\right|r^{n-1}\frac{e^{-\frac{|x-ry|^2}{1-r^2}}}{(1-r^2)^{\frac{n}{2}}}dr\\
&\hspace{-2.5cm}=C\int_0^1\left|\frac{(1-r^2)(n+|x|^2-r^2|y|^2)-(1+r^2)|x-ry|^2}{(1-r^2)^2}\right|r^{n-1}\frac{e^{-\frac{|x-ry|^2}{1-r^2}}}{(1-r^2)^{\frac{n}{2}}}dr.
\end{align*}

For every $x,y\in \mathbb{R}^n$ there exists a polynomial $P_{x,y}$ with degree 4 and a positive function $Q_{x,y}$ such that
$$
\partial_r\Big[r^n\frac{e^{-\frac{|x-ry|^2}{1-r^2}}}{(1-r^2)^{\frac{n}{2}}}\Big]=P_{x,y}(r)Q_{x,y}(r),\quad r\in (0,1).
$$
Then, for every $x,y\in \mathbb{R}^n$, the function 
$$
\partial_r\left[r^n\frac{e^{-\frac{|x-ry|^2}{1-r^2}}}{(1-r^2)^{\frac{n}{2}}}\right]
$$
changes the sign at most four times in $(0,1)$. We deduce that
$$
\int_0^\infty |\phi_\gamma (t)|\big|\partial _tT_t^\mathcal{A}(x,y)\big|dt\leq C\int_0^1\Big|\partial_ r\Big[r^n\frac{e^{-\frac{|x-ry|^2}{1-r^2}}}{(1-r^2)^{\frac{n}{2}}}\Big]\Big|dr\leq C\sup_{r\in (0,1)}r^n\frac{e^{-\frac{|x-ry|^2}{1-r^2}}}{(1-r^2)^{\frac{n}{2}}},\quad x,y\in \mathbb{R}^n.
$$

If $(u,v)\in N_\sigma $, with $\sigma >0$, we have that
$$
\frac{1}{(1+\sigma n)(1+|v|)}\leq \frac{1}{1+|u|}\leq \frac{1+\sigma n}{1+|v|}.
$$
Also, $\frac{1}{2}\min\{1,|x|^{-1}\}\leq (1+|x|)^{-1}\leq \min\{1,|x|^{-1}\}$, $x\in \mathbb{R}^n$. Then, if $(y,x)\in N$, then $(x,y)\in N_{2(1+n)}$.

We take $\beta=2(1+n)$. Since $(y,x)\not\in N$ when $(x,y)\not \in N_\beta$, according to \cite[Proposition 2.1]{MPS} we get, for $(x,y)\not\in N_\beta$,
$$
\sup_{r\in (0,1)}\frac{e^{-\frac{|x-ry|^2}{1-r^2}}}{(1-r^2)^{\frac{n}{2}}}\leq C\left\{\begin{array}{ll}
e^{-|x|^2},&\langle x,y\rangle \leq 0,\\[0.3cm]
\displaystyle \Big(\frac{|x+y|}{|x-y|}\Big)^{\frac{n}{2}}\exp\Big(\frac{|y|^2-|x|^2}{2}-\frac{|x-y||x+y|}{2}\Big),&\langle x,y\rangle >0.
\end{array}
\right.
$$

By proceeding as in the proof of Theorem \ref{Th1.2} we can see that, for every $1<p<\infty$,
$$
\sup_{x\in \mathbb{R}^n}\int_{\mathbb{R}^n}e^{\frac{|x|^2-|y|^2}{p}}\chi _{N_\beta ^c}(x,y)\int_0^\infty |\phi _\gamma (t)|\Big|\partial_tT_t^\mathcal{A}(x,y)\Big|dtdy<\infty,
$$
and
$$
\sup_{y\in \mathbb{R}^n}\int_{\mathbb{R}^n}e^{\frac{|x|^2-|y|^2}{p}}\chi _{N_\beta ^c}(x,y)\int_0^\infty |\phi _\gamma (t)|\Big|\partial_tT_t^\mathcal{A}(x,y)\Big|dtdy<\infty .
$$

Hence the operator $\mathcal{L}_\beta$ defined by 
$$
\mathcal{L}_\beta (f)(x)=\int_{\mathbb{R}^n}f(y)\chi _{N_\beta ^c}(x,y)\int_0^\infty |\phi _\gamma (t)|\Big|\partial_tT_t^\mathcal{A}(x,y)\Big|dtdy,\quad x\in \mathbb{R}^n,
$$
is bounded from $L^p(\mathbb{R}^n,\gamma_{-1})$ into itself, for every $1<p<\infty$.

On the other hand, according to \cite[Lemma 3.3.3]{Sa}
$$
\sup_{t>0}T_t^\mathcal{A}(x,y)\leq Ce^{-|x|^2}\min\{(1+|x|)^n, (|x|\sin \theta (x,y))^{-n}\},\quad (x,y)\in N_\beta ^c,\;x,y\not=0.
$$
We recall that $\theta (x,y)\in [0,\pi)$ represents the angle between the nonzero vectors $x$ and $y$ when $n>1$ and $\theta (x,y)=0$, $x,y\in \mathbb{R}^n$, when $n=1$. By \cite[Lemma 3.3.4]{Sa} the operator $\mathcal{L}_\beta $ is bounded from $L^1(\mathbb{R}^n,\gamma _{-1})$ into $L^{1,\infty}(\mathbb{R}^n,\gamma_{-1})$.

Note that $|\mathcal{A}_{\rm glob}^{i\gamma}(f)|\leq \mathcal{L}_\beta(|f|)$. 

By taking into account that $(-\frac{\Delta}{2})^{i\gamma}$ is a Calder\'on-Zygmund operator, the arguments developed in the proofs of Theorems \ref{Th1.2} and \ref{Th1.3} allow us to finish the proof of this theorem.
\bibliographystyle{acm}

\begin{thebibliography}{10}

\bibitem{AT}
{\sc Abu-Falahah, I., and Torrea, J.~L.}
\newblock Hermite function expansions versus {H}ermite polynomial expansions.
\newblock {\em Glasg. Math. J. 48}, 2 (2006), 203--215.

\bibitem{AFS}
{\sc Aimar, H., Forzani, L., and Scotto, R.}
\newblock On {R}iesz transforms and maximal functions in the context of
  {G}aussian harmonic analysis.
\newblock {\em Trans. Amer. Math. Soc. 359}, 5 (2007), 2137--2154.

\bibitem{BCdL}
{\sc Betancor, J.~J., Castro, A., and de~Le\'on-Contreras, M.}
\newblock The {H}ardy-{L}ittlewood property and maximal operators associated
  with the inverse {G}auss measure.
\newblock Preprint 2019 \href{http://arxiv.org/abs/}{(arXiv:2010.01341)}.

\bibitem{BCCR}
{\sc Betancor, J.~J., Castro, A.~J., Curbelo, J., and Rodr\'{\i}guez-Mesa, L.}
\newblock Characterization of {UMD} {B}anach spaces by imaginary powers of
  {H}ermite and {L}aguerre operators.
\newblock {\em Complex Anal. Oper. Theory 7}, 4 (2013), 1019--1048.

\bibitem{BCFR}
{\sc Betancor, J.~J., Crescimbeni, R., Fari\~{n}a, J.~C., and
  Rodr\'{\i}guez-Mesa, L.}
\newblock Multipliers and imaginary powers of the {S}chr\"{o}dinger operators
  characterizing {UMD} {B}anach spaces.
\newblock {\em Ann. Acad. Sci. Fenn. Math. 38}, 1 (2013), 209--227.

\bibitem{BFMT}
{\sc Betancor, J.~J., Fari\~{n}a, J.~C., Mart\'{\i}nez, T., and Torrea, J.~L.}
\newblock Riesz transform and {$g$}-function associated with {B}essel operators
  and their appropriate {B}anach spaces.
\newblock {\em Israel J. Math. 157\/} (2007), 259--282.

\bibitem{BFRT}
{\sc Betancor, J.~J., Fari\~{n}a, J.~C., Rodr\'{\i}guez-Mesa, L., and Testoni,
  R.}
\newblock Higher order {R}iesz transforms in the ultraspherical setting as
  principal value integral operators.
\newblock {\em Integral Equations Operator Theory 70}, 4 (2011), 511--539.

\bibitem{Bou}
{\sc Bourgain, J.}
\newblock Some remarks on {B}anach spaces in which martingale difference
  sequences are unconditional.
\newblock {\em Ark. Mat. 21}, 2 (1983), 163--168.

\bibitem{Br}
{\sc Bruno, T.}
\newblock Singular integrals and {H}ardy type spaces for the inverse {G}auss
  measure.
\newblock Preprint 2018
  \href{http://arxiv.org/abs/1801.09000}{(arXiv:1801.09000)}.

\bibitem{BrSj}
{\sc Bruno, T., and Sj\"ogren, P.}
\newblock On the {R}iesz transforms for the inverse {G}auss measure.
\newblock Preprint 2019
  \href{http://arxiv.org/abs/1906.03827}{(arXiv:1906.03827)}.

\bibitem{Bu}
{\sc Burkholder, D.~L.}
\newblock A geometric condition that implies the existence of certain singular
  integrals of {B}anach-space-valued functions.
\newblock In {\em Conference on harmonic analysis in honor of {A}ntoni
  {Z}ygmund, {V}ol. {I}, {II} ({C}hicago, {I}ll., 1981)}, Wadsworth Math. Ser.
  Wadsworth, Belmont, CA, 1983, pp.~270--286.

\bibitem{Chen}
{\sc Chen, B.-Y.}
\newblock {\em Differential geometry of warped product manifolds and
  submanifolds}.
\newblock World Scientific Publishing Co. Pte. Ltd., Hackensack, NJ, 2017.

\bibitem{FS}
{\sc Forzani, L., and Scotto, R.}
\newblock The higher order {R}iesz transform for {G}aussian measure need not be
  of weak type {$(1,1)$}.
\newblock {\em Studia Math. 131}, 3 (1998), 205--214.

\bibitem{GCMST1}
{\sc Garc\'{\i}a-Cuerva, J., Mauceri, G., Sj\"{o}gren, P., and Torrea, J.~L.}
\newblock Higher-order {R}iesz operators for the {O}rnstein-{U}hlenbeck
  semigroup.
\newblock {\em Potential Anal. 10}, 4 (1999), 379--407.

\bibitem{GCMST2}
{\sc Garc\'{\i}a-Cuerva, J., Mauceri, G., Sj\"{o}gren, P., and Torrea, J.~L.}
\newblock Spectral multipliers for the {O}rnstein-{U}hlenbeck semigroup.
\newblock {\em J. Anal. Math. 78\/} (1999), 281--305.

\bibitem{G}
{\sc Guerre-Delabri\`ere, S.}
\newblock Some remarks on complex powers of {$(-\Delta)$} and {UMD} spaces.
\newblock {\em Illinois J. Math. 35}, 3 (1991), 401--407.

\bibitem{HTV}
{\sc Harboure, E., Torrea, J.~L., and Viviani, B.}
\newblock Vector-valued extensions of operators related to the
  {O}rnstein-{U}hlenbeck semigroup.
\newblock {\em J. Anal. Math. 91\/} (2003), 1--29.

\bibitem{MPS}
{\sc Men\'{a}rguez, T., P\'{e}rez, S., and Soria, F.}
\newblock The {M}ehler maximal function: a geometric proof of the weak type 1.
\newblock {\em J. London Math. Soc. (2) 61}, 3 (2000), 846--856.

\bibitem{Mu1}
{\sc Muckenhoupt, B.}
\newblock Hermite conjugate expansions.
\newblock {\em Trans. Amer. Math. Soc. 139\/} (1969), 243--260.

\bibitem{Mu2}
{\sc Muckenhoupt, B.}
\newblock Poisson integrals for {H}ermite and {L}aguerre expansions.
\newblock {\em Trans. Amer. Math. Soc. 139\/} (1969), 231--242.

\bibitem{Pe}
{\sc P\'erez, S.}
\newblock The local part and the strong type for operators related to the
  {G}auss measure.
\newblock {\em J. Geom. Anal. 11}, 3 (2001), 497--507.

\bibitem{PeSo}
{\sc P\'erez, S., and Soria, F.}
\newblock Operators associated with the {O}rnstein-{U}hlenbeck semigroup.
\newblock {\em J. London Math. Soc. 61}, 3 (2000), 857--871.

\bibitem{Sa}
{\sc Salogni, F.}
\newblock {\em Harmonic Bergman spaces, Hardy-type spaces and harmonic analysis
  of a symmetric diffusion semigroup on $\mathbb{R}^n$}.
\newblock PhD thesis, Universit\`a degli Studi di Milano-Bicocca, 2013.

\bibitem{San}
{\sc Sansone, G.}
\newblock {\em Orthogonal functions}.
\newblock Dover Publications, Inc., New York, 1991.

\bibitem{StLP}
{\sc Stein, E.~M.}
\newblock {\em Topics in harmonic analysis related to the {L}ittlewood-{P}aley
  theory}.
\newblock Annals of Mathematics Studies, No. 63. Princeton University Press,
  Princeton, N.J.; University of Tokyo Press, Tokyo, 1970.

\bibitem{TZ}
{\sc Torrea, J.~L., and Zhang, C.}
\newblock Fractional vector-valued {L}ittlewood-{P}aley-{S}tein theory for
  semigroups.
\newblock {\em Proc. Roy. Soc. Edinburgh Sect. A 144}, 3 (2014), 637--667.

\end{thebibliography}

\end{document}